\documentclass[reqno]{amsart}
\usepackage[utf8]{inputenc}
\usepackage{enumitem}
\usepackage{mathdots}
\usepackage{amsmath,amssymb,amsbsy,amsfonts,amsthm,latexsym,amsopn,amstext,mathtools,leftidx,
            amsxtra,euscript,amscd,verbatim,epsf,colordvi,graphics}
\usepackage{ytableau}
\usepackage{color}
\usepackage{pgf,tikz}
\usetikzlibrary{decorations.pathreplacing}
\usetikzlibrary{arrows}
\usepackage{mathrsfs}
\usepackage{graphicx}

\usepackage{hyperref}
\usepackage{cleveref}
\usepackage[margin=11pt,font=footnotesize,labelfont=bf]{caption}
\newtheoremstyle{tplain}{3pt}{3pt}{\rmfamily}{}{\bfseries}{.}{0.5em}{}
\theoremstyle{tplain}

\usepackage{pgf,tikz}
\usepackage{tcolorbox}
\usepackage[enableskew]{youngtab}
\definecolor{darkgreen}{cmyk}{1,0,1,0}
\newtheorem{thm}{Theorem}
\newtheorem{lem}{Lemma}
\newtheorem{ex}{Example}
\newtheorem{cor}{Corollary}
\newtheorem{prop}{Proposition}
\newtheorem{obs}{Remark}
\newtheorem{defi}{Definition}
\def \cal{\mathcal}
\def \rm {\mathrm}
\def \mbf {\mathbf}

\def\rev{{\rm rev}}
\def\Z { \mathbb{Z}}
\def\C { \mathbb{C}}
\def\Z { \mathbb{Z}}
\def\C { \mathbb{C}}
\def\N { \mathbb{N}}
\def\LRAII {\mathsf{LR}}
\def \r {\mathbf{r}}
\def \t {\mathbf{t}}
\def \redu {\mathrm{red}}
\def\c { \mathrm{c}}
\def \k {\mathfrak{k}}
\def\wt { \mathrm{wt}}
\def\shape { \mathrm{shape}}
\def\rem { \mathrm{rem}}
\def\suc { \mathrm{suc}}
\def\a { \mathbf{a}}
\def\b { \mathbf{b}}
\def\s { \mathbf{s}}
\def\B { \mathbf{B}}

\newcommand{\bro}{\color{brown}}

\newcommand{\ora}{\color{orange}}

\newcommand{\oo}{\color{blue}}
\newcommand{\blue}{\color{blue}}
\newcommand{\red}{\color{red}}
\newcommand{\green}{\color{green}}

\makeatletter
\newcommand*\bigcdot{\mathpalette\bigcdot@{.5}}
\newcommand*\bigcdot@[2]{\mathbin{\vcenter{\hbox{\scalebox{#2}{$\m@th#1\bullet$}}}}}
\makeatother

\newcommand*\circled[1]{\tikz[baseline=(char.base)]{
            \node[shape=circle,draw,inner sep=1pt] (char) {#1};}}

\newcommand{\YT}[3]{
\vcenter{\hbox{
\begin{tikzpicture}[x={(0in,-#1)},y={(#1,0in)}] % matrix coordinate
\foreach \rowi [count=\i] in {#3} {
 \foreach \e [count=\j] in \rowi {
  \draw (\i,\j) rectangle +(-1,-1);
  \draw (\i-0.5,\j-0.5) node {$#2\e$};
 }
}
\end{tikzpicture}
}}
}

\title[The recording tableaux in the quantum Littlewood-Richardson map] {The recording tableaux in  the quantum \\ Littlewood-Richardson map, \\ the orthogonal transpose symmetry map, and\\   the computation of $\mathfrak{k}$-highest weight tableaux }

\author{Olga Azenhas}
\address{ University of Coimbra, CMUC, Department of Mathematics, Portugal}
\email{oazenhas@mat.uc.pt}



\begin{document}

\begin{abstract} Recently Watanabe has given an algorithm to compute a  bijection, that he calls (quantum) Littlewood-Richardson (LR) map (or quantum LR rule of type AII), between semi-standard Young tableaux of  shape a partition with at most  $2n$ parts and pairs of tableaux consisting of a symplectic tableau with shape a partition with at most $n$ parts, and a recording tableau  of skew-shape given by the two previous shapes.  The recording tableaux   in that algorithm  are shown  to be equinumerous to Littlewood-Richardson-Sundaram tableaux  whose injectivity is shown combinatorially while the surjectivity is concluded via  representation theory of a quantum symmetric pair of  type AII$_{2n-1}$. Henceforth,  the algorithm to compute the quantum LR map provides a new branching model for the branching multiplicities from $GL_{2n}(\C)$ to $Sp_{2n}(\C)$.  Here, as morally suggested  by Watanabe, one provides a combinatorial proof for the surjectivity of the quantum LR map which in turn  exhibits the restriction of the LR orthogonal transpose symmetry map to LR-Sundaram tableaux. The surjectivity is exhibited via the iteration of the reverse Schensted column insertion,  initialized on a symplectic tableau, ruled by the slack data of a quantum recording tableau, and followed with the inverse reduction map on the bumped symplectic column in each iteration.   The explicit inverse  reduction map has been recently computed by the author in arXiv:2606.24840v2 for any symplectic column.
As an application of the explicit surjectivity and henceforth of the inverse of the quantum LR map, we compute and characterize  a family of $\k$- highest  weight semi-standard tableaux
in the recent proof of the Naito-Sagaki conjecture using the Watanabe's
branching rule based on the crystal basis theory for $\imath$quantum
groups of type AII$_{2n-1}$. More precisely, that family of tableaux is generated either  by quantum recording tableaux with one vertical strip or with $1$-$0$-slack sequence and  shape length $2n$ or $2n-1$. For a given $\k$-highest weight symplectic tableau, the ones produced by  recording tableaux with $1$-$0$-slack sequence and  shape length $2n$ or $2n-1$, satisfy certain linear inequalities between   multiplicities of  columns in that $\k$-highest weight  symplectic tableau and  multiplicities of the columns obtained by the inverse reduction map on the successive bumped numbers.
\end{abstract}
\maketitle

\tableofcontents

\section{Introduction}
The quantum Littlewood-Richardson map by Watanabe  \cite{watanabe} for the pair $(GL_{2n}(\mathbb{C}),Sp_{2n}(\mathbb{C}))$ can be seen  as a generalization of the branching rule, known as the Littlewood-Richardson rule, for the pair $(GL_m(\mathbb{C})\times GL_m(\mathbb{C}), GL_m(\mathbb{C}))$. More precisely,  a generalization of the G.P. Thomas \cite{tho78} bijection providing the Littlewood-Richardson rule \cite{LR}
  \begin{align}\label{tho0}SST_m(\mu)\times SST_m(\nu) \overset{\sim}\rightarrow\bigsqcup_{\begin{smallmatrix}\lambda\in Par_{\le m}\\
T\in LR(\lambda/\mu,\nu)\end{smallmatrix}}SST_{m}(\lambda) \times \{T\},\end{align}
 to a bijection providing the quantum Littlewood-Richardson rule \cite{watanabe}
 \begin{align}\label{wat0}\displaystyle SST_{2n}(\lambda) \overset{\sim}\rightarrow\bigsqcup_{\begin{smallmatrix}\mu\in Par_{\le n}\\
\mu\subseteq\lambda\\
Q\in Rec_{2n}(\lambda/\mu)\end{smallmatrix}}SpT_{2n}(\mu) \times \{Q\},
\end{align}
where  $SpT_{2n}(\mu)$ denotes the set of symplectic semi-standard tableaux of shape $\mu$, with entries in $[1,2n]$, and the sets of Littlewood-Richardson-Sundaram tableaux $LRS_{2n}(\lambda/\mu) \subseteq LR_{2n}(\lambda/\mu)$ \cite{sundaram} and of recording tableaux $Rec_{2n}(\lambda/\mu)$ in \eqref{wat0} satisfy $LRS_{2n}(\lambda/\mu)\overset{\sim}\longrightarrow Rec_{2n}(\lambda/\mu)$ via a natural bijection.

\subsection{The crystal version of the Littlewood-Richardson rule} For the sake of comparison, the G.P. Thomas bijection \eqref{tho0} asserts
\begin{align*}&SST_m(\mu)\times SST_m(\nu)\overset{\sim}\longrightarrow \bigsqcup_{\begin{smallmatrix}\lambda\in Par_{\le m}\\{}
\end{smallmatrix}} SST_m(\lambda)\times LR(\lambda/\mu,\nu).
%&\qquad\qquad  (U, V)\mapsto (P(U\otimes V),Q(U\otimes V)),\;
\end{align*}
where the recording tableau of the Schensted column insertion of a tableau pair in $SST_m(\mu)\times SST_m(\nu)$ is stored in the form of an LR tableau in  $LR(\lambda/\mu,\nu)$ for some partition $\lambda$. Very importantly is that
this bijection lifts to a $\mathfrak{gl}_m$-crystal isomorphism \cite{kwon09} as in \eqref{crystalthomas} by showing how the tensor product of two $Gl_m$-irreducible representations  decomposes into irreducible $GL_m$-representations,
while giving simultaneously the crystal version of the original Littlewood-Richardson (LR) rule \cite{littlewood},   the Berenstein-Gelfand-Zelevinsky LR rules \cite{GZ85,gzpolyed,BZphy} on left and right Gelfand-Tsetlin patterns \cite{gt50}, and notably the integer Knutson-Tao hives \cite[Appendix]{kt1}, \cite{fultonbuch} as the interlocking of those Gelfand-Tsetlin patterns \cite{KTT1,KTT2},
\begin{align}\label{crystalthomas}
B(\mu,m)\otimes B(\nu,m)\simeq \bigoplus_{\begin{smallmatrix}T\in LR(\lambda/\mu,\nu)\\
\lambda\in Par_{\le m}\end{smallmatrix}} B(\lambda,m)\times \{T\}\simeq \bigoplus_{\begin{smallmatrix}
\lambda\in Par_{\le m}\end{smallmatrix}} B(\lambda,m)^{c_{\mu,\nu}^\lambda}.
\end{align}
Notably, the highest  and lowest weights of the connected component $B(\lambda,m)\times \{T\}$, determined by the recording LR tableau $T$, exhibit the right respectively left companions of the LR tableau $T$, and their interlocking the integer hive of boundary $(\lambda,\mu,\nu)$. The multiplicity $c_{\mu,\nu}^\lambda$ of the crystal $B(\lambda,m)$ in this decomposition shows that $LR(\lambda/\mu,\nu)$  is equinumerous to the  right respectively left companions in the form of Gelfand-Tsetlin patterns in the Berenstein-Gelfand-Zelevinsky LR rules as well as the number of  integer hives with boundary $(\mu,\nu,\lambda)$ (see \cite{akt16} for details).

It remains to say that beyond the previous information packed in the crystal isomorphism \eqref{crystalthomas}, $c_{\mu,\nu}^\lambda$ is also the number of integral points of a hive polytope with integral boundary $(\mu,\nu,\lambda)$  (that is, the number of integral hives describing different coupling of three irreducible representations), a Knutson-Tao new description of the Berenstein-Zelevinski polytopes (where integral points are Berenstein-Zelevinsky patterns) \cite{bz}. It is then immediate that LR coefficients   can be modelled via the Ehrhart (quasi)polynomial \cite{ehrhart} of a rational polytope whose polynomiality has been proved by Derksen-Weyman and Rassart in \cite{lrpolynomial,rassartlrpolynomial}. (For advances in the Ehrhart theory model for $c_{\mu,\nu}^\lambda$, the reader is referred, for instance, to \cite{alexandersson} and \cite{lrvolume} and references therein.)

From \cite{kwon18,leclen,sathishtorres,az26} and the Henriques-Kamnitzer $\mathfrak{gl}_m$-crystal commuter  \cite{HenKam,HK2}, it turns out that the decomposition \eqref{crystalthomas} has a refinement to include LR-Sundaram (LRS)-tableaux \cite {sundaram,sundaram90} and their   companion pairs whose interlocking exhibits a flagged hive as in \cite{sathishtorres}. For a fixed $n\in \mathbb{N}$, let $m=2n$, $\mu\in Par_{\le n}$ a partition  with at most $n$ parts,  and $\nu, \lambda \in Par_{\le 2n}$ partitions with at most $2n$ parts such that   $\nu$ has even length columns. Let  $LRS_{2n}(\lambda/\mu,\nu)\subseteq LR(\lambda/\mu,\nu)$ be the set of LR-Sundaram tableaux  (or LR symplectic tableaux $LRT^{Sp}_{2n}(\lambda/\mu)$ in \cite{watanabe}) of shape $\lambda/\mu$ and weight $\nu$, and let $ SpT_{2n}(\mu)\subseteq SST_{2n}(\mu)$ be the set of symplectic semi-standard tableaux of shape $\mu$ \cite{watanabe}, in natural bijection with the symplectic Geland-Testlin patterns \cite{az26}. Then \eqref{crystalthomas} refines as

%\iffalse
\begin{align*}%\label{crystalthomasref}
&B(\mu,2n)\otimes B(\nu,2n)\simeq \\
&\bigoplus_{\begin{smallmatrix}T\in LRS(\lambda/\mu,\nu)\\
G^{sp}_\mu \in SpT_{2n}(\mu)\\
\mbox{\tiny left companion of $T$}\\
\lambda\in Par_{\le 2n}\end{smallmatrix}} B( G^{sp}_\mu\otimes Y(rev\nu) )\times \{T\} \bigsqcup\bigoplus_{\begin{smallmatrix}T\in LR(\lambda/\mu,\nu)\setminus LRS(\lambda/\mu,\nu)\\
G_\mu\in  SST_{2n}\\ \mbox{ \tiny left companion of $T$}\\
\lambda\in Par_{\le 2n}\end{smallmatrix}} B( G_\mu\otimes Y(\rm rev\,\nu) )\times \{T\}\\
&\simeq \bigoplus_{\begin{smallmatrix}
G^{Sp}_\mu\otimes Y(rev\nu)\simeq Y(rev\lambda)\\
\lambda\in Par_{\le 2n}\end{smallmatrix}} B(\lambda,2n)^{spc_{\mu,\nu}^\lambda}\bigsqcup\bigoplus_{\begin{smallmatrix}
G_\mu\otimes Y(rev\nu)\simeq Y(rev\,\lambda)\\
G_\mu\in SST_{2n}(\lambda)\setminus SpT_{2n}(\mu)\\
\lambda\in Par_{\le 2n}\end{smallmatrix}} B(\lambda,2n)^{c_{\mu,\nu}^\lambda}
\end{align*}
%\fi
where $spc_{\mu,\nu}^\lambda$ is the cardinality of $LRS_{2n}(\lambda/\mu,\nu)$ or the number  of symplectic semi-standard tableaux $G^{Sp}_\mu\in SpT_{2n}(\mu)$ with  weight $rev(\lambda-\nu)$ such that  $G^{sp}_\mu\otimes Y(rev\nu)\simeq Y(rev\lambda)$.

\subsection{The quantum Littlewood-Richardson rule}
Let $\lambda \in Par_{\le 2n}$ be a partition with at most $2n$ parts. The Littlewood-Richardson map $LR^{AII}$ \cite[Theorem 3.1.4]{watanabe}
is an one-to-one
assignment of a semi-standard tableau $T$ of shape $\lambda\in Par_{\le 2n}$ to a pair $(P^{II}(T), Q^{II}(T)) $ for some shape $ \mu\in Par_{\le n}$,   consisting of a symplectic tableau in $ SpT_{2n}(\mu)$  respectively  a recording tableau of skew shape $\lambda/\mu$ in $Rec_{2n}(\lambda/\mu)$,
\begin{align} \label{lrII}
LR^{AII}: SST_{2n}(\lambda) \overset{\sim}\longrightarrow
\bigsqcup_{\begin{smallmatrix}\mu\in Par_{\le n}\\
\mu\subseteq\lambda\end{smallmatrix}}SpT_{2n}(\mu) \times Rec_{2n}(\lambda/\mu).
\end{align}
Furthermore, for each $\mu \in Par_{\le n}$ with $\mu\subset \lambda$, the set of recording tableaux $Rec_{2n}(\lambda/\mu)$ in the Littlewood-Richardson map $LR^{AII}$ are in natural bijection with the set of LR-Sundaram tableaux  $LRS_{2n}(\lambda/\mu)$,
\begin{align}\label{record}Rec_{2n}(\lambda/\mu) \overset{\sim}\longrightarrow LRS_{2n}(\lambda/\mu).
\end{align}
%where $LRS_{2n}(\lambda/\mu)$ is the set of LR-Sundaram tableaux of shape $\lambda/\mu$ (or LR symplectic tableaux $LRT^{Sp}_{2n}(\lambda/\mu)$ in \cite{watanabe}).
Henceforth, the algorithm defining the map $LR^{AII}$ gives rise to
a bijection between the set of semi-standard tableaux of  given shape, say $\lambda$, and the disjoint union
of several copies of the sets of symplectic tableaux for each   shape $\mu\subseteq \lambda$ with at most $n$ parts, where the multiplicity $c_\mu^\lambda$ of  each shape  $\mu$ equals to the cardinality of $LRS(\lambda/\mu)$, the number of LR-Sundaram tableaux of shape $\lambda/\mu$,
\begin{align} \label{lrsII}
 SST_{2n}(\lambda) \overset{\sim}\longrightarrow
\bigsqcup_{\begin{smallmatrix}\mu\in Par_{\le n}\\
\mu\subseteq\lambda\\
Q\in Rec_{2n}(\lambda/\mu)\end{smallmatrix}}SpT_{2n}(\mu) \times \{Q\}.
\end{align}
Thereby, the  algorithm defining the map $LR^{AII}$ gives a new branching rule for the pair  $(Gl_{2n}(\mathbb{C}), Sp_{2n}(\mathbb{C})$. It is then demanding to have a complete combinatorial proof of this new branching rule. As suggested in  \cite{watanabe} the bijections \eqref{lrII} and \eqref{record} might  be proved in the realm of combinatorics.

 The injectivity of the Littlewood-Richardson map \eqref{lrII} and   the injectivity of the  bijection \eqref{record} are  proved combinatorially in \cite[Proposition 7.1.3]{watanabe} and respectively in  \cite[Lemma8.3.1, Lemma 8.3.2]{watanabe} by showing,
 $$Rec_{2n}(\lambda/\mu)\subseteq \widetilde Rec_{2n}(\lambda/\mu) \overset{\sim}\hookrightarrow  LRS_{2n}(\lambda/\mu).$$

  The surjectivity  of \eqref{lrII} and \eqref{record} is concluded   from the representation theory of a quantum symmetric pair of type $\rm{AII}_{2n-1}$ in \cite[Section 7.2]{watanabe} respectively \cite[Theorem 8.2]{watanabe}. More precisely, the quantum
symmetric pair of type $\rm{AII}_{2n-1}$ consists of the Drinfeld-Jimbo quantum group $\mathbf{U}$ and the Letzter $\imath$quantum group $\mathbf{U}\imath$ \cite{letzter} (a coideal subalgebra of $\mathbf{U}$) of the universal
enveloping algebras of $\mathfrak{gl}_{2n}(\mathbb{C})$ respectively $\mathfrak{sp}_{2n}(\mathbb{C})$.  For details on symmetric quantum pairs, we refer the reader to the survey paper \cite{wang}.
  %As suggested in  \cite{watanabe} the bijections can be proved in the realm of combinatorics.
\iffalse
From \cite{BZphy}  and \cite{az26}, we still have the refinement of \eqref{record}
\begin{align*}
Rec_{2n}(\lambda/\mu)&\overset{\sim}\longrightarrow  \bigsqcup_{\begin{smallmatrix}\nu\in Par_{\le 2n} \\
\mbox{ $\nu$ even}
\end{smallmatrix}}{}^-LRS^\lambda_{\mu,\nu}&\overset{\sim}\longrightarrow \bigsqcup_{\begin{smallmatrix}\nu\in Par_{\le 2n} \\
\mbox{ $\nu$ even}
\end{smallmatrix}}{}^-LR^\lambda_{\mu,\nu}\cap Sp\cal{GT}^{rev(\lambda-\nu)}_{\mu,2n}\\
  &T&\mapsto  T^{\mathfrak{tr}}&\mapsto G_\mu&\mapsto GT~G_\mu\\
\end{align*}
\fi
The Littlewood-Richardson map \eqref{lrII} has a $q$ analogue isomorphism by
 showing how the irreducible $\textbf{U}$-module $V (\lambda)$ decomposes into irreducible $\textbf{U}^\imath$-submodules $V^\imath(\mu)$ \cite{molevq} at
$q=\infty$ which allows to conclude the surjectivity  of \eqref{lrII}
\begin{align*}
LR^{AII}:V(\lambda)\overset{\sim}\rightarrow \bigoplus_{\mu\in Par\le n} V^\imath(\mu)\otimes\mathbb{Q}Rec_{2n}(\lambda/\mu),
\end{align*}
where $\mathbb{Q}Rec_{2n}(\lambda/\mu)$ denotes the $\mathbb{Q}$-vector space with basis $Rec_{2n}(\lambda/\mu)$.

Similarly to the G.P. Thomas bijection \eqref{tho0} whose recording tableaux are Littlewood-Richardson that determine the highest and lowest weight  tableaux of a connected component in the tensor product decomposition, the recording tableaux in the quantum Littlewood-Richardson map \eqref{wat0} determine the
$\mathfrak{k}$-highest and lowest weight semi-standard tableaux in $SST_{2n}(\lambda)$ as defined in the
 recent proof of the Naito–Sagaki conjecture \cite{nsw} via the Watanabe  branching rule for $\imath$quantum groups. Here \cite{nsw} $\mathfrak{k}$ means a certain subalgebra of $\mathfrak{sl}_{2n}(\mathbb{C})$ that is isomorphic to $\mathfrak{sp}_{2n}(\mathbb{C})$ and $\textbf{U}^\imath$ can be seen as a $q$-deformation of $U(\mathfrak{k})$ with a natural embedding in $\textbf{U}$.
It is then useful to have an explicit surjectivity which allows to compute explicitly those highest or lowest $\k$-weight  semi-standard tableaux.
We note that the first complete proof of the Naito-Sagaki conjecture and that does not  not rely on $\imath$quantum groups was provided by Schumann–Torres in \cite{schumanntorres}. Recently another combinatorial proof of the Naito–Sagaki conjecture, which does not rely on either  or this paper, appeared in \cite{muniz}.
%where
%e apply our results to the explicit computation of $\mathfrak{k}$-highest and lowest weight tableaux in the
 %proof of the Naito–Sagaki conjecture via the branching rule for $\imath$quantum groups \cite{nsw}.
%In these works, the question
%is considered by using a different deformation of $sp_{2n}(C)$, built as a coideal subalgebra
%of $U_q(gl_2n)$ to have a natural embedding

%"See introduction of the paper under review"
 Our main results assert as follows.

\noindent \textbf{Theorem} \textbf{A} (Theorem \ref{surjection})[The surjectivity of the map in \cite[Lemma 8.3.2]{watanabe}  Let $T\in LRS_{2n}(\lambda/\mu)$ of even weight $\nu$ and $\nu^t$ its conjugate partition. Let $J_1,\dots,J_{\nu_1}$ be the decomposition of $T$ into vertical strips where each strip is filled with $1,2,\dots, |J_i|=\nu^t_i\in 2\mathbb{Z}$ for $i=1,\dots,\nu_1$.
%Then $\ell(J_i)\in 2\mathbb{Z}$.
Let $\lozenge$ be the map that relabels each vertical strip $J_i$ % =\mu^{(i-1)}/\mu^{(i)}$
with $\nu^t_i$ $i$'s, for $i=1,\dots,\nu_1$. Then we get a new tableau $Q\in \widetilde Rec_{2n}(\lambda/\mu)$ of weight $\nu^t$ and $\lozenge$ is a bijection between $LRS_{2n}(\lambda/\mu)$ and $\widetilde Rec_{2n}(\lambda/\mu)$. Equivalently, $\lozenge$ is the inverse  of the map in \cite[Lemma 8.3.2]{watanabe}.

\medskip
\noindent \textbf{Corollary} (Section \ref{sec:lozenge}) The
  $LR$ orthogonal transpose symmetry map $\blacklozenge$ \cite{azkoma25} restricts to LR-Sundaram tableaux as
\begin{align*}\blacklozenge: LRS_{2n}(\mu,\nu,\lambda)\overset{\lozenge}\longrightarrow \widetilde Rec_{2n}(\mu,\nu^t,\lambda)\underset{\pi\circ t}\hookrightarrow LR({\lambda^t},\nu^t,\mu^t): T\mapsto %\blacklozenge(T)=\picirc t)\lozenge (T)
 \lozenge T=Q\mapsto Q^{\pi\circ t}=\blacklozenge T
\end{align*}
where $Q^{\pi\circ t}$ means the $\pi$-rotation (transposition) followed with transposition (rotation) of the Young diagram $D(\lambda)$.
In other words, the bijection $\lozenge$, and therefore the injection in \cite[Lemma 8.3.2]{watanabe}, exhibits the restriction of the LR orthogonal transpose symmetry map $\blacklozenge$, $c_{\mu,\nu,\lambda}=c_{\lambda^t, \nu^t,\mu^t}$, to LR Sundaram tableaux.

\medskip

 Let $l\in [0,2n]$. In \cite[Proposition 4.3.6]{watanabe} it is shown that the reduction map on $SST_{2n}(\varpi_l)$ is the injective assignment \eqref{redumap}
\begin{align*}\redu:SST_{2n}(\varpi_l)\rightarrow\bigsqcup_{\begin{smallmatrix} 0\le t\le min\{l,2n-l\}\\
l-t\in 2\mathbb{Z}
\end{smallmatrix}}SpT_{2n}(\varpi_t), \;\;
\mathbf{a}\mapsto \redu( \mathbf{a})=\mathbf{a}\setminus \rem(\mathbf{a}).
%\label{redumap}
\end{align*}

\noindent \textbf{Theorem} \textbf{B} %(Theorem \ref{reductionsurj})
[The reduction map, $\redu$, on $SST_{2n}(\varpi_l)$ \cite[Proposition 4.3.6]{watanabe}  is surjective]
Let $l\in [0,2n]$ and $t\in[0,n]$ such that  $0\le t\le min\{l,2n-l\}$ and $l-t\in 2\Z$. Let $\mathbf{a} = (a_1, \dots , a_t)$ be a column  in $SpT_{2n}(\varpi_t)$.  Lemma \ref{lem:reductionsurj} and Theorem \ref {thm:nofactors} exhibit a procedure to explicitly compute  the inverse of the reduction map in the cases where the symplectic column %decomposes into non empty factors $\bf A_i$ of even length  consisting of
consists of consecutive integers  starting with an odd number, respectively when the symplectic column is such that  consecutive integers occur only as an even number followed with an odd number which are within the patterns of the symplectic $\k$-highest (lowest) weight tableaux considered in the last section. An explicit full formula for the inverse of the reduction map has recently been provided by the author in \cite{azreduction}

\medskip
The next theorem is a consequence of the of explicit formula for the inverse reduction map \cite{azreduction} and the reverse Schensted column insertion ruled by the slack data of a recording tableau \cite{azslack} (see subsections \ref{subsec:1verticalstrip} and \ref{sec:slackdata}). It gives the constructive fundamental step for the unwinding of the sequence of successors
\cite[Lemma 8.3.1]{watanabe}  in the algorithm to computing the quantum Littlewood-Richardson map. The unwind procedure goes from the last to the first successor, and  Theorem \textbf{C}   (Theorem \ref{thm:verygeneralstrip}) below gives explicitly ${\LRAII^{AII}}^{-1}(S,Q)$ for the unwind of the last successor  where $Q\in \widetilde Rec_{2n}(\lambda/\mu)$ with $\mu\subset_{vert}$ and  $S\in SpT_{2n}(\mu)$. Theorem \textbf{CC} (Theorem \ref{thm:generalhole}) gives explicitly ${\LRAII^{AII}}^{-1}(S,Q)$ for any $Q\in \widetilde Rec_{2n}(\lambda/\mu)$.
\medskip

\noindent \textbf{Theorem} \textbf{C} (Theorem \ref{thm:verygeneralstrip})[The inclusion $Rec_{2n}(\lambda/\mu)\supseteq \widetilde Rec_{2n}(\lambda/\mu)$  for \cite[Lemma 8.3.1]{watanabe}]
Let $S\in SpT_{2n}(\mu)$ and $Q\in  \widetilde Rec_{2n}(\lambda/\mu)$, $\mu\subset_{vert} \lambda$,  $l=\ell(\lambda)$, with slack row index vector  $\mathbf{r}=\{r_1<\dots< r_{t_0}\}$. Then, $\ell(\mu)\le n, l$, $0\le l- t_0\in 2\Z$, $l\le 2n-t_0$, and it holds:

\begin{align}\label{inverseint}{\LRAII^{AII}}^{-1}(S,Q)&= \c \circ (\rm{red}_{t_0}^{-1},\rm{id})\circ (\;,\underset{{\r}}\leftarrow S)\nonumber\\
&=\c\circ (\rm{red}_{t_0}^{-1},\rm{id})((a_1,\dots,a_{t_0}),S^1 )\nonumber\\
&=\c\circ (\rm{red}_{t_0}^{-1}(a_1,\dots,a_{t_0}), S^1) \nonumber\\
&=\redu_{t_0}^{-1}((a_1,\dots,a_{t_0}))\bigcdot S^1=:S^\r\in SST_{2n}(\lambda).
\end{align}
where $(\;,\underset{{\r}}\leftarrow S)=((a_1,\dots,a_{t_0}),(\underset{{\r}}\leftarrow S)=S^1 )$ is the reverse Schensted column insertion applied to the rows $\{r_1<\dots< r_{t_0}\}$ of $S$ with bumped symplectic column
$(a_1,\dots,a_{t_0})\in SpT_{2n}(\varpi_{t_0})$,
$\redu_{t_0}^{-1}((a_1,\dots,a_{t_0}))=T_0(a_1)T_1\cdots (a_{t_0})T_{t_0}\in SST_{2n}(\varpi_{\ell(\lambda)})$ is the inverse reduction map as in Lemma \ref{lem:reductionsurj}, or Theorem \ref{lem:reductionsurj2}, or Theorem \ref {thm:nofactors} or more generally in \cite{azreduction},
and $S^1\in SpT_{2n}({\mu^{(1)}}')$  with ${\mu^{(1)}}'=\mu^{(1)}-\delta_\r$,   and  $\lambda=\mu^{(1)}-\delta_\r+\varpi_{\ell(\lambda)}$, $\mu^{(1)}:=\mu$.

\medskip
\noindent \textbf{Theorem} \textbf{CC} (Theorem \ref{thm:generalhole}) Let $ S\in SpT_{2n}(\mu)$ and $Q\in  \widetilde Rec_{2n}(\lambda/\mu)$,  with slack sequence $\underline \t=(t^{(N)}_0\ge \dots \ge t^{(1)}_0)$ and slack (row index) vector sequence $\underline \r=[\r^{(N)}\le_\r\dots\le_\r\r^{(1)}]$. Then

\begin{align}\label{inversegeneral}{\LRAII^{AII}}^{-1}(S,Q)=S^{\underline \r}=S^{\r^{(N)}\cdots\r^{(2)} \r^{(1)}}
\end{align}
where for $i=1,\dots,N$,
\begin{align*}
&S^{\r^{(N)}\cdots\r^{(i)}}:=\big(S^{\r^{(N)}\cdots\r^{(i+1)}}\big)^{\r_i}=\c \circ (\rm{red}_{t^{(i)}_0}^{-1},\rm{id})\circ \big(\;,\underset{{\r^{(i)}}}\leftarrow S^{\r^{(N)}\cdots\r^{(i+1)}}\big)
\end{align*}
with  $S^{\r_{N+1}}:=S$.
%$l=\ell(\lambda)$, has slack row index vector  $\mathbf{r}=\{r_1<\dots< r_{t_0}\}$. Then, $\ell(\mu)\le n, l$, $0\le l- %t_0\in 2\Z$, $l\le 2n-t_0$, and it holds:

\medskip
As an application of the explicit surjectivity  and consequently the explicit inverse $\LRAII^{AII}$ of the quantum Littlewood-Richardson map, we characterize in Corollary \ref{lem:H-L}, Theorem \ref{thm:n=oddslack1}, Theorem \ref{thm:n=evenslack1} and their corollaries for $n=1,3$ respectively $n=2,4$, the family of $\k$-highest wight tableaux generated either  by quantum record tableaux with one vertical strip or with $1$-$0$-slack sequences.
Consider the  sequence $\{u_i\}_{i=1}^n$ of numbers depending on the parity of $n$
\begin{align}\{u_i\}_{i=1}^n=\begin{cases}\{2,3,6,7,10,\dots, 2(n-1)-1,2n\},& n\notin 2\Z,\\
\{2,3,6,7,\dots, 2(n-1),2n-1\},& n\in 2\Z\end{cases}\label{numbers:uuint}
\end{align}
A symplectic tableau in $SpT_{2n}(\mu)$ is said to be $\k$-highest weight of shape $\mu$ if  row $i$ has only numbers $u_i$ for $i=1,\dots,n$.
For $n$ odd, one has the following assertion.

\noindent \textbf{Theorem} \textbf{D} (Theorem \ref{thm:n=oddslack1})[$\k$-highest weight tableaux produced by $1$-$0$-slack sequence quantum recording tableaux for $n$ odd] Let $S^{H,\mu}$  be the $\mathfrak{k}$-highest weight tableau
in $SpT_{2n}(\mu)$, and   $Q\in  Rec_{2n}(\lambda/\mu)$, $\ell(\lambda)\in\{2n-1,2n\}$, with  $1$-$0$-slack sequence of the form $\underline \t=(1,\dots,1,0^M)$ and slack vector sequence $\underline\r$ with slack incidence matrix $\delta_{\underline\r}$. Then, for some $0\le k\le n$, ${\LRAII^{AII}}^{-1}(S^{H,\mu},Q)$ returns   the $\mathfrak{k}$-highest  weight tableau in
$SST_{2n}(\lambda)$, with ${\mathfrak{k}}$-weight  $\mu$, of the form
\begin{align}
&Y(M^{2n})\bigcdot\nonumber\\
& \bigcdot\displaystyle\circledcirc_{j=n}^{k+1}\big(\redu_1^{-1}({u_j})\big)^{m_{\redu^{-1}(u_j)}}\bigcdot
\big(\redu_1^{-1}({u_{k}})\big)^{m_{\redu^{-1}(u_k)}}\bigcdot
{\displaystyle\circledcirc_{j=k}^{1}\big(\redu_1^{-1}({u_j})\setminus\{u_n\}\big)^{m_{\redu^{-1}(u_j)\setminus \{2n\}}}}\bigcdot S^{H,\mu'},
\end{align}
where $S^{H,\mu'}$ is the $\k$-highest weight tableau in $SpT(\mu')$, $\mu'= \mu-\delta_{\underline \r^+}$, satisfying the identity on the multiplicities
$$\displaystyle \sum_{j=n}^{k+1} m_{\redu^{-1}(u_j)}+m_{\redu^{-1}(u_k)}+\sum_{j=1}^k m_{\redu^{-1}(u_i)\setminus \{2n\}}=|\delta_{\underline\r}|,$$
and  linear inequalities on the multiplicities of the columns

\begin{align*} &\redu^{-1}(u_i),~~  \redu^{-1}(u_i)\setminus \{u_n\}, \mbox{ $i=2,\dots,n$, on the LHS of $S^{H,\mu'}$},\\
&\mbox{ and the symplectic columns  $u_1u_2\cdots u_i\in SpT_{2n}(\varpi_i) \eqref{noddsymphw-hw}$, $i=1,\dots,n$,}\nonumber\\
&u_1u_2\cdots u_n,~~ u_1u_2\cdots u_{n-1},\dots,u_1u_2,~~u_1
\end{align*} as follows

\begin{align*}
&m_{\redu^{-1}(u_i)}+m_{\redu^{-1}(u_i)\setminus \{2n\}}\le m_{u_1\cdots u_{i-1}}, ~i=2,\dots,n,\\
&0\le m_{\redu^{-1}(u_n)}-\sum_{i=1}^{n-1}m_{\redu^{-1}(u_i)\setminus\{u_n\}}\le m_{u_1\cdots u_{n-1}}.
\end{align*}

With the same setting as above and   $n$ even, one has

\medskip

\noindent \textbf{Theorem} \textbf{E} (Theorem \ref{thm:n=evenslack1})[$\k$-highest weight tableaux produced by $1$-$0$-slack sequence quantum recording tableaux for $n$ even]
For some $0\le k\le n$,
 ${\LRAII^{AII}}^{-1}(S^{H,\mu},Q)$ returns   the $\mathfrak{k}$-highest  weight tableau in
$SST_{2n}(\lambda)$, $\ell(\lambda)\in\{2n-1,2n\}$, with ${\mathfrak{k}}$-weight  $\mu$,    in one of either  forms:

\begin{align*}
&Y(M^{2n})\bigcdot \big(\redu_1^{-1}(2n-1))^{m_{\redu^{-1}(2n-1) }}\bigcdot
{\displaystyle\circledcirc_{j=n-1}^{1}\big(\redu_1^{-1}(u_j)\big)^{m_{\redu^{-1}(u_j) }}}
\bigcdot S^{H,\mu'}\\
=&Y(M^{2n})\bigcdot \big((12\cdots (2n-3).(2n-2).(2n-1)\big)^{m_{12\cdots (2n-3).(2n-2).2n-1 }}\bigcdot
{\displaystyle\circledcirc_{j=n-1}^{1}\big(\redu_1^{-1}(u_j)\big)^{m_{\redu^{-1}(u_j) }}}
\bigcdot S^{H,\mu'},
\end{align*}

or

\begin{align*}
&Y(M^{2n})\bigcdot \big(\redu_1^{-1}(2n-1)\big)^{m_{\redu^{-1}(2n-1) }}
\bigcdot\big(12\cdots (2n-2)\cdot 2n\big)^{m_{12\cdots 2n-2\cdot 2n}}
\bigcdot
\big(12\cdots (2n-2)\big)^{m_{(12\cdots (2n-2)) }}\bigcdot
\nonumber\\
&
\bigcdot
{\displaystyle\circledcirc_{j=n-1}^{1}\big(\redu_1^{-1}(u_j)\setminus\{2n\}\big)^{m_{\redu^{-1}(u_j)\setminus\{2n\} }}}
\bigcdot S^{H,\mu'}\\
=&Y(M^{2n})
\bigcdot\big((12\cdots (2n-3).(2n-2).(2n-1)\big)^{m_{12\cdots (2n-3).(2n-2).2n-1 }}\bigcdot\big(12\cdots (2n-2)\cdot 2n\big)^{m_{12\cdots 2n-2\cdot 2n}}\nonumber\\
&\bigcdot
\big(12\cdots (2n-2)\big)^{m_{12\cdots (2n-2) }}\bigcdot
{\displaystyle\circledcirc_{j=n-1}^{1}\big(\redu_1^{-1}(u_j)\setminus\{2n\}\big)^{m_{\redu^{-1}(u_j)\setminus\{2n\} }}}
\bigcdot S^{H,\mu'}
\end{align*}

or

\begin{align*}
&Y(M^{2n})\bigcdot \big(\redu_1^{-1}(2n-1)\big)^{m_{\redu^{-1}(2n-1) }}
\bigcdot\big(12\cdots (2n-2)\cdot 2n\big)^{m_{12\cdots 2n-2\cdot 2n}}
\bigcdot\nonumber\\
&\bigcdot\displaystyle\circledcirc_{j=n-1}^{k+1}\big(\redu_1^{-1}({u_j})\big)^{m_{\redu^{-1}(u_j)}}
\bigcdot\big(\redu_1^{-1}({u_{k}})\big)^{m_{\redu^{-1}(u_k)}}\bigcdot\big(\redu_1^{-1}(u_k)\setminus\{2n\}\big)^{m_{\redu^{-1}(u_k)\setminus\{2n\} }}\nonumber\\
&\bigcdot
{\displaystyle\circledcirc_{j=k-1}^{1}\big(\redu_1^{-1}(u_j)\setminus\{2n\}\big)^{m_{\redu^{-1}(u_j)\setminus\{2n\} }}}
\bigcdot S^{H,\mu'},
\end{align*}
where the multiplicities satisfy the identity $\displaystyle \sum_{j=n}^{k+1} m_{\redu^{-1}(u_j)}+m_{\redu^{-1}(u_k)}+\sum_{j=k}^1 m_{\redu^{-1}(u_i)\setminus \{2n\}}=|\underline\delta_\r|-M$,
and
 linear inequalities on the multiplicities of the columns

\begin{align*} &\redu^{-1}(u_i),~~  \redu^{-1}(u_i)\setminus \{2n\}, \mbox{ $i=1,2,\dots,n$, } \mbox{ on LHS of $S^{H,\mu'},$}\nonumber\\
&\mbox{ and the symplectic columns  $u_1u_2\cdots u_i\in SpT_{2n}(\varpi_i) \eqref{nevensymphw-hw}$, $i=1,\dots,n$,}\nonumber\\
&u_1u_2\cdots u_n,~ u_1u_2\cdots u_{n-1},\dots,u_1u_2,~~u_1\nonumber
\end{align*} as follows

\begin{align*}
&m_{\redu^{-1}(u_i)}+m_{\redu^{-1}(u_i)\setminus \{2n\}}\le m_{u_1\cdots u_{i-1}}, ~i=2,\dots,n\\
 & m_{12\cdots (2n-2)2n}
 =\sum_{i=1}^{n-1} m_{\redu^{-1}(u_i)\setminus\{2n\}}.
\end{align*}

\medskip

\medskip
We hope that a geometrical object will emerge from the characterization of the $\k$-highest (lowest) weight tableaux by linear inequalities obtained so far  to be  attached to the quantum LR map as it happens in the Thomas LR map \cite{tho78}. The  coupling of the left and right companions of a recording LR tableau in \cite{tho78} (determining also the lowest and highest weights of a connected component in the  crystal tensor product decomposition \eqref{crystalthomas}) led to hives \cite{kt1,knutson}.

\medskip
\medskip
\subsection{Organization}
This paper is organized in five sections. Section \ref{sec:pre} introduces the relevant notation on semistandard tableaux and recalls the Schensted column insertion and its reverse including the relevant properties on the bumping  and reverse bumping routes.  Section \ref{sec:symmetrylrtilde} proves the surjectivity of the map
$\widetilde Rec_{2n}(\lambda/\mu) \overset{\sim}\hookrightarrow  LRS(\lambda/\mu)$ in Theorem \ref{surjection} (Theorem A in the Introduction) and how it exhibits the restriction of the  LR symmetry $c_{\mu,\nu,\lambda}=c_{\lambda^t, \nu^t,\mu^t}$, to LR Sundaram tableaux in Subsection \ref{sec:lozenge}, Corollary 1 in the Introduction. Section \ref{sec:containmentmain} recalls the relevant operations of the algorithm that computes the quantum Littlewood-Richardson map and their properties. Lemma \ref{lem:reductionsurj} and Theorem \ref{lem:reductionsurj2}  %and Theorem \ref{thm:nofactors}
 (Theorem B in the Introduction) on the surjectivity of the reduction map are ones of the main results. These results on the inverse reduction map has been recently  generalized to any symplectic column by the author in \cite{azreduction} as referred in Subsection \ref{sec:inversereduction}.
 The purpose of  Subsection \ref{subsec:essential} is to analyse in detail when the action of the inverse reduce map
on a reverse bumped symplectic columns $\b$ from a tableau $S\in SST_{2n}(\mu)$, with $\ell(\mu)\le n$, returns a column $\redu^{-1}(\b)$ whose concatenation with the returned tableau by the reverse Schensted column bumping on $S$  is a semi-standard tableau. Proposition Proposition \ref{reducbumpinterval1} and Theorem \ref{thm:essentialss} is our current contribution.  With this on hand  we get prepared for the fundamental unwind step of the sequence of successors of that algorithm  to prove the surjectivity of the map
$$\displaystyle SST_{2n}(\lambda) \overset{\sim}\hookrightarrow\bigsqcup_{\begin{smallmatrix}\mu\in Par_{\le n}\\
\mu\subset\lambda\\
Q\in \widetilde Rec_{2n}(\lambda/\mu)\end{smallmatrix}}SpT_{2n}(\mu) \times \{Q\},$$
providing we know that $\widetilde Rec_{2n}(\lambda/\mu) \overset{\sim}\rightarrow  LRS(\lambda/\mu)$. The unwind steps are determined by $Q\in \widetilde Rec_{2n}(\lambda/\mu)$ with   input a symplectic tableau of shape $\mu$, and consist of reverse column insertion ruled by the slack data introduced in  \cite{azslack} and recalled and expanded in Subsection \ref{sec:slackdata} for any quantum recording tableau, followed with reverse removal or the inverse of the reduction map  and concatenation in the plactic monoid.  The fundamental unwind step is explicitly exhibited in Theorem \ref{thm:verygeneralstrip} (Theorem {\bf C} in the Introduction) and its iteration ruled by the slack data appears in Theorem \ref{lem:recordhole1} and Theorem \ref{thm:generalhole} (Theorem {\bf{CC}} in the Introduction).
In the last Section \ref{sec:last}, the inverse of the quantum Littlewood-Richardson map is applied  to characterize  a family of $\mathfrak{k}$-highest weight tableaux in the Naito-Sagaki conjecture recent proof \cite{nsw} based on the  Watanabe  branching rule for $\imath$quantum groups \cite{watanabe}: Corollary \ref{lem:H-L}, Theorem \ref{thm:n=oddslack1}, Theorem \ref{thm:n=evenslack1}, also in the Introduction above, and their corollaries for $n=1,3$ respectively $n=2,4$, provide the family of $\k$-highest wight tableaux generated either  by quantum record tableaux with one vertical strip respectively with  $1$-$0$-slack sequences. In the later case, with shape length $2n-1$ or $2n$.

\section*{Acknowledgements}
The author acknowledges financial support by the Centre for Mathematics of the University of Coimbra (CMUC, https://doi.org/10.54499/UID/00324/2025) under the Portuguese Foundation for Science and Technology (FCT), Grants UID/00324/2025 and UID/PRR/00324/2025.

\section{Preliminaries}\label{sec:pre}

A \emph{partition} $\gamma$ is a weakly decreasing sequence of nonnegative integers $\gamma_1\ge \gamma_2 \ge \cdots$
such that $\gamma_k =0$ for some $k \ge 1$. The   maximal $i$ such that $\gamma_i> 0$ is called the \emph{number}
of \emph{parts} or \emph{length} of $\gamma$, denoted $\ell(\gamma)$. For each $m\ge 0$, the set of partitions of length at most $m$ is denoted by \emph{$Par_{\le m}$}. We assume the inclusion $Par_{\le m}\subseteq Par_{\le k}$ whenever $k\ge m$. Thus we often write the partition $\gamma$ as a vector
 $\gamma = (\gamma_1, \gamma_2, \dots,\gamma_k)$ for $k \ge  \ell(\gamma)$. The \emph{empty} \emph{partition} is the empty sequence $()$ and is regarded as the unique partition of length zero. Given $m\in\Z_{\ge 0}$, $\varpi_m$ denotes the partition of length $m$ whose parts are all $1$, that is, $\varpi_m =(\underbrace{1,\dots,1}_{m})=:(1^m)$. We also write $(a^m)=(\underbrace{a,\dots,a}_m)$ for $a\ge 0$.

 The partition $\gamma$ is said to be \emph{even} if $\gamma_{2i-1}=\gamma_{2i}$ for all $i\ge 1$. In other words, all columns of $\gamma$ have even length and necessarily the length of $\gamma$ is even.

 A partition $\gamma$ is identified
with its \emph{Young diagram} $D(\gamma)$ which is a left and top justified collection of boxes (or cells)
with $\gamma_k$ many boxes in the $k$th row for all $k \in \mathbb{Z}_> 0$. In particular, the empty Young diagram and  the  partition $()$ are identified. The number of cells of $D(\gamma)$ is  the sum of the parts of $\gamma$ and is denoted by $|\gamma|$.
The boxes or cells of the Young diagram of $\gamma$ are identified by its  coordinates $(i,j)$ in the matrix style, that is, $1\le i\le \ell(\gamma)$ and $1\le j\le \gamma_i$.

Let $\gamma$, $\mu$  partitions   with $\mu \subseteq \gamma$, that is,  $\mu_i\le  \gamma_i$ for all $ i\in\mathbb{Z}_>0 $, or the Young diagram of $\mu$ is a subset of
the Young diagram of $\gamma$. The skew-diagram (or skew Young diagram) $\gamma/\mu$ is defined to be $D(\gamma)\setminus D(\mu)$.
For $\lambda\subseteq\gamma$ partitions, we write $\lambda\subseteq_{vert} \gamma$ to mean that  $\gamma/\lambda$  is a \emph{vertical} \emph{strip}, that is, the skew-diagram $\gamma/\lambda$ has at most one box in each row. The number of cells of $\gamma/\lambda$ is $|\gamma/\lambda|:=|\gamma|-|\lambda|.$

 A \emph{
tableau} $T$ of (skew) shape $\gamma/\mu$ is a map (or a filling of $D(\gamma)$)
 $$T:D(\gamma)\rightarrow \mathbb{Z}_\ge0, \; (i,j)\mapsto T(i,j),$$
assigning a positive integer to each box of $D(\gamma)\setminus D(\mu)$ and   $0$  to the boxes of $D(\mu)$. Denote by $\rm{Tab}(\gamma/\mu)$ the set of tableaux of shape $\gamma/\mu$.  We say that the tableau $T$ is \emph{semi-standard} if an addition the assignment is
such that it is weakly increasing as we go from left to right along a row and strictly
increasing as we go from top to bottom along a column excluding the boxes in $\mu$,
\begin{align}&T(i,j)\le T(i,j+1),\; T(i,j)<T(i+1,j), \mbox{ for all $(i,j)\in D(\gamma)\setminus D(\mu)$,}\nonumber\\
&\mbox{ and }  T(i,j)=0, \mbox{ for $(i,j)\in  D(\mu)$,}
\nonumber
\end{align}
where we set $T(a,b):=\infty$ if $(a, b) \notin D(\gamma)$.
 Usually $T(i,j)$ is just referred as
the entry in the box $(i,j)$ and we omit the zeroes  in the boxes of $\mu$. A positive integer
$m\ge \ell(\gamma)$ will be fixed and $[0, m]:=\{0, 1, \dots, m\}$ will be used as a co-domain for the map $T$. We call
$[m]:=\{1, \dots, m\}$ the alphabet of the semi-standard tableau $T$.
 In this case, we will denote the \emph{set} of \emph{semi-standard} \emph{tableaux} of \emph{shape} $\gamma/\mu$ by
$SST_m(\gamma/\mu)$. When $\mu=()$, we just write $SST_m(\gamma)$. The \emph{weight} or \emph{content} of $T$ is the nonnegative vector $\mathrm{wt}(T)=(T[1],\dots,  T[m])$, where $T[i]:=\#\{(a,b)\in D(\gamma):T(a,b)=i\}$ for $i\in[m]$, that is, $T[i]$ is  the number of occurrences of $i$ in the tableau $T$.

The {\emph{reverse}} {\emph{row}} \emph{word} of a semi-standard tableau
$T$, denoted $w(T)=w_1\cdots w_l$, with $l$ the number of non zero entries of $T$,  is obtained by reading the entries of its rows (excluding the entry 0) right to left
starting from the top row and proceeding downwards. The \emph{column} \emph{word} of a semi-standard tableau $T$ denoted $w^{col}(T)=w'_1\cdots w_l'$  is the sequence of entries obtained by  reading the entries (excluding the entry $0$) of its columns
 from bottom to top and left to right.
The \emph{reverse} \emph{column} \emph{word}  of $T$, $w'_l\cdots w'_1$.
 %denoted $w^{col}(T)=w'_1\cdots w_l'$
is obtained by reading the entries of its columns (excluding the entry 0) right to left
starting from the top and proceeding downwards.
The \emph{weight} of the \emph{word}  $w(T)$ is the weight of $T$.

 A \emph{Yamanouchi} \emph{word} is a word $u_1 \cdots u_l$ such
that, for each $1 \le k \le l$,  the  weight of the subword  $u_1 \cdots u_k$ is a partition. A \emph{Yamanouchi} \emph{tableau} of shape $\mu$, denoted $Y(\mu)$, is a semi-standard tableau of shape and weight $\mu$. That is,  its  reverse column word is Yamanouchi and $Y(i,j)=i$ for all $1\le i\le \ell(\mu)$ and $1\le j\le\mu_i$.

\subsection{ Schensted column insertion and  reverse }\label{subsec:insert}
The Schensted \emph{column} \emph{insertion} or
\emph{column} \emph{bumping} \cite{fulton,stanley},  takes a positive integer $x$ and a tableau $T\in SST(\gamma)$ (semistandard tableau of shape $\gamma$) and
puts $x$ in a new box at the bottom of the first column (counting from left to right)  if  it is strictly
larger than all the entries of the column. If not, it \emph{bumps} the  \emph{smallest}
\emph{entry}
in the \emph{column} that is \emph{larger} \emph{than} or \emph{equal} to $x$. The bumped entry moves to the
next column, going to the end if possible, and bumping an element to the next column
otherwise. The process terminates when  the bumped entry  goes to the bottom of the next
column, or until it becomes the only entry of a new column. The returned tableau is
denoted $x\bigcdot T:=x \rightarrow T$, that is, $x\bigcdot T$ is the result of column inserting $x$ into $T$, $x \rightarrow T$. We call to $ x \bigcdot T$ the multiplication of $x$ and $T$ in the plactic monoid.

The \emph{reverse Schensted column insertion } takes a tableau $U\in SST(\gamma)$ and an entry $y$ at the bottom of a column of $U$, whose cell has  \emph{row} \emph{coordinate}, say $r_y$,  and bumps it while deletes its box. If the column is the first, the new tableau is obtained by deleting the bottom box of the first column of $U$. If not, the bumped $y$ moves  to the next column on its left and \emph{bumps} the  \emph{largest}
\emph{entry} in the \emph{column} that is \emph{smaller} than or \emph{equal} to $y$. In both cases, if the bumped entry belongs to the first column the process ends and returns a pair $(b, ~\underset{ r_y}\leftarrow U)$ consisting of the bumped entry $b\le y$ from the first column of $U$ together with a new  tableau denoted $\underset{r_y}\leftarrow U$, with one box less than $U$. If not, the process continues until an entry is bumped from the first column of $U$.
\begin{obs}\label{noexplicit} When we want to emphasize that in the reverse Schensted column insertion we are considering  the pair $(b, ~\underset{ r_y}\leftarrow U)$ without making explicit the bumped entry $b$ from the first column of $U$, we just write
\begin{align*}(\quad , ~\underset{ r_y}\leftarrow U).
\end{align*}
\end{obs}

We now recall the bumping routes defined by the column Schensted insertion and by the reverse column Schensted insertion.
\begin{obs}\label{routes}\cite[Appendix A.2, Exercise 3]{fulton}  Let us consider two successive column-insertions of $ x < x'$ into a tableau $T$. First column-inserting $x$ in $T$ and then column-inserting $x'$ in
the resulting tableau $x\rightarrow T$, yielding  to two \emph{bumping} \emph{routes} $R$ and $R'$, and
two new boxes $B$ and $B'$. Then $R'$ lies strictly below $R$,
and $B'$ is Southwest of $ B$.

If $x\ge x'$, $R'$ lies weakly above $R$ and $B'$ Northeast of $B$.

\end{obs}
Since column insertion is reversible, one has the following \emph{reverse bumping route} property for \emph{reverse} \emph{column} \emph{insertion}.

\begin{obs}\label{revroutes}
\begin{enumerate}
\item[(a)]
Let $ z$ and   $z'$  be two bottom column entries of a tableau $U$, with $z'$  \emph{strictly} \emph{Southwest} of $z$ (not in the same row), or let $z<z'$ be the two bottom entries of a column of  $U$. Let $r_z<{r_{z'}}$ be the corresponding cell row coordinates.

 First consider the reverse column-inserting of $z'$  and then reverse column-inserting of $z$ in
the resulting tableau $(\underset{r_{z'}}\leftarrow U)$. This yields   two reverse bumping routes $Z'$ respectively $Z$ such that $Z'$ \emph{lies} \emph{strictly} \emph{below} $Z$.  In particular, if the bumping routes terminating rows terminating rows are respectively $i'>i$, the
two new entries $a'\le z'$ respectively $a\le z$,   in the first column of $U$ are such that $(\underset{r_{z'}}\leftarrow U)(i',1)=a'
> (\underset{r_{z}}\leftarrow(\underset{r_{z'}}\leftarrow U))(i,1)=a
$ and $i'>i$. The bumped the pair $U(i,1)=b<U(i',1)=b'$ from the first column of $U$ are such that $ (b,b')\le (a,a')\le (z,z')$, that is, $b\le a\le z$, $b'\le a'\le z'$.
Generalizing the notation in Remark \ref{noexplicit}, we write

\begin{align}(\quad , \underset{ (r_z,r_{z'})}\leftarrow U):=((b,b'),\underset{r_{z}}\leftarrow(\underset{r_{z'}}\leftarrow U))
\end{align}
where  the  entry in the cell $(u,1)$ of the first column of $\underset{ (r_z,r_{z'})}\leftarrow U$ is such that $$ (\underset{ (r_z,r_{z'})}\leftarrow U)(u,1)=\begin{cases}(\underset{r_{z}}\leftarrow(\underset{r_{z'}}\leftarrow U))(i,1)=a& \mbox{ if $u=i$},\\
(\underset{r_{z'}}\leftarrow U)(i',1)=a'& \mbox{ if $u=i'>i$},\\
U(u,1)&  \mbox{ if $u\neq i, i'$}.
\end{cases}$$

 \item[(b)] In case $y'$  \emph{strictly} \emph{Southwest} of $y$ (not in the same row) and we first consider reverse column-inserting of $y$  followed with the  reverse column-inserting $y'$ in
the resulting tableau $(\underset{r_{y}}\leftarrow T)$ then the corresponding reverse bumping routes $Z$ and $Z'$ are such that $Z'$ \emph{lies} \emph{weakly} \emph{below} $Z$.
The two new entries $b$ respectively $b'$ in the first column of $T$, $(\underset{r_{y'}}\leftarrow(\underset{r_{y}}\leftarrow T))(i',1)=b'>(\underset{r_{y}}\leftarrow T)(i,1)=b$ with $i'>i$, or $i=i'$ and $(\underset{r_{y}}\leftarrow T)(i,1)=b\le (\underset{r_{y'}}\leftarrow(\underset{r_{y}}\leftarrow T))(i,1)=b'$.

 \item [(c)] If two or more columns have the \emph{bottom} \emph{entries} in a \emph{same} \emph{row}, let $y'$ be the \emph{right} \emph{most} \emph{bottom} \emph{entry} in that \emph{row}.  Consider the reverse column-inserting of $y'$ in $T$ and then again the reverse column-inserting of $y'$ in
the resulting tableau $\underset{r_{y'}}\leftarrow T$. This yields   two bumping routes $Z_1$ respectively $Z_2$ such that $Z_2$ \emph{lies} \emph{weakly} \emph{below} $Z_1$ and two new entries $b_1$ respectively $b_2$ in the first column of $T$, such that  $(\underset{r_{y'}}\leftarrow(\underset{r_{y'}}\leftarrow T))(i_2,1)=b_2> (\underset{r_{y'}}\leftarrow T))(i_1,1)=b_1$ with $i_2>i_1$, or $i_1=i_2$ and $(\underset{r_{y'}}\leftarrow T)(i_1,1)=b_1\le (\underset{r_{y'}}\leftarrow(\underset{r_{y'}}\leftarrow T))(i_1,1)=b_2$.
\end{enumerate}
\end{obs}
\medskip

 %For $\lambda\subseteq\gamma$ partitions, we write $\lambda\subset_{vert} \gamma$ to mean that  $\gamma/\lambda$  is a \emph{vertical} %\emph{strip}, that is, the skew-diagram $\gamma/\lambda$ has at most one box in each row. The \emph{length} of $\gamma/\lambda$ is %$|\gamma/\lambda|:=|\gamma|-|\lambda|.$

 Let $\lambda\in Par_{\le m}$ and $k\in [0,m]$.
 The following assignment \cite{watanabe} is the combinatorial Pieri's rule, defined by the column insertion of a column tableau into a tableau, is a bijection

\begin{align}\bigcdot:SST_m(\varpi_k)\times SST_m(\lambda)&\rightarrow\bigsqcup_{\begin{smallmatrix}
\gamma\in Par_{\le m}\\
\lambda\subseteq_{vert} \gamma,~|\gamma/\lambda|=k
\end{smallmatrix}}
SST_m(\gamma) \nonumber\\
(S,T)&\mapsto S\bigcdot T:=w_k\rightarrow(\cdots\rightarrow( w_2 \rightarrow(w_1\rightarrow T))\cdots)
\end{align}
where $S=(w_1<w_2<\dots<w_k)$ is a column of length $k$ with entries in $[1,m]$. In particular, if $S(i,1)\le T(i,1)$ for every $1\le i\le m$, then $S\bigcdot T$ is the concatenation of $S$ and $T$.

The process $ S\bigcdot T=w_k\rightarrow(\cdots( w_2 \rightarrow(w_1\rightarrow T))$ is reversible,
\begin{align}\label{reverse} &  (S, T)=((w_1,\dots, w_{k-1}, w_k),\underset{\gamma/\lambda}\leftarrow S\bigcdot T):=\nonumber\\
&:=((w_1,\dots,w_{k-1},w_k),\underset{r_{w'_{1}}}\leftarrow(\cdots( (w_{k-1},w_k),\underset{r_{w'_{k-1}}} \leftarrow(w_k,\underset{r_{w'_k}}\leftarrow S\bigcdot T))\cdots))\end{align}
where the  entries of the vertical strip $\gamma/\lambda$ from bottom to top define the word $w'_k\cdots w'_1$ and $r_{w'_i}$ indicates the row index of $w'_i$ with respect to the partition $\gamma$, for $1\le i\le k$. For $U\in SST_m(\gamma)$ such that $\lambda\subseteq_{vert}\gamma$ the  inverse map is defined by the reverse column insertion by taking the entries of $U$ in the vertical strip $\gamma/\lambda$  from bottom to top,
 \begin{align} (S,\underset{\gamma/\lambda}\leftarrow U)\in SST_m(\varpi_k)\times SST_m(\lambda), \mbox{ for some $S\in  SST_m(\varpi_k)$.} \label{keystepinverseLRAII}
 \end{align}

\section{ The linear time bijection between the sets $ \widetilde Rec_{2n}(\lambda/\mu)$ and  LR-Sundaram tableaux $LRS_{2n}(\lambda/\mu)$}\label{sec:symmetrylrtilde}
Recall  a partition $\nu$ is said to be even if  all columns of $\nu$ have even length and necessarily the length of $\nu$ is even. In other words, $\nu_{2i-1}=\nu_{2i}$ for all $i\ge 1$. In this case, $\nu^t$ the transpose or conjugate of $\nu$ has all rows of even length.
We start by recalling the definition of Littlewood-Richardson-Sundaram tableau (or symplectic Littlewood-Richardson tableau).
Fix $n\in\mathbb{N}$. Let $\lambda \in Par_{\le 2n}$ and $\mu \in Par_{\le n}$.

\begin{defi}\label{def:lrs} Let $\lambda \in Par_{\le 2n}$ and $\mu \in Par_{\le n }$ be such that $\mu\subset\lambda$. A semi-standard tableau $T \in Tab(\lambda/\mu) $ is said to be
an $n$-symplectic Littlewood-Richardson tableau or Littlewood-Richardson-Sundaram tableau if it satisfies the following.
\begin{enumerate}
\item[$(1)$] $T \in SST_{2n}(\lambda/\mu)$.

\item [(2)] %Let $(w_1, \dots,w_N)$ denote the column-word $w^{col}(T)$ of $T$. Then,
The reversed column word of $T$
%$(w_N, \dots,w_1)$
is a Yamanouchi word.
% For each $r \in [1,N]$ and $k \in [1, 2n -1]$,
%the number of occurrences of $k$ in the subsequence $(w_N, \dots ,w_r)$ is greater than or
%equal to that of $k +1$.

\item [$(3)$] The sequence $\rm{wt}(T) = (T[1], T[2],\dots, T[2n])$ is an even partition,

\item [$(4)$] If $T(i, j) = 2k + 1$ for some $(i, j) \in  D(\lambda/\mu)$ and $k \in\mathbb{Z}_{\ge 0}$ then we have
$i \le n + k$.

\end{enumerate}
\end{defi}

For $\lambda \in Par_{\le 2n}$ and $\mu \in Par_{\le n}$, the subset of $SST_{2n}(\lambda/\mu)$ satisfying conditions $(2)$, $(3)$ and  $(4)$ in Definition \ref{def:lrs} is denoted by $LRS_{2n}(\lambda/\mu)$.

\begin{obs} Condition $(4)$ above can be replaced by  $T(n + i, 1) \ge 2i$ for every $i\ge 0$.
\end{obs}

\begin{obs}  Let $T\in LRS_{2n}(\lambda/\mu)$ of content $\nu$. Let $N:=\nu_1$ and $\nu^t=(\nu_1^t,\dots,\nu^t_{N})$ the transpose of $\nu$. Put  $\mu^{(0)}:=\lambda$ and $\mu^{(N)}:=\mu$. For $k=1,\dots,N$,
 %let $T$ restricted to the shape $\mu^{(k-1)}$, and
 define $\mu^{(k)}$ the shape obtained  by erasing from NE to SW in $T_{|\mu^{(k-1)}}$, $T$ restricted to the shape  $\mu^{(k-1)}$, the first rightmost $\nu^t_k$ cells filled with $1,2,\dots, \nu^t_k$. Hence

 \begin{align}\mu^{(0)}=\lambda\supset_{vert} \mu^{(1)}\supset_{vert} \cdots\supset_{vert}\mu^{(N)} =\mu \label{nested1}
 \end{align}
This decomposes the shape $\lambda/\mu$ into $N$ vertical strips which completely define $T$

\begin{align}J_1:=\mu^{(0)}/\mu^{(1)},J_2:= \mu^{(1)}/\mu^{(2)}, J_i=\mu^{(i-1)}/\mu^i,\dots, J_{N}=\mu^{(N-1)}/\mu^{(N)}\label{nested2}
\end{align}

%Let $\nu^t=(\nu_1^t,\dots,\nu^t_{N})$ be the transpose of $\nu$.
 Then, for $k=1,\dots,N$,
the vertical strip $J_k=\mu^{(k-1)}/\mu^{(k)}$ has  $\nu^t_k=|\mu^{(k-1)}|-|\mu^{(k)}|$ cells and  is filled in  $T$, top to bottom,  with
$12\cdots \nu_k^t$. %Often we call \emph{string} to  $J_k$ when seen filled with its
%reverse column-word $12\cdots \nu_k^t$, for $k\ge 1$.

Concatenating those column words as $12\cdots \nu^t_{N}\bigcdot\cdots\bigcdot 12\bigcdot\cdots\bigcdot\nu^t_2\bigcdot 12\cdots \nu^t_{1}$ gives the reverse-column word of the  Yamanouchi tableau of shape $\nu$ which is also obtained by rectification of the word of $T$.
 For an illustration, see Example \ref{ex:reclrs}.

%The length of the vertical strip $J_k=\mu^{(k-1)}/\mu^{(k)}$ is
%$Q[k]$,  $\ell(J_k)=\ell(\mu^{(k-1)}/\mu^{(k)})=Q[k]$, $k=1,2,3$.
\end{obs}

%Let $\lambda \in Par_{\le 2n}$ and $\mu \in Par_{\le n }$ be such that $\mu\subset\lambda$. A tableau $Q$
%of shape $\lambda/\mu$ is said to be a \emph{recording} \emph{tableau} if there exists $T \in SST_{2n}(\lambda)$ such that
%$Q^{AII}(T) = Q$. Let $Rec_{2n}(\lambda/\mu)$ denote the set of \emph{recording} \emph{tableaux} \emph{of} \emph{shape} $\lambda/\mu$.

We now recall  the definition of the set $\widetilde Rec_{2n}(\lambda/\mu)$.

\begin{defi}\label{def:tilde}\cite{watanabe}\label{def:tilderec} Let $\lambda \in Par_{\le 2n}$ and $\mu \in Par_{\le n }$ be such that $\mu\subset\lambda$. Let $\widetilde Rec_{2n}(\lambda/\mu)$ denote the set of
tableaux $Q$ of shape $\lambda/\mu$ satisfying the following:
\begin{enumerate}
\item[$(R1)$] The entries of $Q$ strictly decrease along the rows from left to right.

\item[$(R2)$] The entries of $ Q$ weakly decrease along the columns from top to bottom.

\item[$(R3)$] For each $k > 0$, the number $Q[k]$ of occurrences of $k$  is even.

\item[$(R4)$] For each $k > 0$, it holds that

$$Q[k] \ge 2(\ell(\mu^{(k-1)})- n),$$
\noindent where $\mu^{(k-1)}$ is the partition such that
$$D(\mu^{(k-1)}) = D(\mu) \cup\{(i, j) \in D(\lambda/\mu) | Q(i, j) \ge  k\}.$$

\item[$(R5)$] For each $r, k >0 $, let $Q_{\le r}[k]$ denote the number of occurrences of $k$ in $Q$ in the
$r$-th row or above. Then, the following inequality holds:
$$Q_{\le r}[k + 1] \le Q_{\le r}[k].$$

\end{enumerate}

\end{defi}

\medskip
\begin{obs} Let $N$ be the largest entry in $Q$ where we are assuming the blank boxes filled with $0$. When $N=0$ one has $\lambda=\mu \Leftrightarrow\lambda/\lambda=()\Leftrightarrow \rm{wt}(Q)=()$. Let $N\ge 1$.
Since $Q[k]\in 2\mathbb{Z}$ and  $Q[k+1]\le Q[k]\le \ell(\lambda)\le 2n$, for any $k>0$,
 the weight of $Q$ is   the partition $\rm{wt}(Q)=(Q[1], \dots, Q[N])$
 and  its  transpose or conjugate is an even partition $\nu=(\nu_1=N,\dots,\nu_{Q[1]})$.
Thus
 $Q$ is also defined by the sequence of nested partitions %in \eqref{nested}

\begin{align}\mu^{(0)}=\lambda\supset_{vert} \mu^{(1)}\supset_{vert} \cdots\supset_{vert}\mu^{(N)} =\mu,~~
\end{align}

\noindent  where $\mu^{(i-1)}/\mu^{(i)}$ is a vertical strip of even length $Q[i]\ge 2$ satisfying $(R4)$ and $(R5)$ conditions, for $i=1,\dots,N$.
%These vertical strips are denoted by $J_i:=\mu^{(i-1)}/\mu^{(i)}$ to emphasize that as vertical strips of $Q$ they are %filled with $i$'s in $Q$, $i=1,\dots,N$.

\iffalse

Since $Q[k]\in 2\mathbb{Z}$ and  $Q[k+1]\le Q[k]\le \ell(\lambda)\le 2n$, for any $k>0$,
 the weight of $Q$ is  defined to be the partition $\rm{wt}(Q)=(Q[1], \dots, Q[\lambda_1])$. Its  transpose or conjugate is an even partition $\nu=(\nu_1,\dots,\nu_{Q([1]})$.
Thus $Q$ is also defined by the sequence of nested partitions
%$$\mu^{(0)}=\lambda\supseteq \mu^{(1)}=(3,2,2,1,0^2)\supseteq \mu^{(2)}=(3,1,1,0^3)\supseteq\mu^{(3)} =\mu$$

$$\mu^{(0)}=\lambda\supseteq_{vert} \mu^{(1)}\supseteq_{vert} \cdots\supseteq_{vert}\mu^{(\nu_1)} =\mu$$

\noindent where $\mu^{i-1}/\mu^i$ are vertical strips of even length $Q[i]$ satisfying $(R4)$ and $(R5)$ conditions, for $i=1,\dots,\nu_1$.  As before these vertical strips are denoted by $J_i:=\mu^{i-1}/\mu^i$ to emphasize that as vertical strips of $Q$ they are filled with $i$'s in $Q$., $i=1,\dots,\nu_1$.
\fi

\end{obs}

\begin{prop} Given $Q\in \widetilde Rec_{2n}(\lambda/\mu)$ and  $k>0$, the following conditions  hold
\begin{enumerate}
%\item[$(a)$] $Q[k]\in 2\mathbb{Z}$ and  $Q[k+1]\le Q[k]\le \ell(\lambda)\le 2n$.

%Q_{\ell(\mu^{(r-1)})}[r]\le R_{\le i}[r]+(\ell(\mu^{(r-1)}-i)$

\item[$(a)$] $Q[k]=Q_{\le \ell(\mu^{(k-1)})}[k]$.

\item[$(b)$] If $1\le i\le \ell(\mu^{(k-1)})$ %is such that $Q(i,j)=k$,
then $Q[k]=Q_{\le \ell(\mu^{(k-1)})}[k]\le Q_{\le i}[k]+(\ell(\mu^{(k-1)})-i)$.

\item[$(c)$] $Q[k] \ge 2(\ell(\mu^{(k-1)})- n)$ (condition $(R4)$) if and only if
$Q_{\le i}[k] \ge 2(i- n)$, for all
 $1\le i\le \ell(\mu^{(k-1)})$.

%\item[$(c)$] $Q\in \widetilde Rec_{2n}(\lambda/\mu)$  $ LRS(\lambda/\mu$.
\end{enumerate}
\end{prop}
\begin{proof} $(a)$ Evident because  $J_k:=\mu^{(k-1)}/\mu^{(k)}=\{(i,j)| Q(i,j)=k\}\subseteq \mu^{(k-1)}$ and $\ell(J_k)=Q[k]$. Therefore $$Q(i,j)=k\Rightarrow 1\le i\le \ell(\mu^{(k-1)}).$$

$(b)$ Let   $1\le i\le \ell(\mu^{(k-1)})$. Then
\begin{align*}&J_k:=\mu^{(k-1)}/\mu^{(k)}=\{(s,j)| Q(s,j)=k, s\le i\}\cup \{(s,j)| Q(s,j)=k, s> i\}\\
&\subseteq \{(s,j)| Q(s,j)=k, s\le i\}\cup \{i+1,\dots, \ell(\mu^{(k-1)})\}
\end{align*}
Therefore $\ell(J_k)=Q[k]\le Q_{\le i}[k]+ (\ell(\mu^{(k-1)}-i).$

$(c)$ Let $1\le i\le \ell(\mu^{(k-1)})$. From $(R4)$, $Q[k]\ge 2(\mu^{(k-1)}-n)$. By contradiction suppose, $Q_{\le i}[k] < 2(i- n)$. Then
\begin{align*} Q[k]&\le Q_{\le i}[k]+ (\ell(\mu^{(k-1)}-i)\\
&< 2(i- n)+(\ell(\mu^{(k-1)}-i)\\
&=\ell(\mu^{(k-1)}-2n+i\\
&\le\ell(\mu^{(k-1)}-2n+\ell(\mu^{(k-1)}=2(\ell(\mu^{(k-1)}-n).\\
\end{align*}
Hence $Q[k]<2(\ell(\mu^{(k-1)}-n)$ a contradiction with $(R4)$.
\end{proof}

\begin{cor}
$(a)$ For $\lambda=\mu \in Par_{\le n }$, $\widetilde Rec_{2n}(\lambda/\lambda)=\{D(\lambda)\}$.

$(b)$ $ \label{recolumn}\widetilde Rec_{2n}(\varpi_l)\neq \emptyset$ if and only if $l$ even and $l\le 2n$.
In this case,

\begin{align*} \widetilde Rec_{2n}(\varpi_l)=\{\YT{0.15in}{}{
 {{1}},
 {{1}},
 {\vdots},
 {{1}},
 {{1}},
}\}
\end{align*}
\end{cor}

\begin{proof} From $(R3)$, $l=Q[1]$ is even and from $(R4)$, $Q[1]=l\ge 2(l-n)\Leftrightarrow l\le 2n$.
\end{proof}

Let $\nu$ be an even partition. Then the $\nu_1$ columns are of even  length and  its transpose or conjugate partition $\nu^t=(\nu^t_1,\dots,\nu^t_{\nu_1})$ is such that $\nu_i^t\in 2\Z$, for all $i\ge 1$.

The following theorem asserts that the conditions on a LR-Sundaram tableau $T$ translate to conditions on the lengths of their vertical strips $J_i$ \eqref{nested1}, \eqref{nested2}  and gives the right inverse of the map in \cite[Lemma 8.3.2]{watanabe}.

\begin{thm} \label{surjection} Let $T\in LRS_{2n}(\lambda/\mu)$ of weight $\nu$ and $J_1,\dots,J_{\nu_1}$ its decomposition into vertical strips each strip filled in $1,2,\dots, |J_i|=\nu^t_i\in 2\mathbb{Z}$ for $i=1,\dots,\nu_1$.
%Then $\ell(J_i)\in 2\mathbb{Z}$.
Relabel each vertical strip $J_i=\mu^{(i-1)}/\mu^{(i)}$ with $\nu^t_i$ $i$'s, for $i=1,\dots,\nu_1$. We get a new tableau $Q\in \widetilde Rec_{2n}(\lambda/\mu)$ of weight $\nu^t$. We denote this map by $\lozenge$.
\end{thm}

\begin{proof}
We have to prove that $Q$ satisfies $(R4)$,
$Q[r] \ge 2(\ell(\mu^{(r-1)})- n),$ for each $r=1,\dots,\nu_1$.

 Let $r>0$,  $n< i\le \ell(\mu^{(r-1)})$ and  $(i,j)$ a cell  in  the vertical strip $J_r$ such that $T(i,j)=2k+1$, for some $k>0$. Since $T$ is an LR-Sundaram tableau  it implies $n+k\ge i$.
 One then has $Q(i,j)=r$ and $Q_{\le i}[r]=2k+1$ and necessarily
$$ Q_{\le i}[r]=2k+1> 2(i-n)\Leftrightarrow Q_{\le i}[r]=2k+1\ge 2(i-n)+1$$

Otherwise one obtains
\begin{align} 2k+1\le 2(i-n)\Leftrightarrow 2k+2n\le 2i-1\Leftrightarrow k+n\le i-1/2\Rightarrow k+n<i
\end{align}
A contradiction with the Sundaram condition.

If $T(i+1,j')=2(k+1)$ for some $j'\le j$ and the cell $(i+1,j') \in J_r$ then $Q(i+1,j')=r$ and
$$Q_{\le i+1}[r]=2(k+1)=2k+2\ge  2(i-n)+1+1=2(i+1-n)$$

 \medskip
 By induction on $\nu_1$. If $\nu_1=1$, then $\nu=(1^m, 0^{2n-m})$ with $m$ even. %let $T^1$ obtained from $T$ by erasing all strings $J_1,\dots, J_{\nu_1-1}$.
 Then
$T$ is the vertical strip $\lambda/\mu$ filled with $\{1,\dots,m$ and $Q[1]=m$. If $\ell(\lambda)- n\le 0$ there is nothing to prove, $Q[1]=m>2(\ell(\lambda)-n)$. Otherwise we have to check the $\ell(\lambda)-n$ entries $T(n+1,1)=m-(\ell(\lambda)-n)+1,\dots, T(\ell(\lambda),1)=m$.

If $\ell(\lambda)-n$ is even the entries form a sequence of $\frac{\ell(\lambda)-n}{2}$ pairs of odd and even pairs in consecutive rows in the same column and this case has been already studied above. In particular $Q[1]=Q_{\le \ell(\lambda)}[1]=m>2(\ell(\lambda)-n)$

If $\ell(\lambda)-n$ is odd we have just to check $T(n+1,1)=even\le m$ because the remaining  $\ell(\lambda)-n-1$ entries were already studied above. In particular $Q[1]=Q_{\le \ell(\lambda)}[1]=m>2(\ell(\lambda)-n)$.

Indeed, $Q_{\le n+1}[1]=T(n+1,1)=even\ge 2(n+1-m)=2$

For $\nu_1>1$, let $T^{(1)}$ be the LRS tableau of shape $\mu^{(1)}$ obtained from $T$ by suppressing the string $J_1$. Then $T^1\in LRS(\mu^1/\mu)$, of weight  the even partition $\nu^{(1)}=\nu-(1^{\ell(\nu)},0^{2n-\ell(\nu)})$, that is, its columns are still even and
its conjugate partition is $(\nu'_2,\dots,\nu'_{\nu_1})$ and with strings $J_2,\dots J_{\nu_1}$. By induction one has

$$Q[r] \ge 2(\ell(\mu^{(r-1)})- n), \mbox{ for each $r=2,\dots,\nu_1$}$$
and for each $r=2,\dots,\nu_1$
$$Q_{\le i}[r] \ge 2(i- n), \mbox{ for $1\le i \le \ell(\mu^{(r-1)})$}$$

On the other hand because $T$ is an LR tableau

$$Q_{\le r}[k + 1] \le Q_{\le r}[k], \mbox{ for $r\ge 1, k>0$}.$$

We want to show that

$$Q[1] \ge 2(\ell(\mu^{(0)})- n)=2(\ell(\lambda)- n)$$

  One has $\ell(\lambda)\ge \ell(\mu^{(1)})$ and $\mu^{(1)}\subseteq \lambda$, henceforth
  $$ Q_{\le \ell(\mu^{(1)})}[1]\ge Q_{\le \ell(\mu^{(1)})}[2]\ge  2(\ell(\mu^{(1)})- n)$$

  If $\ell(\lambda)= \ell(\mu^{(1)})$, it is done. Otherwise, $\lambda=(\mu^{(1)},1^{\ell(\lambda)- \ell(\mu^{(1)})})$ (we  are omitting the zero entries in $\lambda$ and $\mu^{(1)}$). It remains to check the $m=\ell(\lambda)-\ell(\mu^{(1)})(\Leftrightarrow \ell(\lambda)=\ell(\mu^{(1)})+m)$ entries of $T$,

  $T(\ell(\mu^{(1)})+1,1)=Q_{\le \ell(\mu^{(1)})}[1]+1,\dots, T(\ell(\lambda),1)=Q_{\le \ell(\mu^{(1)})}[1]+m=Q[1]$.

  As in the case $\nu_1=1$ it remains to study the case $m=odd\Leftrightarrow T(\ell(\mu^{(1)})+1,1)=Q_{\le \ell(\mu^{(1)})}[1]+1=even\Leftrightarrow Q_{\le \ell(\mu^{(1)})}[1]=odd$.
  If  $Q_{\le \ell(\mu^{(1)})}[1]=odd$, then

  $$ Q_{\le \ell(\mu^{(1)})}[1]=odd\ge Q_{\le \ell(\mu^{(1)})}[2]\ge  2(\ell(\mu^{(1)})- n)=even\Rightarrow Q_{\le \ell(\mu^{(1)})}[1]+1\ge 2(\ell(\mu^{(1)})+1- n)$$

  Hence
  $$Q_{\le \ell(\mu^{(1)})}[1]\ge 2(\ell(\mu^{(1)})- n)+1\Rightarrow Q_{\le \ell(\mu^{(1)})}[1]+1\ge 2(\ell(\mu^{(1)})- n)+2=2((\ell(\mu^{(1)})+1- n)$$

  Finally, since the cell $(\ell(\mu^{(1)})+1,1) \in J_1$ it means that $Q(\ell(\mu^{(1)})+1,1)=1$ and
  $$Q_{\le \ell(\mu^{(1)})+1}[1]=Q_{\le \ell(\mu^{(1)})}[1]+1\ge 2((\ell(\mu^{(1)})+1- n)$$
  as desired. (We have assumed that $n\le \ell(\mu^{(1)})$ which is the worst case.)

\end{proof}

\begin{ex}\label{ex:reclrs} Let $n=3$, $\lambda=(4,3,2,2,1,0)$, $\mu=(3,1,0)$, $\nu=(3,3,1,1,0^2)$, $\nu^t=(4,2,2)$ and $T\in LRS(\lambda/\mu, \nu)$ and $Q\in \widetilde Rec_{6}(\lambda/\mu,\nu^t)$
\begin{align} Q=\YT{0.15in}{}{
 {{},{},{},{1}},
 {{},{2},{1}},
 {{3},{2}},
 {{3},{1}},
 {{1}},
}\in \widetilde Rec_{6}(\lambda/\mu,\nu^t)\underset{\sim}{\overset{\lozenge}\longleftarrow} T=\YT{0.15in}{}{
 {{},{},{},{1}},
 {{},{1},{2}},
 {{1},{2}},
 {{2},{3}},
 {{4}},
}\in LRS_6(\lambda/\mu,\nu)
\end{align}
with $\rm{wt}(Q)=(4,2,2,0^3)=\nu^t=(Q[1],Q[2],Q[3])$ and the conjugate partition is the even partition $\nu=(3,3,1,1,0^2)$.
The tableaux $Q$ and $T$  are both defined by the sequence of nested partitions where $N=3=\nu_1$,

$$\mu^{(0)}=\lambda\supset_{vert} \mu^{(1)}=(3,2,2,1,0^2)\supset_{vert} \mu^{(2)}=(3,1,1,1,0^3)\supset_{vert}\mu^{(3)} =\mu$$

%$$\mu^{(0)}=\lambda\supseteq \mu^{(1)}\supseteq \cdots\supseteq\mu^{(\nu_1)} =\mu$$

 %Let $\mu^{(0)}=\lambda$. For $k=1,\dots,\nu_1$,
 %define $\mu^{(k)}$ the shape obtained  by erasing in row $i$ of $T$ restricted to the shape  $\mu^{(k-1)}$, the rightmost cell filled with $i$ %if any.
This decomposes $\lambda/\mu$ into vertical strips $$J_1:=\mu^{(0)}/\mu^{(1)},J_2:= \mu^{(1)}/\mu^{(2)}, J_3=\mu^{(2)}/\mu$$

The size of the vertical strip $J_k=\mu^{(k-1)}/\mu^{(k)}$ is $Q[k]$,  $|J_k|=|\mu^{(k-1)}/\mu^{(k)}|=|\mu^{(k-1)}|-|\mu^{(k)}|=Q[k]$, $k=1,2,3$.
The vertical strip $J_i$ in $T$ is filled with the word $12\dots \nu_i^t=Q[i]$  and in $Q$ with $\nu_i^t=Q[i]$ $i$'s.
%We write $J_{k}\ge J_{k+1}$ to mean
One has $Q[k]\ge Q[k+1]$,
 condition $(R5)$ is verified $$Q_{\le r}[k + 1] \le Q_{\le r}[k], \mbox{ for $r, k>0$}.$$
%We write $J_{k}\ge_{sp} J_{k+1}$ if $J_k\ge J_{k+1}$ and in
as well as  condition $(R4)$ is verified.
The  filling of $J_3$, $J_2$ and $J_1$ in $T$  are respectively the column words $12$,$12$ and $12345$.
 Furthermore concatenating the reverse column words of the strings $J_3,J_2, J_1$ in $T$ gives the Yamanouchi tableau
 $$Y(\nu)=\YT{0.15in}{}{
 {{1},1,1},
 {{2},2,2},
 {{3}},
 {{4}},
}
$$
\end{ex}
\noindent the rectification of the word of $T$.

\subsection{The restriction of the LR orthogonal transpose linear time map to LR-Sundaram tableaux}\label{sec:lozenge}
In \cite{azkoma25} a complete account has been provided on the LR-symmetries under the action of the dihedral group   $\mathbb Z_2\times { D}_3$. In particular, the LR-symmetry maps are exhibited  on Littlewood-Richardson  tableaux as well as on the companion pairs. The LR symmetry maps restricted to  LR-Sundaram tableaux  do not  need to return  LR-Sundaram tableaux for the same $n$ but its characterization  play an important role on  branching rules from $Gl_{2n}(\C)$ to $Sp_{2n}(\C)$.  Such restriction, in the case of the fundamental symmetry map, $c_{\mu,\nu,\lambda}=c_{\nu,\mu,\lambda}$, has been studied by Kumar-Torres in \cite{sathishtorres} via hives   which in turn reduces to the restriction of right and left companions as discussed  in \cite{az18v5,az26}. The restriction of the orthogonal-transpose  linear map $\blacklozenge$ \cite{acm09,azkoma25} exhibiting $c_{\mu,\nu,\lambda}=c_{\lambda^t,\nu^t,\mu^t}$ for LR tableaux, to LR-Sundaram tableaux has not been studied before but surfaces in the branching rule defined by the quantum Littlewood-Richardson map.
It turns out that the $Rec_{2n}(\lambda/\mu)=\widetilde Rec_{2n}(\lambda/\mu$ characterizes explicitly the restriction of the LR orthogonal transpose  symmetry map on LR-Sundaram tableaux. For convenience let us consider our partitions $\mu,\nu,\lambda$ inside of a fixed rectangle and define $\vee{\lambda}$ the complement of  the Young diagram of $\lambda$ inside the given rectangle. With this convention $LRS_{2n}(\mu,\nu,\lambda)$ denotes the set of LR-Sundaram tableaux of shape ${\lambda^\vee}/\mu$ and weight $\nu$ and similarly for $\widetilde Rec_{2n}(\mu,\nu^t,\lambda)$. See Example \ref{ex:symmetry}.

 Denote the bijection in Theorem \ref{surjection} by $\lozenge$. Then
 the $LR$ orthogonal transpose symmetry map $\blacklozenge$ \cite{azkoma25} is defined by
\begin{align*}\blacklozenge:& LRS_{2n}(\mu,\nu,\lambda)&\overset{\lozenge}\longrightarrow &\widetilde Rec_{2n}(\mu,\nu^t,\lambda)&\underset{\pi\circ t}\hookrightarrow& LR({\lambda^t},\nu^t,\mu^t)&\\
&T&\mapsto& Q&\mapsto& Q^{\pi\circ t}=\blacklozenge T&
\end{align*}
where $Q^{\pi\circ t}$ means the $\pi$-rotation (transposition) followed with transposition (rotation) of $D(\lambda)$. The extra  condition $(R3)$,  Definition \ref{def:tilderec}, on $Q$ characterizes the LR tableau $\blacklozenge T$, the image of the LRS tableau $T$ by the orthogonal transpose map $\blacklozenge$.  The bijection $\lozenge$ and its inverse is so natural that  we can intertwine LR-Sundaram tableaux  with their $Q$ presentations in $\widetilde Rec_{2n}(\mu,\nu^t,\lambda)$.
For notation and further details we refer the reader to \cite{az99, acm09,azkoma25}.
We resume to Example \ref{ex:reclrs}.

\begin{ex} \label{ex:symmetry}Let $n=3$, $\lambda^\vee=(4,3,2,2,1)$, $\lambda=(3,2,2,1)$ $\mu=(3,1,0)$, $\mu^t=(2,1,1,0)$, $\nu=(3,3,1,1,0^2)$, $\nu^t=(4,2,2)$ and $T\in LRS_6(\mu, \nu,\lambda)$
\begin{align}  T=\YT{0.15in}{}{
 {{},{},{},{1}},
 {{},{1},{2},{}},
 {{1},{2},{},{}},
 {{2},{3},{},{}},
 {{4},{},{},{}},
}\in LRS_6(\mu,\nu,\lambda)\mapsto Q=\YT{0.15in}{}{
 {{},{},{},{1}},
 {{},{2},{1},{}},
 {{3},{2},{},{}},
 {{3},{1},{},{}},
 {{1},{},{},{}},
}\in \widetilde Rec_{6}(\mu,\nu^t,\lambda)\mapsto
Q^{\pi\circ t}=\YT{0.15in}{}{
 {{},{},{},{},{1}},
 {{},{},{},{1},{}},
 {{},{1},{2},2,{}},
 {{1},{3},{3},{},{}},
}
=\blacklozenge T\nonumber\\
\blacklozenge T \in LR({\lambda^t},\nu^t,\mu^t).\nonumber
\end{align}
\end{ex}

The action of the symmetry map $\blacklozenge$ restricted to $ LRS_{2n}(\mu, \nu,\lambda)$ returns tableaux in  $LR({\lambda^t},\nu^t,\mu^t)$ satisfying the following conditions on the parts of  $\nu^t$:  $\nu_k^t \in 2\Z$, and $\nu_k^t\ge 2(\ell(\mu^{(k-1)})-n)$, for all $k\ge 1$,
where $\mu^{(k-1)}$ is the partition such that
$$D(\mu^{(k-1)}) = D(\mu) \cup\{(i, j) \in D(\lambda/\mu) | Q(i, j) \ge  k\}.$$

\section{The containment of the set $\widetilde Rec_{2n}(\lambda/\mu)$ into the set of recording tableaux $Rec_{2n}(\lambda/\mu) $}\label{sec:containmentmain}

\subsection{Symplectic tableaux}
%We fix $\lambda\in Par_{\le 2n}$.
From now on we fix $n\in\mathbb{N}$.
\begin{defi}\cite{king76} \label{def:symp}A semistandard tableau $G \in SST_{2n}(\lambda)$ is said to
be symplectic if
$$G(k, 1) \ge 2k-1, \mbox{  for all $  k \in  [ \ell(\lambda)]$}.$$
Let $SpT_{2n}(\lambda)$ denote the set of all symplectic tableaux of shape $\gamma$ on the alphabet $[2n]$.
\end{defi}

\begin{prop} \label{prop:symplectic1}  %Let  $\gamma$ be a partition with length $\ell(\gamma)\le 2n-1$.
Let $G \in SST_{2n}(\lambda)$.
\begin{enumerate}
\item\cite{watanabe} If $SpT_{2n}(\lambda)\neq\emptyset$ then $\ell(\lambda)\le n$.
\item If $G=(a_1,\dots,a_l)\in SpT_{2n}(\varpi_l)$ then $l\le n$ and
$l\le \lfloor\frac{a_l+1}{2}\rfloor$.
\item \cite{watanabe}If $G$ is not symplectic, then there exists a unique
$i \in[2, 2n]$ such that
\begin{align}G(i, 1) < 2i -1 \mbox{ and } G(k, 1) \ge 2k -1  \mbox{ for all $k \in [1, i -1 ]$}.\label{G1}\end{align}
Moreover, we have
\begin{align}G(i - 1, 1) = 2i - 3=2(i-1)-1  \mbox{  and $G(i, 1) = 2i -2=2(i-1)$ }.\label{G2}\end{align}
\end{enumerate}
\end{prop}

\begin{proof} (2) follows from $(1)$ and since $G$ is symplectic, $a_l\ge 2l-1\Leftrightarrow l\le \frac{a_l+1}{2}\Rightarrow l\le \lfloor\frac{a_l+1}{2}\rfloor$.
\end{proof}

Conditions \eqref{G1} and \eqref{G2} say that the first symplectic fail needs to be checked  in the first column of $G$ among consecutive integer entries consisting of an odd number followed with an even number, and if the fail occurs it happens for the first time in a such even entry.
In other words, if the first column has no consecutive integers consisting of an odd number followed with an even number it is symplectic.

\begin{cor}\label{symplecticcolumn} The following holds:
\begin{enumerate}
\item[(a)] For $0\le l\le n$, $Sp_{2n}(\varpi_l)\neq \emptyset$. Namely,  $Sp_{2n}(())=\{()\}$, and,  for $1\le l\le n$, $G=(1,3, \dots,2l-1)\in Sp_{2n}(\varpi_l)$ or $H=(2,4,\dots, 2l)\in Sp_{2n}(\varpi_l)$. In particular, $Sp_{2n}(\varpi_1)=SST_{2n}(\varpi_1)$.
\item [(b)]  $SpT_{2n}(\lambda)\neq\emptyset$ if and only if  $\ell(\lambda)\le n$.

\item [(c)] If $G'$ is obtained from $G\in  Sp_{2n}(\varpi_l)$ by suppressing
$0\le t_0\le l$ entries then $G'\in Sp_{2n}(\varpi_{l-t_0})$ and the suppressed part $G''\in Sp_{2n}(\varpi_{t_0})$.

\item [(d)] Let $\mu\in Par_{\le n}$. If $G'$ is obtained from $G\in  Sp_{2n}(\mu)$ by reverse Schensted column insertion, then $G'$ is still symplectic
 and $G'\in Sp_{2n}(\mu')$ for some partition $\mu'\subset_{vert}\mu$. That is, $G'=(\underset{\mu/\mu'}\leftarrow G)$.
 \item [(e)]
  Let $\lambda\in Par_\le n$,
 $$SpT_{2n}(\lambda)=\{S\in SST_{2n}(\lambda)| (S(1,1),\cdots,S(\ell(\lambda),1)\in SpT_{2n}(\varpi_{\ell(\lambda)}) \}.$$
 \end{enumerate}
\end{cor}
\begin{proof}  $(a)$
 The $i$th entry of $G$ is $2i-1$, and the $i$th entry of $H$ is $2i>2i-1$, for $i=1,\dots,l$.

 $(b)$ The "only if part" follows from Proposition \ref{prop:symplectic1}.

 The "if part" is a consequence of $(a)$. Let $\ell(\lambda)=l$ and let $T$ with first  column $G$ or $H$ and the remain columns of $T$  added according to the semistandard-ness and with entries not exceeding $2n$. Then $T\in Sp_{2n}(T)$.

 $(c)$ Straightforward from Definition \ref{def:symp}.

 $(d)$ During the reverse bumping route a box is deleted from the bottom of a column. If the column is the first the process stops and returns a symplectic tableau with one less box. If not, the bumped entry proceeds to the column on its left and looks for the largest entry which  smaller or equal and replaces while bumping it. The procedure terminates when an entry is bumped while being replaced by an equal or larger entry. Therefore, the returned tableau $G'$ is symplectic and $G'\in Sp_{2n}(\mu')$ for some partition $\mu'\subset_{vert}\mu$.
 %If the reverse column insertion is iterated, providing the iteration goes from SW to NE,  from Remark \ref{routes},
 %The returned
 %$G'\in Sp_{2n}(\mu')$ for some partition $\mu'\subset_{vert}\mu$.
\end{proof}

\iffalse
\begin{obs} While the column $S$  formed by the bumped elements from a symplectic tableau $G$ by the reverse column insertion is symplectic because $S$ is contained in the first column of $G$, the returned tableau  by this operation does not need to be symplectic. For instance,
$$\YT{0.2in}{}{
 {{1},2},
 {{3}},
}\in Sp_4((2,1))$$
and the reverse column insertion applied to $2$ returns
$$\YT{0.2in}{}{
 {{1}},
 {{2}},
}\notin Sp_4((1,1)).$$

\end{obs}
\fi

\begin{ex} Let $G\in SST_{10}(1^6)$ be the column $(1,3,4,6,8,9)$. For the first symplectic fail we consider the pair $3,4$ where $3=2\times (3-1)-1$ and  symplectic fail occurs in the entry $4=2(3-1)$. Removing the entries $3,4$ it remains the symplectic column $(1,6,8,9)\in Sp_{10}(1^6)$.

 Let $G\in SST_{6}(1^6)$  with column $(1,2,4,5,6)$. The first fail occurs in the entry $2=2(2-1)$.  Removing the entries $1,2$ it remains the symplectic column $(4,5,6)\in Sp_6(1^3)$

 Let $G\in SST_{14}(1^{10})$ be the column $(1,3,4,5,6,7,11,12,13,14)$. The first symplectic fail occurs in the entry $4$. Removing the pair $3,4$ it remains $(1,5,6,7,11,12,13,14)$ still not symplectic, it fails in the entry $14=2(8-1)$ .
\end{ex}

We fix $\lambda\in Par_{\le 2n}$.
The Littlewood–Richardson map of type AII, ${\LRAII}^{AII}$ is the
algorithm  \cite[Section 3.1]{watanabe,watanabe25} which takes $T \in  SST_{2n}(\lambda)$ as input, and returns a pair of tableaux $(P^{II}(T),Q^{II}(T))\in SpT_{2n}(\mu)\times \widetilde Rec_{2n}(\lambda/\mu)$.

Set
\begin{align}\label{rec=rec2}Rec_{2n}(\lambda/\mu) := \{Q^{II}(T) | T \in SST_{2n}(\lambda) \mbox{ such that $sh(P^{II}(T)) = \mu $}\}.
\end{align}

An explicit description of the set $Rec_{2n}(\lambda/\mu) $ is given in [Wat25, Theorem 3.1.4 (2)],
$$Rec_{2n}(\lambda/\mu)=\widetilde Rec_{2n}(\lambda/\mu)
$$
by providing a combinatorial proof for the inclusion $Rec_{2n}(\lambda/\mu)\subset \widetilde Rec_{2n}(\lambda/\mu)$ while the  inclusion
$Rec_{2n}(\lambda/\mu)\supset \widetilde Rec_{2n}(\lambda/\mu)$ is concluded via representation theory.
In turn it is shown in \cite{watanabe} that $$\widetilde Rec_{2n}(\lambda/\mu)\overset{\sim}\rightarrow LRS_{2n}$$
whose combinatorial proof for the surjectivity we have given in Theorem \ref{surjection}. Our goal now is to prove combinatorially the inclusion
$Rec_{2n}(\lambda/\mu)\supseteq \widetilde Rec_{2n}(\lambda/\mu)$. This is equivalent to provide an algorithm to compute ${{\LRAII}^{AII}}^{-1}$.

\subsection{ Surjectivity of the reduction map and  its inverse} For the reader convenience this section recalls several properties of the removal and reduction maps in \cite{watanabe}.
We fix $l \in [0, 2n]$ and $\mathbf{a} = (a_1, \dots , a_l)$ a column  in $SST_{2n}(\varpi_l)$.
We often regard $\mathbf{a}$ as a set.

The \emph{removal} \emph{subword }of $\mathbf{a}$ \cite{watanabe} is defined to be the subword $\mathrm{rem}(\mathbf{a})$ of  $\mathbf{a}$ obtained by the following recursive formula:

\begin{align}\label{removal}
\mathrm{rem}(\mathbf{a}) :=
\begin{cases}
\emptyset,& \mbox{ if } l\le 1,\\
\mathrm{rem}(a_l,\dots, a_{l-2})  (a_{l-1},a_l),& \mbox{ if } l\ge 2, a_l\in   2\mathbb{Z}, a_{l-1} = a_l-1,\\
 &\mbox{ and }
a_l < 2l -|\rem(a_1,\dots,a_{l-2})|-1,\\
\rem(a_1,\dots,a_{l-1})& \mbox{ otherwise}.
\end{cases}
\end{align}

\begin{defi}\cite{watanabe} \label{def:rem} For the column
$\mathbf{a}$, the new column $\mathrm{red}(\mathbf{a})$, reduction of $\mathbf{a}$, is defined to be the one obtained from $\mathbf{a}$ by removing the entries in
the set $\mathrm{rem}(\mathbf{a})$,
$$\redu( \mathbf{a})=\mathbf{a}\setminus \rem(\mathbf{a})\in SST_{2n}(\varpi_{l-|\rem(\mathbf{a})|}).$$
\end{defi}
In fact we shall see in Proposition \ref{prop:redu} \cite{watanabe} that $\redu( \mathbf{a})\in SpT_{2n}(\varpi_{l-|\rem(\mathbf{a})|})$ where $l-|\rem(\mathbf{a})|$ satisfy further conditions.

Some useful properties of removable entries in $\mathbf{a}$.

\begin{prop}\label{prop:removcontain}\cite[Proposition 2.5.3]{watanabe}
\begin{enumerate}
\item[(1)] If $ a_l \in \rem(\mbf{a})$, then $a_l\in 2\mathbb{Z}$.

\item[(2)] If $a_l\notin \rem(\mbf{a})$ then $ \rem(\mbf{a})= \rem(a_1,\dots, a_{l-1})$.

\item[(3)] If $a_l$ is odd then $ a_l\notin \rem(a)$, and $\rem(\mbf{a})= \rem(a_1,\dots, a_{l-1})$.

\item[(4)] If $a_l \in \rem(\mbf{a})$, then $a_l\in 2\mathbb{Z}$ and $\rem(a_1,\dots, a_{l-1})= \rem(a_1,\dots, a_{l-2})$.

\item [(5)] $\rem(a_1,\dots, a_k) \subset \rem(\mbf{a})$ for all $k \in[0,l]$.

%\item [(6)] $|rem(a)|\in 2\mathbb{Z}$.
\end{enumerate}
\end{prop}
Points $(2)$ and $(3)$ together in the next proposition are an alternative to  Definition \ref{def:rem} to compute $\rem(\bf a)$ and allowing its computation easily.

For each $x\in\mathbb{Z}$, set \cite{watanabe} \begin{align}\label{s} \s(x)=\begin{cases}
x+1, &  \mbox{ if } x\notin 2\mathbb{Z}, \\
x-1, &  \mbox{ if }  x\in 2\mathbb{Z}.
\end{cases}
\end{align}
Indeed $\s^2(x)=x$, for each $x\in \Z$. We call it the parity swapping involution which has the following properties
\begin{align}\label{minmax}
min(a_i,\s(a_i))\notin 2\Z,~~min(a_i,\s(a_i))-1\in 2\Z, \\
 max(a_i,\s(a_i))\in 2\Z, ~~ max(a_i,\s(a_i))+1\notin 2\Z.
\end{align}

\begin{prop}\cite[Proposition 4.2.2, 4.2.7, Corollary 4.2.8]{watanabe}\label{prop:rem}
\begin{enumerate}

\item[(1)] If $i\in [1,l]$ and $a_j\notin \rem(\mbf{a})$, for $j\in[i,l]$ then $\rem(\mbf{a})=\rem(a_1,\dots,a_{i-1})$.

\item[(2)] for each $i\in [1,l]$, $a_i\in \rem(\mbf{a})$ if and only if one of the following holds
\begin{enumerate}
\item[(a)] $a_i$ odd, $i<l$, $a_{i+1}=a_i+1$ and $a_i<2i-|\rem(a_1,\dots,a_{i-1})|$

\item[(b)] $a_i$ even, $i>1$, $a_{i}=a_{i-1}+1$ and $a_i<2i-|\rem(a_1,\dots,a_{i-2})|-1$
\end{enumerate}

\item [(3)]  $a_i \in \rem(\mbf{a})$ if and only if $s(a_i) \in \rem(\mbf{a})$.
Consequently, $|\rem(\mbf{a})|\in 2\mathbb{Z}$.

%\item[(4)] $0\le l-|rem(\mbf{a})|\le min(2n-l,l)$. In particular, $|rem(\mbf{a})|\ge 2(l-n)$.

\item[(4)] For $i\in[1,l]$,
\begin{enumerate}
\item[(a)] $a_i$ odd, then $|\rem(a_1,\dots,a_{i-1})|\ge 2i-a_i-1$

\item[(b)] $a_i$ even, then $|\rem(a_1,\dots,a_{i})|\ge 2i-a_i$
\end{enumerate}
\end{enumerate}

\end{prop}

\begin{ex} %$a=(1,3,4,5,6,7,11,12,13,14)$, $rem(a)=(3,4,5,6,13,14)$ and  $red(a)=(1,7,11,12)$

Let $n=4$,

 $l=6$, $\rem(123456)=(123456)$,  $\redu(123456)=()$.

{ $l=4$, $\rem(1256)=(1,2)$ since $5>2\times 3-2$, $\redu(1256)=(56)$.}

$l=5$, $\rem(4,5,6,7,8)=(7,8)$, $\redu(4,5,6,7,8)=(4,5,6)$.

$l=5$, $\rem(1,2,3,4,8)=(1,2,3,4)$, $\redu(1,2,3,4,8)=(8)$.

$\rem(1245678)=(125678)$, $7<2\times 6-4$, $4=6-2$

$\rem(12478)=(12)$, $7\nless 2\times 4-2=6$
\end{ex}
The next corollary  follows from the definition of removal subword in \eqref{removal} and item $(2)$ in the previous proposition.

\medskip

\begin{cor} \label{empty0} For $l \in [0, 2n]$,

\begin{enumerate}
\item[(1)] $l$ even $\Rightarrow \rem(1,2,\dots,l)=(1,2,\dots,l)=\redu(1,2,\dots,l)=\emptyset$.

\item[(2)] $l$ odd $\Rightarrow \rem(1,2,\dots,l)=\rem(1,2,\dots,l-1)= (1,2,\dots,l-1)\Rightarrow \redu(1,2,\dots,l)=\{l\}$.

\item[(3)] $\rem(\mbf{a})=\mbf{a}$ if and only if $l$ is even and $\mbf{a}=(1,2,\dots,l)$.

\item[(4)] $\redu(\mathbf{a})=\emptyset$ if and only if $l$ is even and $\mathbf{a}=(1,2,\dots,l)$.

\end{enumerate}
\end{cor}
\begin{proof} "Only if part of (3)". For $l=0$, there is nothing to prove, $\mbf{a}=\emptyset$ and $rem(\emptyset)=\emptyset$. If $l=1$, it follows from \eqref{removal}, $rem(a)=\emptyset$. Let $l\ge 2$, and $\mbf{a}=(a_1,\dots,a_l)$. Let us prove
that $a_1,a_2\in rem(a) \Leftrightarrow a_1=1$ and $a_2=2$. From Proposition \ref{prop:rem}, (2),

$$a_1\in \rem(a)\Leftrightarrow 1\le a_1 \mbox{ odd},  a_{2}=a_1+1 \mbox{  and } a_1<2-|\rem(\emptyset)|=2$$

This implies  $a_1=1$ and $a_2=2$

$$a_2=2\in \rem(a)\Leftrightarrow   2=a_{2}=a_1+1=1+1 \mbox{  even and } 2<2\time 2-|\rem(\emptyset)|-1=4-1=3$$
%Therefore $rem(a_1,a_2)=(a_1,a_2)\Leftrightarrow a=(1,2)$

%By induction assume that $a_1,\dots,a_k\in
%Let $l= k\ge 2$, $a=(a_1,\dots,a_k)$ and suppose  the claim true for $k$, that is $rem(a_1,\dots,a_k)=a$ if and only if $k$ is even and $a=(1,2,\dots,k)$. Let us prove for $k+1$ and $k+2$ and consider $a=(a_1,\dots,a_k,a_{k+1},a_{k+2})$.

By induction assume $a_1,\dots, a_k\in \rem(a_1,\dots,a_k,\dots, a_{l})\Leftrightarrow $ $k$ even and $a_i=i$, for $i=1,\dots, k$. Then from Proposition \ref{prop:rem}, (2), $a_{k+1}\in \rem(a)\Leftrightarrow$ $a_{k+1}$ odd, $k+1<l$, $a_{k+2}=a_{k+1}+1$ and $$k=a_k~even<a_{k+1}~odd <2{k+1}-|\rem(a_1,\dots,a_{k})|=2(k+1)-k=k+2$$
(Note that $a_{k+1}$ even $\Rightarrow a_{k+1}=a_{k}+1$ odd which is absurd.)

Therefore $a_{k+1}=k+1$ odd and $a_{k+2}=k+2$ even. Also $a_{k+2}=k+2\in \rem(a)$ since $a_{k+2}$ even, $k+2>1$, $k+2=a_{k+2}=a_{k+1}+1$ and $k+2=a_{k+2}<2(k+2)-|\rem(a_1,\dots,a_{k})|-1=2(k+2)-k-1=k+3$.
\end{proof}

Some properties of the reduction and successor maps   follow. More precisely, the reduction map on columns is the injective assignment as shown in the next proposition.

\iffalse
\begin{defi}\cite{watanabe} For the column
$\mathbf{a}$, the new column $\mathrm{red}(\mathbf{a})$ is defined to be the one obtained from $\mathbf{a}$ by removing the entries in
the set $\mathrm{rem}(\mathbf{a})$,
$$\redu( \mathbf{a})=\mathbf{a}\setminus \rem(\mathbf{a}).$$
\end{defi}
\fi

\begin{prop} \label{prop:redu}\cite{watanabe,nsw} Let $l\in [0,2n]$ and  $\mathbf{a}=(a_1,\dots,a_l) \in SST_{2n}(\varpi_l)$.
Then
\begin{enumerate}

\item[(1)] $\redu(\mathbf{a})$ is symplectic.

    \item[(2)] $\redu(\mathbf{a})=\mathbf{a}$ if and only if $\mathbf{a}$ is symplectic.

\item[(3)] \cite[Proposition 4.3.6]{watanabe} The reduction map $\redu$ on $SST_{2n}(\varpi_l)$ is an injective assignment

\begin{align}\redu:SST_{2n}(\varpi_l)&\rightarrow\bigsqcup_{\begin{smallmatrix} 0\le t\le min\{l,2n-l\}\\
l-t\in 2\mathbb{Z}
\end{smallmatrix}}SpT_{2n}(\varpi_t)\nonumber\\
\mathbf{a}&\mapsto \redu( \mathbf{a})=\mathbf{a}\setminus \rem(\mathbf{a}).
\label{redumap}
\end{align}
\end{enumerate}
\end{prop}

\begin{defi}\cite{watanabe} Let $S \in SST_{2n}$. Let $S_1$ denote the first
column of $S$, and $S_{\ge 2}$ the rest of the tableau $S$.  The new tableau $\suc(S)$ is defined to be

\begin{align} \mathrm{suc}(S) := \redu(S_1) \bigcdot S_{\ge 2} \mbox{ the Schensted column insertion of  $\mathrm{red}({S_1})$ in $S_{\ge 2}$}.
\end{align}
In addition the \emph{successor} \emph{map} on $SST_{2n}$ is the assignment
\begin{align}SST_{2n}(\lambda)&\rightarrow \bigsqcup_{\begin{smallmatrix}
\mu\in Par_{\le 2n}
\end{smallmatrix}} SST_{2n}(\mu)\\
S&\mapsto  \suc(S)=\mathrm{red}(S_1) \bigcdot S_{\ge 2}
\end{align}
\end{defi}

In particular, if $S$ is a column $\mathrm{suc}(S)=\mathrm{red}(S)$. From \cite[Corollary 4.4.4]{watanabe}  the successor map is also injective.

\begin{prop} \cite[Proposition 4.3.6]{watanabe}, \cite{nsw} Let $T\in SST_{2n}(\lambda)$ and let $\mu$ be the shape of $\suc(T)$. Then
\begin{enumerate}
\item[(a)] $\mu\subset_{vert} \lambda$, $\lambda/\mu$ is a vertical strip.

\item[(b)] $\lambda=\mu$ if and only if $T=\suc(T)$.

\item[(c)] $T$ is symplectic if and only if $\suc(T)=T$.

\item [(d)]The $\mathrm{suc}$ map on $SST_{2n}(\lambda)$ is injective.
\end{enumerate}
\end{prop}

For $T\in SST_{2n}(\lambda)$define $\suc^k(T):=\suc(\suc^{k-1}(T))$, for $k\ge 1$. This sequence stabilizes in a finite number of iterations, that is, there is $N\ge 0$ such that $\suc^{N+1}(T)=\suc^{N}(T)$. In other words,  $\suc^{N}(T)$ is  symplectic, for some $N\ge 0$.

\begin{defi}\label{def:successors}\cite{watanabe} Let $T\in SST_{2n}(\lambda)$ and let $N$ be a nonnegative integer satisfying  $suc^{N+1}(T)=\suc^{N}(T)$.
Set $P^{AII}(T)=\suc^N(T)$ and $Q^{AII}(T)$ to be the tableau that records the process of transformations from $T$ to $P^{AII}(T)$. Let the shapes of $T$, $\suc^1(T)$, $\suc^2(T),\dots,\suc^N(T)$ be $\lambda^0=\lambda\supset_{vert}\lambda^1\supset_{vert}\lambda^2\supset_{vert}\dots\supset_{vert}\lambda^N=sh(P^{AII}(T))$ respectively, and define $ Q^{AII}(T)$ as the tableau obtained by placing the  entry $j$ in the vertical strip $\lambda^{(j-1)}/\lambda^j$ for $j=1,2,\dots,N$
\end{defi}

%For $\lambda\in Par_{\le 2n}$ and $\mu\in Par_{\le n}$, set
%\begin{align}\label{rec=rec3}Rec_{2n}(\lambda/\mu) := \{Q^{AII}(T) | T \in SST_{2n}(\lambda) \mbox{ such that $sh(P^{AII}(T)) = \mu $}\}.
%\end{align}

With the setting in the previous definition, Watanabe has shown combinatorially in \cite{watanabe} that $ Rec_{2n}(\lambda/\mu)\subset \widetilde Rec_{2n}(\lambda/\mu)$ and thus that
\begin{align} \label{lrIIsurj}
{\LRAII^{AII}}:
SST_{2n}(\lambda) &\longrightarrow
\bigsqcup_{\begin{smallmatrix}\mu\in Par_{\le n}\\
\mu\subset\lambda\end{smallmatrix}}SpT_{2n}(\mu) \times \widetilde Rec_{2n}(\lambda/\mu)\nonumber\\
T&\mapsto (P^{AII}(T)=suc^{N}(T),Q^{AII}(T))
\end{align}
is an injection. In particular, $SpT_{2n}(\lambda)\subset SST_{2n}(\lambda)$ and $LR^{AII}(S)=(\mathrm{suc}^1(S),Q(S)=(S,D(\lambda/\lambda))=(S,\emptyset)$ for all $S\in SpT_{2n}(\lambda)$.

\subsubsection{Expanding map}\label{sec:inversereduction}

We now show that the reduction map \eqref{redumap} is also surjective. This section is updated following the recent author preprint \cite{azreduction} where a complete explicit formula for the inverse reduction map is given. For a fixed $n\in\mathbb{N}$, let $l\in [0,2n]$ and $t\in[0,n]$ such that  $0\le t\le min\{l,2n-l\}$ and $l-t\in 2\Z$. Then we define the expanding map
%Let $\mathbf{a} = (a_1, \dots , a_t)\in SpT_{2n}(\varpi_t)$ such that, for every $1\le i< t$,
%$$a_i\notin 2\Z\Rightarrow a_{i+1}>a_i+1~~~( \mbox{ that is,  } a_i\notin 2\mathbb{Z}, ~~~ a_{i+1}\in 2\mathbb{Z}\Rightarrow a_{i+1}-a_i\ge 3).$$ %Then

\begin{align}\label{exp}
\rm{exp}_t:SpT_{2n}(\varpi_t)&\rightarrow SST_{2n}(\varpi_l)\quad\\
\mathbf{a}&\mapsto\rm{exp}(\mathbf{a}):=T,\nonumber
\end{align}
such that $\redu(\rm{exp}(\bf a))=a$. That is, $\rm{red}_{t}^{-1}:=\redu^{-1}_{|SpT_{2n}(\varpi_t)}=\rm{exp}_t$.

%The theorems in this section show that the reduction map, $\redu$, \eqref{redumap}, \cite{watanabe,nsw},  is surjective and explicitly exhibits %its inverse in the cases where the symplectic column decomposes into non empty factors $\bf A_i$ of even length  consisting of consecutive integers  starting with an odd number, and
The next lemma characterizes the symplectic columns when
the symplectic column is such that  consecutive integers may occur only as an even number followed with an odd number. For example,
%for the first case, we mean $ (5678)\in SpT_{10}(\varpi_4)$ or $ (78)( 13, 14, 15, 16)\in SpT_{20}(\varpi_6)$ or $(7,8)(13,14,15,16) (19,20$, and for the second case
$(2,3,6,7)$, or $(1,4,5,7,10,12)$, or $(1,4,7,10,12)$ where consecutive integers do not occur.
%%%%%%%%%%%%%%%%%%%%%%%%%%%%%%%%%%%%%%%%%%%%%%%%%%%%%%%%%%%%%%%%%%%%%%

\begin{lem}\label{lem:swapping}\cite{azreduction}Let $\mathbf{a} = (a_1, \dots , a_t)\in SST_{2n}(\varpi_t)$. The following are pairwise equivalent
%\begin{enumerate}

\begin{enumerate}
\item for every $1\le i\le t$, $\s(a_i)\notin \bf{a}$
\item ${\bf{a}}\cap \s(\bf{a})=\emptyset$
 \item  $[x, y]$ is an interval factor $\subseteq \bf a $ $\Rightarrow$ $x=y$  $ \vee$ ($x\in 2\Z, y=x+1\notin 2\Z$).
 \item for every $1\le i<t$,
$a_i\notin 2\Z\Rightarrow a_{i+1}-a_i\ge 2$ ( $a_{i+1}\ge \s(a)+1$)
%\item for every $1\le i<t$,  $  a_i\notin 2\mathbb{Z}, ~~~ a_{i+1}\in 2\mathbb{Z}\Rightarrow a_{i+1}-a_i\ge 3.$
%\item ${\bf{a}}\cap s(\bf{a})=\emptyset$
%\item %$x\in[2n]\wedge  s(x)\
\item for every $1< i\le t$,
$a_i\in 2\Z\Rightarrow a_{i}-a_{i-1}\ge 2$ ( $\s(a)\ge a_{i-1}+ 1$)
%\end{enumerate}
%\item if ${\bf{a}}\cap s(\bf{a})=\emptyset$, it holds
%\begin{enumerate}
\item for every $1\le i<t$,
\begin{align*}
&(  a_i\notin 2\mathbb{Z}, ~~ a_{i+1}\in 2\mathbb{Z}\Rightarrow a_{i+1}-a_i\ge 3)\\
&\qquad\qquad\wedge\\
&( a_i,~ ~a_{i+1}\notin 2\mathbb{Z} \mbox{ or } a_i, ~~a_{i+1}\in 2\mathbb{Z}\Rightarrow a_{i+1}-a_i\ge 2)\\
 &\qquad\qquad\wedge\\
 &( a_i\in 2\mathbb{Z},~ ~a_{i+1}\notin 2\mathbb{Z}\Rightarrow a_{i+1}-a_i\ge 1).
 \end{align*}
\end{enumerate}
%\end{enumerate}

\end{lem}

\begin{thm}\label{thm:nofactors} Let $l\in [0,2n]$ and $t\in[0,n]$ such that  $0\le t\le min\{l,2n-l\}$ and $l-t\in 2\Z$. Let $\mathbf{a} = (a_1, \dots , a_t)\in SpT_{2n}(\varpi_t)$ such that, for every $1\le i\le t$, $\s(a_i)\notin \bf{a}$. Then,
%$$a_i\notin 2\Z\Rightarrow a_{i+1}>a_i+1\Leftrightarrow a_i\notin 2\Z~~~(\mbox{ that is,  } a_i\notin 2\mathbb{Z}, ~~~ a_{i+1}\in %2\mathbb{Z}\Rightarrow a_{i+1}-a_i\ge 3),$$ and

\begin{align}\label{redu:nofactors}
\rm{red}_{t}^{-1}:=\redu^{-1}_{|SpT_{2n}(\varpi_t)}:SpT_{2n}(\varpi_t)&\rightarrow SST_{2n}(\varpi_l)\quad\\
\mathbf{a}&\mapsto\rm{red}_{t}^{-1}(\mathbf{a}):=T_0(a_1)T_1\cdots(a_t) T_t,\nonumber
\end{align}
where \begin{align}\label{inverseredu:nofactors}
&T_0=(1,2,\dots l_1),\\
&(a_i)T_i=
\begin{cases}
 ({a}_i)( a_i+1,\dots, a_i+l_{i+1}),
  &
 \mbox{if } a_i\in 2\mathbb{Z}\\
 ({a}_i)( a_i+1+1,a_i+1+2,\dots, a_i+1+l_{i+1}),
 &\mbox{if } a_i\notin 2\mathbb{Z}\\
\end{cases},& 1\le i\le t, \label{inverseredu:nofactors2}
\end{align}
and
$l_1,\dots, l_t$, $l_{t+1}$ are nonnegative even numbers defined by
\begin{align}\label{l:nofactors}
&l_1=
\begin{cases} min\{a_1-2, l-t\},& \mbox{if } a_1\in 2\mathbb{Z}\\
min\{ a_1-1, l-t\},&\mbox{if } a_1\notin 2\mathbb{Z},\\
\end{cases}
\end{align}
\begin{align}
&l_{i+1}=
\begin{cases}min\{a_{i+1}-a_i-2, l-t-\sum_{k=1}^i l_k\}, &\mbox{if }
a_i, a_{i+1}\in 2\mathbb{Z} \mbox{ or } a_i, a_{i+1}\notin 2\mathbb{Z}
\\
min\{{ a_{i+1}-a_i-1},~~ l-t-\sum_{k=1}^i l_k\},&\mbox{if }
a_i\in 2\mathbb{Z}, a_{i+1}\notin 2\mathbb{Z} \\
min\{a_{i+1}-a_i-3,  l-t-\sum_{k=1}^i l_k\},&\mbox{if }
 a_i\notin 2\mathbb{Z}, ~~~ a_{i+1}\in 2\mathbb{Z},
\end{cases},\quad \mbox{ for } 1\le i<t,\label{ells}
\end{align}
and $l_{t+1}=l-t-\sum_{i=1}^{t}l_i.$
\end{thm}

%%%%%%%%%%%%%%%%%%%%%%%%%%%%%%%%%%%%%%%%%%%%555CHANGE%%%%%%%%%%%%%%%%%%%%%%%%%%%%%%%%%%%%%%%%%%%%%%%%%%%%%%%%%%%%%%%%%%%%%%%%%%%%%%%
\begin{lem}\label{lem:reductionsurj} Let $l\in [0,2n]$ and $t\in[n]\cap 2\Z$ such that  $2\le t\le min\{l,2n-l\}$ and $l-t\in 2\Z$. Let $a_1\notin 2\Z$,  and   $\mathbf{a} = (a_1, a_1+1,\dots , a_1+t-2,a_1+t-1)\in SpT_{2n}(\varpi_t)$. Then
  $a_1\ge t+1$,  and
\begin{align}\label{lem:inverseredx}
\rm{red}_{t}^{-1}:=\redu^{-1}_{|SpT_{2n}(\varpi_t)}:SpT_{2n}(\varpi_t)&\rightarrow SST_{2n}(\varpi_l)\quad\\
\mathbf{a}&\mapsto\rm{red}_{t}^{-1}(\mathbf{a})=T,\nonumber
\end{align}
where \begin{align}T=(1,2,\dots, l_1)\mathbf{a}(a_1+t-1+1,\dots, a_1+t-1+l_{t+1})\label{lem:T},
\end{align}
and $l_1$, $l_{t+1}$  are nonnegative even numbers defined by
\begin{align}\label{lem0}
l_1=min\{a_1-t-1,l-t\},\quad
l_{t+1}=l-t-l_1.
\end{align}
\end{lem}

\begin{proof} If ${\bf a}\in  SpT_{2n}(\varpi_t)$ then $a_1+t-1\ge 2t-1\Leftrightarrow a_1\ge t\Leftrightarrow a_1\ge t+1$ since $a_1\notin 2\Z$ and $t\in 2\Z$.
This condition guarantees  %column $\bf a$ is in fact symplectic.
that,  for $0\le j\le t-1$,
$$a_1+j\ge t+1+j\ge j+2+j=2(j+1)>2(j+1)-1.$$
The entries of column $\bf a$ are bounded by $2n$.  In the worst case when $t=n$, one has  for $a_1=t+1$, $a_1+t-1=t+1+t-1=2t= 2n$. Henceforth in the worst case,  for $a_1=t+1$, ${\bf a}\in SpT_{2n}(\varpi_t)$.

We now show that $T\in SST_{2n}(\varpi_l)$. Since $0\le l_1\le l-t$, it follows that $l_{t+1}=l-t-l_1\ge l-t-(l-t)\ge 0$ is a non negative even number such that $l_{t+1}=0$ when $l_1=l-t$. Therefore $\ell (T)=l_1+t+l_{t+1}=l$, and, in addition the entries are bounded by $2n$. In fact, if $l_1=a_1-t-1<l-t$, one has
 $$a_1+t-1+l_{t+1}=a_1+t-1+l-t-(a_1-t-1)=l+t\le 2n.$$

We next  show that $\redu(T)=\bf a$.
Note $a_1+t-2\notin 2\Z, a_1+t-1\in 2\Z$ and
\begin{align}&T=(1,2,\dots, a_1-t-1)\mathbf{a}(a_1+t-1+1,\dots, a_1+t-1+l_{t+1})\nonumber\\
=&(1,2,\dots, a_1-t-1)(a_1,a_1+1,\dots, a_1+t-2)(a_1+t-1,a_1+t-1+1,\dots, a_1+t-1+l_{t+1}).
\end{align}

From Corollary \ref{empty0}, one has $\rem(1,2,\dots, a_1-t-1)=(1,2,\dots, a_1-t-1)$. From Proposition \ref{prop:removcontain}, $(5)$,and Proposition \ref{prop:rem}, one has $$\rem(1,2,\dots, a_1-t-1)(a_1, a_1+1,\dots , a_1+t-2,a_1+t-1)=(1,2,\dots, a_1-t-1)$$
because for $0\le j\le t-2$ and $j\in 2\Z$,
\begin{align*}a_1+j\nless 2(j+1+a_1-t-1)-(a_1-t-1)-j=a_1+j-t.
\end{align*}

Then, when $l_{t+1}=l-t-l_1>0\Leftrightarrow l_1=a_1-t-1<l-t$, one has
\begin{align*}&\rem((1,2,\dots, a_1-t-1)(a_1,a_1+1,\dots, a_1+t-2)(a_1+t-1,a_1+t-1+1,\dots, a_1+t-1+l_{t+1}))\\
&=(1,2,\dots, a_1-t-1)\cup\rem(a_1+t-1+1,\dots, a_1+t-1+l_{t+1}) \\
&=(1,2,\dots, a_1-t-1)(a_1+t-1+1,\dots, a_1+t-1+l_{t+1}).
\end{align*}
In fact, for $1\le j\le l_{t+1}$ and $j\notin 2\Z$,
\begin{align*}a_1+t-1+j< 2(a_1-1+j)-(a_1-t-1)-(j-1)=2a_1-2+2j-a_1+t+1-j+1=a_1+t+j.
\end{align*}
Hence $\redu(T)=\bf a$.
\end{proof}

\begin{ex} Let $n=7$ and $t=4\in[6]\cap 2\Z$
\begin{enumerate}
\item $0\le l=8\le 14$ such that  $2\le 4\le min\{l=8,14-8\}$ and $l-t=8-4\in 2\Z$. Let $a_1=7\notin 2\Z$, and $a_1=7\ge t+1=5$ such that   $\mathbf{a} = (7, 8, 9,10)$ is  a column  in $SpT_{12}(\varpi_4)$. Then $l_1=a_1-4-1=2$, $l_5=8-4-2=2$, and
    $$T=(1,2)(7,8,9,10)(11,12)\in SST_{12}(\varpi_8)$$ and $\redu(T)=\mathbf{a}$.
    \item $0\le l=10\le 14$ such that  $2\le 4\le min\{l=10,14-10\}$ and $l-t=10-4\in 2\Z$, and $a_1=7\ge t+1=5$ such that   $\mathbf{a} = (7, 8, 9,10)$ is  a column  in $SpT_{14}(\varpi_4)$. Then $l_1=a_1-4-1=2$, $l_5=10-4-2=4$, and $$T=(1,2)(7,8,9,10)(11,12,13,14)\in SST_{14}(\varpi_8)$$ and $\redu(T)=\mathbf{a}$. Note $13<2\times  9-4=14$.
        \end{enumerate}
\end{ex}

\begin{thm}\cite{azreduction}\label{lem:reductionsurj2} Let $l\in [0,2n]$ and   $ t_i\neq 0\in 2\Z$, $i=1,\dots,k$, such that $t=t_1+\cdots+t_k\in[n]$,
$$2\le t=\sum_{i=1}^kt_i\le min\{l,2n-l\}$$ and $l-t\in 2\Z$.
Let
\begin{align*}\mathbf{a} &=\bigoplus_{i=1}^k{\bf A_i}\in SpT_{2n}(\varpi_{t})
\end{align*}
where for $i=1,\dots,k$, ${a_i}\notin 2\Z$ and
\begin{align*}&{\bf A_i}=({a_i}, a_i+1,\dots , a_i+t_i-2,a_i+t_i-1).
\end{align*} Then $a_i\ge 2(t_1+\cdots+t_{i-1})+t_i+1$, $1\le i\le k$, and

\begin{align*}%\label{thm: inversefactors}
\rm{red}_{t}^{-1}:=\redu^{-1}_{|SpT_{2n}(\varpi_t)}:SpT_{2n}(\varpi_t)&\rightarrow SST_{2n}(\varpi_l)\quad\\
\mathbf{a}&\mapsto\rm{red}_{t}^{-1}(\mathbf{a}):=T,\nonumber
\end{align*}
where \begin{align*}T&=(1,\dots, l_1){\bf {A_1}},( a_1+ 1+t_1-1,\dots,  a_1+ 1+t_1-1+l_2){\bf{A_2}}
( a_2+ 1+t_2-1,\dots,  a_2+ 1+t_2-1+l_3)\\
&\cdots{\bf A_{k-1}}( a_{k-1}+ 1+t_{k-1}-1,\dots,  a_{k-1}+ 1+t_{k-1}-1+l_{k})\\
&{\bf {A_k}}( {a_k}+ 1+t_k-1,\dots,  a_k+ 1+t_k-1+l_{t+1})
\end{align*}
with
$l_1,\dots, l_k$, $l_{t+1}$  nonnegative even numbers recursively defined for $i=1,\dots,k$, by
\begin{align}%\label{prop0}
&l_i=min\{~\displaystyle (a_j-2\sum_{m=1}^{j-1}t_m-t_{j}-1)-\sum_{m=1}^{i-1}l_m,~ i\le j\le k,~~l-t-\sum_{m=1}^{i-1}l_m\},\\
& \mbox{ and}\nonumber\\
%&l_1=min\{a_j-2\sum_{i=1}^{j-1}t_i-t_j-1, ~1\le j\le k,~~l-t\},~~\\
%&l_2=min\{~\displaystyle (a_j-2\sum_{i=1}^{j-1}t_i-t_{j}-1)-l_1,~ 2\le j\le k,~~l-t-l_1\}\\
%&\vdots\nonumber\\
%&l_k=min\{\displaystyle a_k-2\sum_{i=1}^{k-1}t_i-t_k-1-(l_1+\cdots +l_{k-1}),~~  l-t-\sum_{i=1}^{k-1}l_i \}\\
%&l_j=min\{a_j-a_{j-1}-t_{j-1}-2,l-t-\sum_{i=1}^{j-1}l_i\}, ~j=2,\dots,k,\quad\\
&l_{t+1}=l-t-\sum_{i=1}^{k}l_i.
\end{align}
\end{thm}

\begin{thm}\cite{azreduction}\label{thm:xAy} Let $l\in [0,2n]$ and $t_1+p_0+p_1\in[n]$, $t_1\in 2\Z$ such that  $2\le t_1+p_0+p_1\le min\{l,2n-l\}$ and $l-(t_1+p_0+p_1)\in 2\Z$. Let
$$\mathbf{a} = x_1x_2\cdots x_{p_0} {\bf A_1}y_1\cdots y_{p_1}\in SpT_{2n}(\varpi_{t_1+p_0+p_1})$$
 where $a_1\notin 2\Z$, ${\bf A_1}=(a_1, a_1+1,\dots , a_1+t_1-2,a_1+t_1-1)$ and
 $$\s(x_1x_2\cdots x_{p_0}y_1\cdots y_{p_1})\cap\mathbf{a}=\emptyset .$$ Then
  $a_1\ge t_1+2p_0+1$,  and
\begin{align}%\label{lem:inverseredx}
\rm{red}_{t_1+p_0+p_1}^{-1}:=\redu^{-1}_{|SpT_{2n}(\varpi_{t_1+p_0+p_1})}:SpT_{2n}(\varpi_{t_1+p_0+p_1})&\rightarrow SST_{2n}(\varpi_l)\nonumber\\
\mathbf{a}&\mapsto\rm{red}_{t_1+p_0+p_1}^{-1}(\mathbf{a})=T,\nonumber
\end{align}
where \begin{align}T=T_0(x_1)T_1\cdots  (x_{p_0})T_{p_0}\mathbf{A_1}Y_0(y_1)Y_1\cdots (y_{p_1})Y_{p_1}%\label{lem:T}.
\end{align}
%with $\tilde l_1$ %, $\tilde l_{t+r+1}$  are
%the nonnegative even number defined by
%\begin{align}%\label{lem0}
%\tilde l_1=min\{a_1-t_1-2p_0-1,l-t_1-p_0-p_1\},\quad
%\tilde l_{t+r+1}=l-t-r-p-( l_1+\cdots+l_{r+1}).
%\end{align}
is such that
\begin{enumerate}
\item  $\tilde l_1$ %, $\tilde l_{t+r+1}$  are
is the nonnegative even number defined by
\begin{align}%\label{lem0}
\tilde l_1=min\{a_1-t_1-2p_0-1,l-t_1-p_0-p_1\},\quad
%\tilde l_{t+r+1}=l-t-r-p-( l_1+\cdots+l_{r+1}).
\end{align}

\begin{align}\label{inverseredu:nofactors+++}
&T_0=(1,2,\dots l_1),\\
&({x}_j)T_j=
\begin{cases}
 ({x}_j)( x_j+1,\dots, x_j+l_{j+1}),
  &
 \mbox{if } x_j\in 2\mathbb{Z}\\
 ({x}_j)( x_j+1+1,x_j+1+2,\dots, x_j+1+l_{j+1}),
 &\mbox{if } x_j\notin 2\mathbb{Z}\\
\end{cases},& 1\le j\le p_0, \label{inverseredu:nofactors2++++}
\end{align}
with
$l_1,\dots, l_{p_0}$, $l_{p_0+1}$  nonnegative even numbers defined by
\begin{align}\label{l:nofactors++++}
&l_1=
\begin{cases} min\{x_1-2, \tilde l_1\},& \mbox{if } x_1\in 2\mathbb{Z}\\
min\{ x_1-1, \tilde l_1\},&\mbox{if } x_1\notin 2\mathbb{Z},\\
\end{cases}
\end{align}
\begin{align}
&l_{i+1}=
\begin{cases}min\{x_{j+1}-x_j-2, \tilde l_1-\sum_{m=1}^j l_m\}, &\mbox{if }
x_j, x_{j+1}\in 2\mathbb{Z} \mbox{ or } x_j, x_{j+1}\notin 2\mathbb{Z}
\\
min\{{ x_{j+1}-x_j-1},~~ \tilde l_1-\sum_{m=1}^j l_m\},&\mbox{if }
x_j\in 2\mathbb{Z}, x_{j+1}\notin 2\mathbb{Z} \\
min\{x_{j+1}-x_j-3,  \tilde l_1-\sum_{m=1}^j l_m\},&\mbox{if }
 x_j\notin 2\mathbb{Z}, ~~~ x_{j+1}\in 2\mathbb{Z},
\end{cases},\quad \mbox{ for } 1\le j<p_0,\label{ells+++}
\end{align}
and \begin{align}
 l_{p_0+1}=\tilde l_1-\sum_{m=1}^{p_0}l_m;\end{align}

\item put $y_0:=a_1+t_1-1\in 2\Z$, and define
\begin{align}\tilde l_{2}:=l-t_1-p_0-p_1- \sum_{m=1}^{p_0+1}l_m,
\end{align}

\begin{align}&Y_0=(y_0+1,y_0+2,\dots, y_0+ l_{ t_1+p_0+1}),\\
&({y}_j)Y_j=
\begin{cases}
 ({y}_j)( y_j+1,\dots, y_j+ l_{t_1+p_0+j+1}),
  &
 \mbox{if } y_j\in 2\mathbb{Z}\\
 ({y}_j)( y_j+1+1,y_j+1+2,\dots, y_j+1+ l_{t_1+p_0+j+1}),
 &\mbox{if } y_j\notin 2\mathbb{Z}\\
\end{cases},& 1\le j\le p_1, \label{inverseredu:nofactors2+++}
\end{align}
with
$ l_{t_1+p_0+j+1}$, $0\le j\le p_1$,
 nonnegative even numbers defined by
\begin{align}\label{l:nofactors+++}
& l_{t_1+p_0+1}=
\begin{cases} min\{y_1-y_0-2, \tilde l_2\},& \mbox{if } y_1\in 2\mathbb{Z}\\
min\{ y_1-y_0-1, \tilde l_2\},&\mbox{if } y_1\notin 2\mathbb{Z},\\
\end{cases}
\end{align}
\begin{align}
& l_{t_1+p_0+j+1}=
\begin{cases}min\{y_{j+1}-y_j-2, \tilde l_2-\sum_{m=1}^j  l_{t_1+p_0+m}\}, &\mbox{if }
y_j, y_{j+1}\in 2\mathbb{Z} \mbox{ or } y_, y_{j+1}\notin 2\mathbb{Z}
\\
min\{{ y_{j+1}-y_j-1},~~ \tilde l_2-\sum_{m=1}^j  l_{t_1+p_0+m}\},&\mbox{if }
y_j\in 2\mathbb{Z}, y_{j+1}\notin 2\mathbb{Z} \\
min\{y_{j+1}-y_j-3,  \tilde l_2-\sum_{m=1}^j  l_{t_1+p_0+m}\},&\mbox{if }
 y_j\notin 2\mathbb{Z}, ~~~ ,y_{j+1}\in 2\mathbb{Z},
\end{cases},\quad \mbox{ for } 1\le j<p_1,\label{ells++++}
\end{align}
and

\begin{align} l_{t_1+p_0+p_1+1}=\tilde l_2-\sum_{m=1}^{p_1}l_{t_1+p_0+m}.\end{align}
 \qquad\qquad\qquad\qquad\qquad\qquad\qquad\qquad\qquad\qquad\qquad\qquad\qquad\qquad\qquad\qquad\qquad\qquad\qquad
 \qquad$\Box$
\end{enumerate}
\end{thm}

A complete explicit formula generalizing the previous theorem is given in \cite{azreduction}.
\begin{ex}\begin{enumerate}
\item Let $n=12$, $l=12$, $t=6\le min\{12, 24-12\},$ $l-t=6$, and ${\bf a}=(9,10)(17, 18, 19, 20)\in SpT_{24}(\varpi_6)$, $t_1=2<t_2=4$, $\tilde l_1=a_2-t_2-t_1-5=17-6-5=6$. Then $l_1=\tilde l_1=6\le a_1-3=6$ and $l_2=0$, $T=(123456)(9,10)(17,18,19,20)\in SST_{24}(\varpi_{12})$.
    \item For ${\bf a}=(7,8)(17, 18, 19, 20)\in SpT_{24}(\varpi_6)$, one has $\tilde l_1=6> a_1-3=4$, and

     $T=(1234)(76)(9,10)(17,18,19,20).$
\end{enumerate}
\end{ex}

\begin{ex} Let $n=12$.
\begin{enumerate}

\item Let $t=6\le n$, $l=12$, $6\le min\{12, 2\times 2\times 12-12\}$, $l-t=6$ and $t_1=4>t_2=2$ and  $t_1=4>t_2=2$. Let  $${\bf a}=(9,10,11,12)(17,18)\in SpT_{24}(\varpi_6),$$ $a_2-a_1-t_1-2=17-9-4-2=2$, $a_1-t_1-1=4$. Then $$T=(1,2,3,4)(9,10,11,12)(13,14)(17,18)\in SST_{24}(\varpi_{12})$$ and $\redu(T)=\bf a$. Note $13<2\times 9-4=14$, and $17\nless 2\times 11-6=16$.

\item  Let $t=8$, $t_1=4=t_2=4$, $l=12$,  $8\le min\{12, 2\times 2\times 12-12\}$, $l-t=4$. Let  $${\bf a}=(9,10,11,12)(17,18, 19, 20)\in SpT_{24}(\varpi_8),$$ $a_2-a_1-t_1-2=17-9-4-2=2$, $l_1=a_1-t_1-1=4$, and $l_2=min\{2, l-t-l_1=4-4=0\}=0$. Then $$T=(1,2,3,4)(9,10,11,12)(17,18, 19,20)\in SST_{24}(\varpi_{12})$$ and $\redu(T)=\bf a$.
    \end{enumerate}
\end{ex}
%The next two theorems show that the reduction map, $\redu$, \eqref{redumap}, \cite{watanabe,nsw},  is surjective and explicitly exhibits its %inverse in the cases where the symplectic column decomposes into non empty factors $\bf A_i$ of consecutive integers of even length  starting %with an odd number, and when the symplectic column only has  consecutive integers consisting of an even number followed with an odd number.
%\iffalse

%%%%%%%%%%%%%%%%%%%%%%%%%%%%%%%%%%%%%%%%%%%%%%%CHANGE CHANGE%%%%%%%%%%%%%%%%%%%%%%%%%%%%%%%%%%%%%%%%%%%%55555

 \begin{ex} { Let $n=3$.}

  \begin{enumerate}
  \item
  $t=2$, $l=4\le 6$, $2\le min\{4,6-4\}$, $l-t=2$ even. Let $\a=(56)\in SpT_{6 }(\varpi_2)$. Then one has
  \begin{align}\nonumber
\rm{red}_{t}^{-1}:=\redu^{-1}_{|SpT_{6}(\varpi_2)}:SpT_{6}(\varpi_2)&\rightarrow SST_{6}(\varpi_4)\quad\\
(56)&\mapsto\rm{red}_{2}^{-1}(56):=T_0T_1 T_{2}=(12)(5)(6)\nonumber
\end{align}

   Then the inverse reduction with respect $l=4$, is $\redu^{-1}_2(56)=(1256)=T_0T_1T_2\in SST_6(\varpi_4)$ with $T_0=12$, $T_1=5$, $T_2=6$. One has $\redu(1256)=(56)$, $\rem(1256)=(12)$.

\item $t=2$, $l=2\le 6$, $2\le min\{2,6-2\}$, $l-t=0$ even. Let $\a=(56)\in SpT_{6 }(\varpi_2)$. Then the inverse reduction with respect $l=2$, $\redu^{-1}_2(56)=(56)=T_0T_1T_2\in SST_6(\varpi_2)$ with $T_0=()$, $T_1=5=T_2=6$,
 Then one has
  \begin{align}\nonumber
\rm{red}_{t}^{-1}:=\redu^{-1}_{|SpT_{6}(\varpi_2)}:SpT_{6}(\varpi_2)&\rightarrow SST_{6}(\varpi_2)\quad\\
(56)&\mapsto\rm{red}_{2}^{-1}(\mathbf{56}):=T_0T_1 T_{2}=()(5)(6)\nonumber
\end{align}
  One has $\redu(56)=(56)$, $\rem(56)=()$,

\end{enumerate}
\end{ex}
\begin{ex}

  \begin{enumerate}
  \item Let $n=8$.
  $t=5$, $l=11\le 16$, $5\le min\{11,16-11\}$, $l-t=6$ even. Let $  \a=(5,6, 10, 11, 15)\in SpT_{16 }(\varpi_5)$. One has, $l_1=2,$ $l_2=0$,
  $l_3=2$, $l_4=0$, $l_5=2$, $l_6=0$, $l_1+l_2+l_3+l_4+l_5+l_6+t=l=11$
  \begin{align}\nonumber
\rm{red}_{5}^{-1}:=\redu^{-1}_{|SpT_{6}(\varpi_2)}:SpT_{16}(\varpi_{11})&\rightarrow SST_{16}(\varpi_5)\nonumber\\
(56)&\mapsto\rm{red}_{2}^{-1}(5,6, 10, 11, 15):=T_0T_1 T_{2}T_3T_4T_5 \nonumber\\
&T_0T_1 T_{2}T_3T_4T_5=(12).(\mathbf{5}).{(\mathbf{6}.{78})}.({\bf 10}).{({\bf 11}.13.14)}.(\bf 15)\nonumber
\end{align}

   The inverse reduction with respect $l=11$, is $\redu^{-1}_2(56)=(1256)=T_0T_1T_2\in SST_6(\varpi_4)$ with $T_0=12$, $T_1=5$, $T_2=6$. One has $\redu(1256)=(56)$, $\rem(1256)=(12)$.

%\item $a=(4,7)$ $T_0T_1T_2=(12 47

\item Let $n=9$, $\a=(3,4,9,10,11,15,16)\in SpT_{18}(\varpi_7)$,  $l=11$, $t=7\le min\{l=11,18-l=7\}$, $l-t=4\in 2\Z$

 The inverse reduction with respect $l=11$,
 $$\redu^{-1}_7(\a)=({\bf 3},{\bf 4},5,6,{\bf 9}, {\bf 10},{\bf 11}, {\bf 15}, {\bf 16},17,18)\in SST_{18}(\varpi_{11})$$

 \item With $l=9$, $t=7$ in $(2)$, $l-t=2$, one has
 the inverse reduction with respect $l=9$,
 $$\redu^{-1}_7(\a)=({\bf 3},{\bf 4},5,6,{\bf 9}, {\bf 10},{\bf 11}, {\bf 15}, {\bf 16})\in SST_{18}(\varpi_{9})$$

 \item With $l=7=t$ in $(2)$, $l-t=0$, one has
 the inverse reduction with respect $l=7$,
 $$\redu^{-1}_7(\a)=({\bf 3},{\bf 4},{\bf 9}, {\bf 10},{\bf 11}, {\bf 15}, {\bf 16})=\a\in SST_{18}(\varpi_{7})$$
\end{enumerate}
\end{ex}
\iffalse
\begin{cor}Let $l\in [0,2n]$ and $t\in[n]\cap 2\Z$ such that  $2\le t+r\le min\{l,2n-l\}$ and $l-t\in 2\Z$. Let $a_1\notin 2\Z$,  and   $\mathbf{a} = x_1x_2\cdots x_r (a_1, a_1+1,\dots , a_1+t-2,a_1+t-1)\in SpT_{2n}(\varpi_{t+r})$. Then
  $a_1\ge t+1$,  and
\begin{align}%\label{lem:inverseredx}
\rm{red}_{t}^{-1}:=\redu^{-1}_{|SpT_{2n}(\varpi_t)}:SpT_{2n}(\varpi_t)&\rightarrow SST_{2n}(\varpi_l)\quad\\
\mathbf{a}&\mapsto\rm{red}_{t}^{-1}(\mathbf{a})=T,\nonumber
\end{align}
where \begin{align}T=(1,2,\dots, l_1)\mathbf{a}(a_1+t-1+1,\dots, a_1+t-1+l_{t+1})%\label{lem:T}.
\end{align}
and $l_1$, $l_{t+1}$  are nonnegative even numbers defined by
\begin{align}%\label{lem0}
l_1=min\{a_1-t-1,l-t\},\quad
l_{t+1}=l-t-l_1.
\end{align}

\end{cor}
\fi

\subsection{Symplectic columns by reverse Schensted column bumping, the  action of the inverse reduction map and concatenation }\label{subsec:essential}

The purpose of this subsection is to analyse  when the action of the inverse reduce map
on a reverse bumped symplectic column $\b$ from a tableau $S\in SST_{2n}(\mu)$, with $\ell(\mu)\le n$, returns a column $\redu^{-1}(\b)$ whose concatenation with the returned tableau by the reverse column bumping on $S$, $(  \longleftarrow S)$, is a semi-standard tableau. This is   done  below in Proposition \ref{reducbumpinterval1} and for $\b$ a symplectic column in the conditions of %Lemma \ref{lem:reductionsurj} or
Theorem  \ref{thm:nofactors}.

Given the vector $(r_{1}<\dots<r_{m})$ with entries in $[1,n]$, its \emph{index} \emph{vector} $\delta_{(r_{1}<\dots<r_{m})}=(x_s)_{s\in [1,n]}$ is the vector in $\{0,1\}^n$ such that $x_s=1$ if $s\in\{r_1,\dots,r_m\}$ and $0$, otherwise.
More precisely,
let $t\in [n]$ and $S\in SST_{2n}(\mu)$    with $\ell(\mu)=t$.
 %with first column
 %$\mathbf{a} = (a_1, \dots , a_t)$ $\in$ $ SST_{2n}(\varpi_t)$. %such that $\s(\a)\cap \bf{a}=\emptyset$.
Let $ 1\le m\le  t$ and let $r_{t_1}<\dots<r_{t_m}$, with $ \{t_1<\dots<t_m\}\subseteq [1,t]$, be $m$ row coordinates of   $S$ such that
$\mu-\delta_{(r_{t_1},\dots,r_{t_m})}$ is a partition. Let

$$(\b, \underset{ (r_{t_1},\dots,r_{t_m})}\longleftarrow S),$$
where
\begin{align*}\b=(a_{{t_1}}, \dots,a_{{t_m}})\in SpT_{2n}(\varpi_m),
\end{align*}
and  $$(\underset{(r_{t_1},\dots,r_{t_m})}\longleftarrow S)\in SST_{2n}(\mu-\delta_{(r_{t_1},\dots,r_{t_m})}).$$
%has first column $S^1=(a_1, \dots,a_{t_1-1},  x_{t_1},a_{t_1+1},\dots ,a_{t_z-1},x_{t_z},a_{t_z+1},\dots, a_{t})\in SST_{2n}(\varpi_t)$.
%or
%$$S^1=(a_1, \dots,a_{t_1-1},  x_{t_1},a_{t_1+1},\dots ,a_{t_{q}-1},x_{t_{q}},x_{t_{q+1}},\dots, x_{t_z}),  \mbox{ with %$1\le t_1\le t_q\le t_z\le t$.}$$
For $ 1\le m\le t\le l\le 2n$,
such that
 $ m\le min\{l,2n-l\}$ and $l-m\in 2\Z$, we analyze when one has
\begin{align}\rm{red}_{m}^{-1}(\mathbf{b})\bigcdot (\underset{(r_{t_1},\dots,r_{t_m})}\longleftarrow S) \in SST_{2n}(\varpi_l+\mu-\delta_{(r_{t_1},\dots,r_{t_m})}).
\end{align}

Let $u=(u_1,\dots,u_t)\in \Z_+^t$ and $(v_1,\dots,v_{t'})\in \Z_+^{t'}$. We write $u\ge v$ whenever $u_i\ge v_i$ for $i=1,\dots,min\{t,t'\}$.
\begin{lem}\label{lem:symp} If  ${\bf a}=(a_1,\dots,a_t)\in SpT_{2n}(\varpi_t)$ and ${\bf b}=(b_1,\dots,b_{t'})\in SST_{2n}(\varpi_{t'})$, for some $t'\le t$, are such that ${\bf b} \ge {\bf a}$, that is, $b_i\ge a_i$, $1\le i\le t'$,  then ${\bf b}\in SpT_{2n}(\varpi_{t'})$.
\end{lem}
\begin{proof}If ${\bf a}=(a_1,\dots,a_t)\in SpT_{2n}(\varpi_t)$ and ${\bf b} \ge {\bf a}$, then $b_k\ge a_k\ge 2k-1$, for $k=1,\dots,t'$. Hence, ${\bf b}\in SpT_{2n}(\varpi_{t'})$.
\end{proof}
\begin{lem}\label{keybasic} Let $m\in \N$ and  $(a_1< \dots<a_{m-1}< a_{m})\in\Z^m_{\ge 1}$. Then
%$$(1,2,\dots,m-1,m)\le (a_1, \dots,a_{m-1}, a_{m})$$
%and
\begin{itemize}
\item[(a)] $(1,2,\dots,m-1,m)\le (a_1, \dots,a_{m-1}, a_{m})$
and
\item[(b)]$ (1,2,\dots,m-1,l)\bigcdot (a_1, \dots,a_{m-1}, a_{m})\in SST(\varpi_l+\varpi_m)$ is a two-column tableau of length $l$ for any $l\ge m$.
\item[(c)] for $p\ge a_m-1$,
$$ (1,2,\dots,p, p+1)\bigcdot (a_1, \dots,a_{m-1}, a_{m})\in SST(\varpi_{p+1}+\varpi_m),  $$
is a two column tableau of length $p+1\ge m$.
\end{itemize}
\end{lem}
\begin{proof} $(a)$ One has $1\le a_1$. Let $k\ge 1$. Assuming by induction that $(1,2,\dots,k)\le  (a_1, \dots,a_{k})$ it follows that $k+1\le a_k+1\le a_{k+1}$ and $ (1,2,\dots,k+1)\le (a_1, \dots,a_{k+1})$.

$(b)$ Follows from $(a)$.

$(c)$ It follows from $(b)$ since $p\ge a_m-1\ge a_{m-1}\ge m-1\Rightarrow p+1\ge m$.
\end{proof}

%Given the vector $(r_{1}<\dots<r_{m})$ with entries in $[1,n]$, its \emph{index} \emph{vector} $\delta_{(r_{1}<\dots<r_{m})}=(x_s)_{s\in [1,n]}$ %is the vector in $\{0,1\}^n$ such that $x_s=1$ if $s\in\{r_1,\dots,r_m\}$ and $0$, otherwise.

Recall the reverse Schensted column insertion in Section \ref{subsec:insert}, Corollary \ref{symplecticcolumn}, $(d)$, and the inverse reduction map  in   Lemma \ref{lem:reductionsurj}, Theorem  \ref{thm:nofactors} and Theorem \ref{thm:xAy}.
\begin{prop}\label{reducbumpinterval1} % Let $l\in [0,2n]$ and $t\in[n]\cap 2\Z$ such that  $2\le t\le min\{l,2n-l\}$ and %$l-t\in 2\Z$.

Let $k+t\in[n]\cap 2\Z$, for some $ k\ge 0$ and $t\ge 1$,
$a_1\notin 2\Z$,  and   $\mathbf{a} = (a_1, a_1+1,\dots ,a_1+{k-1},a_1+k, a_1+k+1,\dots,a_1+k+t-1)\in SpT_{2n}(\varpi_{k+t})$ with (necessarily) $a_1\ge k+t+1$.
Let
$$(\b, \underset{ (k+1,\dots,k+t)}\longleftarrow \a),$$
where
$$\b=(a_1+k, a_1+k+1,\dots,a_1+k+t-1)  \mbox{ and  }(\underset{(k+1,\dots,k+t)}\longleftarrow \a)=(a_1, a_1+1,\dots ,a_1+{k-1})\in SpT_{2n}(\varpi_{k}).$$
Then
\begin{enumerate}
\item $\b=(a_1+k, a_1+k+1,\dots,a_1+k+t-1)\in SpT_{2n}(\varpi_{t})$.
\item for $k+t\le l\le 2n$, %and $t\in[n]\cap 2\Z$
such that
 $ t\le min\{l,2n-l\}$ and $l-t\in 2\Z$,
 %\begin{enumerate}

\medskip
$(a)$ if $k,t\in[n]\cap 2\Z\Leftrightarrow$ $ a_1+k\notin 2\Z$, $a_1+k+t-1\in 2\Z$, it holds

\begin{align} \label{redub}
%\rm{red}_{t}^{-1}:=\redu^{-1}_{|SpT_{2n}(\varpi_t)}:SpT_{2n}(\varpi_t)&\rightarrow SST_{2n}(\varpi_l)\quad\\
%\mathbf{a}&\mapsto
{\redu}_{t}^{-1}(\mathbf{b})=(1,2,\dots, l_1){\b}(a_1+k-1+t+1,\dots, a_1+k-1+t+l_{t+1})\in SST_{2n}(\varpi_l)
%T,\nonumber
\end{align}
where $b_i=a_1+k+i-1$, $i=1,\dots,t$,
%\begin{align}
%T=(1,2,\dots, l_1)\mathbf{a}(a_1+t-1+1,\dots, a_1+t-1+l_{t+1})\label{lem:T},
%\end{align}
and $l_1$, $l_{t+1}$  are nonnegative even numbers defined by
\begin{align*}
k\le l_1=min\{a_1+k-t-1,l-t\},\quad
l_{t+1}=l-t-l_1,
\end{align*}
and the  concatenation of $\redu_{t}^{-1}(\b) \in SST_{2n}(\varpi_l)$ and $(\underset{(k+1,\dots,k+t)}\longleftarrow \a)\in SpT_{2n}(\varpi_{k})$ %.(a_1, a_1+1,\dots ,a_1+{k-1})$
is a tableau in $ SST_{2n}(\varpi_l+\varpi_k)$, with $l\ge k+t$. That is,
 \begin{align}
 &\redu_{t}^{-1}(\b)\bigcdot (\underset{(k+1,\dots,k+t)}\longleftarrow \a)=\nonumber\\
 &={\redu}_{t}^{-1}(\mathbf{b})\bigcdot (a_1, a_1+1,\dots ,a_1+{k-1})\nonumber\\
 &=(1,2,\dots, l_1,a_1+k, a_1+k+1,\dots,a_1+k+t-1,a_1+k+t-1+1,\dots, a_1+k+t-1+l_{t+1})\bigcdot\nonumber\\
 &\bigcdot(a_1, a_1+1,\dots ,a_1+{k-1})\in SST_{2n}(\varpi_l+\varpi_k),\; l\ge k+t,\; l_1\ge k.
 \end{align}
 In other words the row coordinates of the bumped entries in $\a$ and defining $\b$,  are pushed down or fixed in $\redu_{t}^{-1}(\b)$.
% That is, the  concatenation of $\redu_{t}^{-1}(\b)$ and $(a_1, a_1+1,\dots ,a_1+{k-1})$ is a tableau in $ %SST_{2n}(\varpi_l+\varpi_k)$, with $l\ge k+t$.

{ In particular, when $k=0$, $2\le t\in 2\Z$, then$\b=\a$ and $(\underset{(1,\dots,t)}\longleftarrow \a)=( )\in SpT_{2n}()$.
Moreover,
\begin{align*}&{\redu}_{t}^{-1}(\mathbf{b})=\\
 &=(1,2,\dots, l_1,a_1, a_1+1,\dots,a_1+t-1,a_1+t-1+1,\dots, a_1+t-1+l_{t+1})\in SST_{2n}(\varpi_l),\\
 &\; l=l_1+l_{t+1}+t ,
 \end{align*}
 \noindent with  $l_1=min\{a_1-t-1,l-t\}$, $a_1\ge t+1.$ }

\medskip
$(b)$ if $k,t\in [n]$, and $k,t \notin 2\Z$ $\Leftrightarrow a_1+k\in 2\Z,          \, a_1+k+t-1\in 2\Z$, it holds

 $$\b=(a_1+k)\sqcup (a_1+k+1,\dots,a_1+k+t-1)\in SpT_{2n}(\varpi_{t}), \;a_1+k+1\notin 2\Z$$
 $$=(a_1+k)\sqcup{\overline \b},\;\; \overline \b\in  SpT_{2n}(\varpi_{t-1}).$$

 If $\overline\b=\emptyset\Leftrightarrow t=1$
 \begin{align}\label{b=empty}
{\redu}_{t}^{-1}(\mathbf{b})=(1,2,\dots, l_1)\sqcup (a_1+k)\sqcup (a_1+k+1,\dots, a_1+k+1+l_{2}),
%& \mbox{ if $\overline\b=\emptyset\Leftrightarrow t=1$},\label{eq1}%\\
%(1,2,\dots, l_1)(a_1+k){\overline\b}(a_1+k-1+t+1,\dots, a_1+k-1+t+l_{t+1}),& \mbox{ if $t>1$}.\label{eq2}
\end{align}
%\end{enumerate}
where
\begin{align}%\label{l:nofactors}
&l_1=
%\begin{cases}
min\{a_1+k-2, l-1\}, \;l_2= l-1-l_1, \nonumber
 %\\
%min\{ a_1k-1, l-t\},&\mbox{if } a_1\notin 2\mathbb{Z},\\
%\end{cases}
\end{align}
and \begin{align}&{\redu}_{t}^{-1}(\mathbf{b})\bigcdot(\underset{(k+1,\dots,k+t)}\longleftarrow \a)=\nonumber\\
 &{\redu}_{t}^{-1}(\mathbf{b})\bigcdot (a_1, a_1+1,\dots ,a_1+{k-1})=\nonumber\\
 &=(1,2,\dots, l_1,a_1+k,a_1+k+1,\dots, a_1+k+1+l-1-l_1)\bigcdot (a_1, a_1+1,\dots ,a_1+{k-1})\\
 &\in SST_{2n}(\varpi_l+\varpi_k),\; l\ge k+1,\, l_1\ge k.\nonumber
\end{align}

Otherwise, $\overline\b\neq\emptyset\Leftrightarrow 2\le t-1\in 2\Z$

\begin{align} \redu_{t}^{-1}(\mathbf{b})=(1,2,\dots, l_1)\sqcup(a_1+k){\overline\b}\sqcup(a_1+k-1+t+1,\dots, a_1+k-1+t+l-t-1-l_1),\;
l_1\ge k
%& \mbox{ if $t>1$}.
\label{notempty}
\end{align}
and \begin{align}&{\redu}_{t}^{-1}(\mathbf{b})\bigcdot(\underset{(k+1,\dots,k+t)}\longleftarrow \a)=\nonumber\\
&={\redu}_{t}^{-1}(\mathbf{b})\bigcdot (a_1, a_1+1,\dots ,a_1+{k-1})\in  SST_{2n}(\varpi_l+\varpi_k),\; l\ge k+t,\, l_1\ge k.
\end{align}
\end{enumerate}
\end{prop}
\begin{proof}
$(1)$ Follows from Lemma \ref{reducbumpinterval1}.

$(2)$ Note $l-t\ge k+t-t\ge k$ and $a_1+k-t-1\ge k+t+1+k-t-1=2k$.

$(a)$. From Lemma \ref{lem:reductionsurj}, ${\redu}_{t}^{-1}(\mathbf{b})$ is given by \eqref{redub}. Therefore $l_1\ge k$. This means that the row coordinates of the entries of the bumped column $\b$ as entries of $\a$ have been pushed down or kept in $\redu_{t}^{-1}(\b)$, equivalently  $(1,2,\dots, k)\le (a_1, a_1+1,\dots ,a_1+{k-1})$.

Since $\ell(\redu_{t}^{-1}(\b))=l$, the concatenation of $\redu_{t}^{-1}(\b)$ and $(a_1, a_1+1,\dots ,a_1+{k-1})$ is a tableau in $ SST_{2n}(\varpi_l+\varpi_k)$, with $l\ge k+t$.

$(b)$ If $\overline\b=\emptyset\Leftrightarrow t=1$, \eqref{b=empty} is given by Theorem \ref{thm:nofactors}.

If $\overline\b\neq\emptyset\Leftrightarrow 2\le t-1\in 2\Z$, \eqref{notempty} is given by
 and Theorem \ref{thm:xAy} in the form $(a_1+k)\b$.
\end{proof}

\begin{thm} \label{thm:essentialss}Let $t\in [n]$ and $S\in SST_{2n}(\mu)$    %with $\ell(\mu)=t$.
 with first column
 $\mathbf{a} = (a_1, \dots , a_t)$ $\in$ $ SST_{2n}(\varpi_t)$. %such that $\s(\a)\cap \bf{a}=\emptyset$.
Let $ 1\le z\le  t$ and let $r_{t_1}<\dots<r_{t_z}$ with $ \{t_1<\dots<t_z\}\subseteq [1,t]$ be $z$ row coordinates of   $S$ such that
$\mu-\delta_{(r_{t_1},\dots,r_{t_z})}$ is a partition. Let

$$(\b, \underset{ (r_{t_1},\dots,r_{t_z})}\longleftarrow S),$$
where
\begin{align*}\b=(a_{{t_1}}, \dots,a_{{t_z}})\in SpT_{2n}(\varpi_z)  \mbox{ is such that $\s(\b)\cap \bf{b}=\emptyset$},
\end{align*}
and  $$(\underset{(r_{t_1},\dots,r_{t_z})}\longleftarrow S)\in SST_{2n}(\mu-\delta_{(r_{t_1},\dots,r_{t_z})})$$
has first column $S^1=(a_1, \dots,a_{t_1-1},  x_{t_1},a_{t_1+1},\dots ,a_{t_z-1},x_{t_z},a_{t_z+1},\dots, a_{t})\in SST_{2n}(\varpi_t)$.
%or
%$$S^1=(a_1, \dots,a_{t_1-1},  x_{t_1},a_{t_1+1},\dots ,a_{t_{q}-1},x_{t_{q}},x_{t_{q+1}},\dots, x_{t_z}),  \mbox{ with %$1\le t_1\le t_q\le t_z\le t$.}$$
Then
\begin{enumerate}
\item %$\b=(a_{{t_1}}, \dots,a_{{t_z}})\in SpT_{2n}(\varpi_z)$ is such that $\b\cap \s(\b)=\emptyset$,
$\b=(a_{t_1}, \dots,a_{t_z})\le (x_{t_1},\dots, x_{t_z}).$
%and in either cases of $S^1$,
   %$$S^1=(a_1, \dots,a_{t_1-1},  x_{t_1},a_{t_1+1},\dots ,a_{t_z-1},x_{t_z},a_{t_z+1},\dots, a_{t})\in %SpT_{2n}(\varpi_{t}),\;\oo 1\le t_z\le t$$
    %and $$S^1=(a_1, \dots,a_{t_1-1},  x_{t_1},a_{t_1+1},\dots ,a_{t_{q}-1},x_{t_{q}},x_{t_{q+1}},\dots, x_{t_z})\in %SpT_{2n}(\varpi_{t}), \;1\le t_1\le t_q\le t_z\le t.$$
\item For $1\le z\le t\le l\le 2n$ %and $t\in[n]\cap 2\Z$
such that
 $ z\le min\{l,2n-l\}$ and $l-z\in 2\Z$,
 \end{enumerate}

 %$(a)$ if $\oo 1\le t_z\le t$, %$\b\cap\s(\b)=\emptyset$,

 \begin{align}\redu_z^{-1}(a_{t_1}, \dots,a_{t_z})=\rm{red}_{z}^{-1}(\mathbf{b}):=T_0(a_{t_1})T_1\cdots(a_{t_z}) T_z\in SST_{2n}(\varpi_l),\nonumber
\end{align}
where \begin{align}%\label{inverseredu:nofactors}
&T_0=(1,2,\dots l_1),\\
&(a_{t_i})T_i=
\begin{cases}
 ({a}_{t_i})( a_{t_i}+1,\dots, a_{t_i}+l_{i+1}),
  &
 \mbox{if } a_{t_i}\in 2\mathbb{Z}\\
 ({a}_{t_i})( a_{t_i}+1+1,a_{t_i}+1+2,\dots, a_{t_i}+1+l_{i+1}),
 &\mbox{if } a_{t_i}\notin 2\mathbb{Z}\\
\end{cases},& 1\le i\le z,% \label{inverseredu:nofactors2}
\end{align}
and
$l_1,\dots, l_z$, $l_{z+1}$ are nonnegative even numbers defined by
\begin{align}%\label{l:nofactors}
&l_1=
\begin{cases} min\{a_{t_1}-2, l-z\},& \mbox{if } a_{t_1}\in 2\mathbb{Z}\\
min\{ a_{t_1}-1, l-z\},&\mbox{if } a_{t_1}\notin 2\mathbb{Z},\\
\end{cases}
\end{align}
\begin{align}
&l_{i+1}=
\begin{cases}min\{a_{t_{i+1}}-a_{t_i}-2, l-z-\sum_{k=1}^i l_k\}, &\mbox{if }
a_{t_i}, a_{t_{i+1}}\in 2\mathbb{Z} \mbox{ or } a_{t_i}, a_{t_{i+1}}\notin 2\mathbb{Z}
\\
min\{ a_{t_{i+1}}-a_{t_i}-1,~~ l-z-\sum_{k=1}^i l_k\},&\mbox{if }
a_{t_i}\in 2\mathbb{Z}, a_{t_{i+1}}\notin 2\mathbb{Z} \\
min\{a_{t_{i+1}}-a_{t_i}-3,  l-z-\sum_{k=1}^i l_k\},&\mbox{if }
 a_{t_i}\notin 2\mathbb{Z}, ~~~ a_{t_{i+1}}\in 2\mathbb{Z},
\end{cases},\quad \mbox{ for } 1\le i<z,\label{ellsagain}
\end{align}
and $\displaystyle l_{z+1}=l-z-\sum_{i=1}^{z}l_i.$

Moreover,
\begin{align*}&T_0(a_{t_1})\bigcdot (a_1, \dots,a_{t_1-1},  x_{t_1})\in SST_{2n}(\varpi_{l_1+1}+\varpi_{t_1})\\
&T_1(a_{t_2})\bigcdot(a_{t_1+1},\dots ,a_{t_2-1},x_{t_2})
\in SST_{2n}(\varpi_{l_2+1}+\varpi_{t_2-t_1})\\
&\vdots\\
&T_{z-1}(a_{t_z})\bigcdot(a_{t_{z-1}+1},\dots ,a_{t_z-1},x_{t_z})
\in SST_{2n}(\varpi_{l_{z}+1}+\varpi_{t_z-t_{z-1}})\\
&T_z\bigcdot (a_{t_z+1},\dots, a_{t})\in SST_{2n}(\varpi_{l_{z+1}}+\varpi_{t-t_z})\\
\end{align*}
and
\begin{align}T_0(a_{t_1})T_1\cdots(a_{t_z}) T_z\bigcdot (a_1, \dots,a_{t_1-1},  x_{t_1},a_{t_1+1},\dots ,a_{t_z-1},x_{t_z},a_{t_z+1},\dots, a_{t})\in SST_{2n}(\varpi_l+\varpi_t).
\end{align}
\begin{enumerate}
\item [(3)]$\rm{red}_{z}^{-1}(\mathbf{b})\bigcdot (\underset{(r_{t_1},\dots,r_{t_z})}\longleftarrow S)$ $\in SST_{2n}(\varpi_l+\mu-\delta_{(r_{t_1},\dots,r_{t_z})})$.
    \end{enumerate}
\end{thm}

\begin{proof} Let $\B:=T_0(a_{t_1})T_1\cdots(a_{t_z}) T_z\in SST_{2n}(\varpi_l)$ where $\ell(T_i)=l_{i+1}$, for $i=0,1,\dots,z$,  and write $S^1$ as a union of $z+1$ columns $$S^1_i=(a_{t_i+1},\dots ,a_{t_{i+1}-1},x_{t_{i+1}})\in SpT_{2n}(\varpi_{t_{i+1}-t_i}), \mbox{  for $i=0,1,\dots,z$,}$$
 %of length $t_{i+1}-t_i$, for $i=0,1,\dots,z$,
 with $t_0:=0$, $t_{z+1}:=t$, and $x_{t_{z+1}}:=a_t$,
\begin{align}S^1=(a_1, \dots,a_{t_1-1},  x_{t_1})\sqcup \bigsqcup_{i=1}^{z-1} (a_{t_i+1},\dots ,a_{t_{i+1}-1},x_{t_{i+1}})\sqcup(a_{t_z+1},\dots, a_{t}).\label{S1column}
\end{align}

Let $T_0(a_{t_1})\in  SST_{2n}(\varpi_{l_1+1})$.
Since $\a\cap \s(\a)=\emptyset$, one has from Lemma \ref{lem:swapping},

\begin{align}
  a_{t_1}\in 2\mathbb{Z}\Rightarrow a_{t_1}-2\ge a_{t_1-1},\\
 a_{t_1}\notin 2\mathbb{Z}\Rightarrow a_{t_1}-1\ge a_{t_1-1}.
\end{align}

\emph{If $\underline{l_1=l-z}$}, that is,

\begin{align}%\label{l:nofactors}
&l_1=l-z\le
\begin{cases} a_{t_1}-2,& \mbox{if } a_{t_1}\in 2\mathbb{Z}\\
 a_{t_1}-1,&\mbox{if } a_{t_1}\notin 2\mathbb{Z},\\
\end{cases}
\end{align}
then $l_i=0$, for $i=2,\dots, z+1$, and $T_i=\emptyset$, $i=1,\dots,z$.

Recall $ 1\le t_1<t_2<\cdots<t_z\le t$ and observe that
\begin{align}&l-z\ge t-z\ge t-t_z\nonumber\\
&t-z=t_1+(t_2-t_1)+\cdots+(t_z-t_{z-1})+(t-t_z)-z\nonumber\\
&=(t_1-1)+(t_2-t_1-1)+\cdots+(t_z-t_{z-1}-1)+(t-t_z).
\end{align}
Consequently,
\begin{align}l-z\ge t_j-t_{j-1}-1, \quad l-z\ge t-z\ge t_j-j,\; \mbox{ for $j=1,\dots z.$} \label{chicken}
\end{align}

Therefore
$$\redu_{z}^{-1}(\mathbf{b}):=T_0(a_{t_1},\dots, a_{t_z}) =(1,\dots, l-z)(a_{t_1},\dots,  a_{t_z})\in SST_{2n}(\varpi_l)$$
and from \eqref{chicken}
\begin{align}&T_0(a_{t_1})\bigcdot (a_1, \dots,a_{t_1-1},  x_{t_1})=\nonumber\\
&=(1,\dots, t_1-1,\dots, l-z,a_{t_1})\bigcdot (a_1, \dots,a_{t_1-1},  x_{t_1})\in SST_{2n}(\varpi_{l-z+1}+\varpi_{t_1}),\\
&\mbox{ by $t_1-1\le a_{t_1-1}$,  and $t_1\le l-z+1\le a_{t_1}\le x_{t_1}$},\nonumber
\end{align}
\begin{align}& T_0(a_{t_1},a_{t_2})\bigcdot (a_1, \dots,a_{t_1-1},  x_{t_1},a_{t_1+1},\dots a_{t_2-1},x_{t_2})=\nonumber\\
&=(1,\dots, t_2-2,\dots, l-z,a_{t_1},a_{t_2})\bigcdot
(a_1, \dots,a_{t_1-1},  x_{t_1},a_{t_1+1},\dots a_{t_2-1},x_{t_2})\in SST_{2n}(\varpi_{l-z+2}+\varpi_{t_2}),\\
&\mbox{ by $t_2-2=t_1+(t_2-t_1)-2\le  a_{t_2-2}=a_{t_1+(t_2-t_1)-2}$ and $a_{t_1}\le a_{t_2-1}$, $a_{t_2}\le x_{t_2}$}, \nonumber\\
&l-z+2\ge t_2,\;\; l-z+1\ge t_2-1\ge t_1.\nonumber
\end{align}
 The set $(1,\dots, t_2-2,\dots, l-z)$ has cardinality at least $t_2-2$ where
$$(1,\dots, t_2-2)\le (a_1, \dots,a_{t_1-1},  x_{t_1},a_{t_1+1},\dots,a_{t_1+(t_2-t_1)-2}=a_{t_2-2}),\;\; a_{t_2-2}< a_{t_2-1},$$
$$(1,\dots, t_2-2, a_{t_1},a_{t_2})\le (a_1, \dots,a_{t_1-1},  x_{t_1},a_{t_1+1},\dots,a_{t_1+(t_2-t_1)-2}=a_{t_2-2}, a_{t_2-1}, x_{t_2}),$$
and $(1,\dots, t_2-2,\dots, l-z),$
$t_2-2\le l-z\le a_{t_1}-1<a_{t_1}<a_{t_2}$.
%If $t_2-t_1-2=-1$ one has $t_2-2=t_1-1\le a_{t_1-1}\le a_{t_1}-1$.
%If $t_2-2=t_1$, one has $t_2-2=t_1< a_{t_1}\le x_{t_1}$
\begin{align}
&\vdots\nonumber\\
&T_0(a_{t_1},\dots,a_{t_z})\bigcdot (a_1, \dots,a_{t_1-1},  x_{t_1},a_{t_1+1},\dots,\dots a_{t_2-1},x_{t_2},a_{t_2+1},\dots, ,a_{t_{z-1}-1},x_{t_{z-1}},a_{t_{z-1}+1},\dots, a_{t_z-1},x_{t_z} )\nonumber\\
&=(1,\dots, t_z-z,\dots, l-z,a_{t_1},a_{t_2},\dots,a_{t_z})\bigcdot \nonumber\\
&\bigcdot (a_1, \dots,a_{t_1-1},  x_{t_1},a_{t_1+1},\dots ,a_{t_{z-1}-1},x_{t_{z-1}},a_{t_{z-1}+1},\dots, a_{t_z-1},x_{t_z} )\in SST_{2n}(\varpi_l+\varpi_{t_z}),\\
&\mbox{  by $t_z-z %=t_1+(t_z-t_1)-z=t_1+(t_2-t_1-1)+\cdots+(t_z-t_{z-1}-1)-1
\le a_{t_z-z}<a_{t_z-(z-1)}<a_{t_z-(z-2)}<\cdots<a_{t_z-1}$}, \nonumber\\
&a_{t_i}\le x_{t_i},\; 1\le i\le z,\;\; l=l-z+z\ge t\ge t_z.\nonumber
\end{align}

Note $$t_i\le t_{i+1}-1,\; i=1,\dots,z-1,\Rightarrow t_1\le t_2-1\le t_3-2\le \cdots\le t_{z-1}-(z-2)\le t_z-(z-1)$$
$$\Rightarrow t_1\le t_z-(z-1),\; \;t_2\le t_z-(z-2),\dots, t_{z-2}\le t_z-2,\;t_{z-1}\le t_z-1,$$

$$a_{t_1}=a_{t_{z-(z-1)}}\le a_{{t_z}-(z-1)},\dots,a_{t_{z-2} }\le a_{t_z-2 },\;a_{t_{z-1} }\le a_{t_z-1 } ,\;a_{t_z}\le x_{t_z},$$
\noindent and
$$(1,\dots, t_z-z, a_{t_1},\dots,a_{t_{z-2}},a_{t_{z-1} },a_{t_z})\le (a_1, \dots,a_{t_z-z},a_{t_z-(z-1)},a_{t_z-(z-2)},\dots,a_{t_z-2 }, a_{t_z-1 } , x_{t_z})$$ $$
\le (a_1, \dots,a_{t_1-1},  x_{t_1},a_{t_1+1},\dots,\dots a_{t_2-1},x_{t_2},a_{t_2+1},\dots, ,a_{t_{z-1}-1},x_{t_{z-1}},a_{t_{z-1}+1},\dots, a_{t_z-1},x_{t_z} ).$$

%Observe $t_z-z\le l-z\le a_{t_1}-1<a_{t_1}<a_{t_2}<\cdots<a_{t_z}$. If $t_z-t_1-z=-1$ one has $t_z-z=t_1-1\le a_{t_1-1}\le %a_{t_1}-1$; $t_z-z=t_1$ one has $t_z-z=t_1< a_{t_1}\le x_{t_1}$. %$t_z-z=t_1+j\le t_2-1$ then $t_z-z\le %a_{t_z-z}=a_{t_1+j}\le a_{t_2-1}\le a_{t_2}-1$ $\dots$

%Consequently,
Assuming $t_z<t\Leftrightarrow t_z+1\le t$,
\begin{align}
&T_0(a_{t_1},\dots,a_{t_z})\bigcdot (a_1, \dots,a_{t_1-1},  x_{t_1},a_{t_1+1},\dots a_{t_2-1},x_{t_2} ,a_{t_{z-1}},x_{t_{z-1}},a_{t_{z-1}}+1,\dots, a_{t_z-1},x_{t_z})(a_{t_z+1},\dots, a_{t})\nonumber\\
&=(1,\dots, t-z,\dots, l-z,a_{t_1},a_{t_2},\dots,a_{t_z})\bigcdot \nonumber\\
&\bigcdot (a_1, \dots,a_{t_1-1},  x_{t_1},a_{t_1+1},\dots ,a_{t_z-1},x_{t_z},a_{t_z+1},\dots, a_{t})\in SST_{2n}(\varpi_l+\varpi_{t}),\quad l\ge t.\nonumber
\end{align}
Similarly as before
$$t_1\le t_2-1\le t_3-2\le \cdots\le t_{z-1}-(z-2)\le t_z-(z-1)\le t-z$$
$$t_1\le t-z,\; \;t_2\le t-(z-1),\dots, t_{z-2}\le t-2,\;t_{z-1}\le t-2,\; t_z\le t-1$$
$$a_{t_1}<a_{t-z+1},\dots,a_{t_{2} }\le a_{t-z+2 },\;a_{t_{z-1} }\le a_{t-1 } ,\;a_{t_z}\le x_{t_z}< a_{t_z+1}\le a_t,$$
$$(1,\dots, t-z, a_{t_1},\dots,a_{t_z})\le (a_1, \dots,a_{t-z},a_{t-(z-1)},a_{t-(z-2)},\dots, a_{t-1 } , a_t)$$
$$\le (a_1, \dots,a_{t_1-1},  x_{t_1},a_{t_1+1},\dots,\dots a_{t_2-1},x_{t_2},a_{t_2+1},\dots, ,a_{t_{z-1}-1},x_{t_{z-1}},a_{t_{z-1}+1},\dots, a_{t_z-1},x_{t_z},a_{t_z+1},\dots, a_{t} ).$$

\medskip
For $\underline{l_1<l-z}$,
\begin{align}
\begin{cases}l_1= a_{t_1}-2,\;  a_{t_1}\in 2\mathbb{Z}\Rightarrow l_1=a_{t_1}-2\ge a_{t_1-1}\Rightarrow l_1+1=a_{t_1}-1\ge a_{t_1-1}+1\ge t_1,\\
\mbox{ or }\\
 l_1=a_{t_1}-1,\; a_{t_1}\notin 2\mathbb{Z}\Rightarrow l_1=a_{t_1}-1\ge a_{t_1-1}\Rightarrow l_1+1=a_{t_1}\ge t_1.\\
\end{cases}
\end{align}
This means that the \textit{row} \emph{index} \emph{of} $l_1$ \emph{in} $T_0=(1,\dots,l_1)$ \emph{is} \emph{at} \emph{least} $a_{t_1-1}\ge t_1-1$.

When $t_1=1$ and $a_1=1,2$, one has $l_1=0$ and
\begin{align*}
&T_0(a_{t_1})T_1\cdots(a_{t_z}) T_z\bigcdot (a_1, \dots,a_{t_1-1},  x_{t_1},a_{t_1+1},\dots ,a_{t_z-1},x_{t_z},a_{t_z+1},\dots, a_{t})=\\
&=(a_{1})T_1\cdots(a_{t_z}) T_z\bigcdot ( x_{1},a_{2},\dots ,a_{t_z-1},x_{t_z},a_{t_z+1},\dots, a_{t}).
\end{align*}

Hence, either
\begin{align}
\begin{cases}T_0(a_{t_1})=(1,\dots,l_1)(a_{t_1})=(1,\dots, a_{t_1-1},\dots, l_1,a_{t_1}),& if \mbox{ $l_1>a_{t_1-1}$}, \\ T_0(a_{t_1})=(1,\dots,l_1)$ $(a_{t_1})$ $=$ $(1,\dots, a_{t_1-1},a_{t_1})&, \mbox{ if $a_{t_1-1}= l_1$},
\end{cases}
\end{align}
and

\begin{align}T_0(a_{t_1})\bigcdot S^1_1=T_0(a_{t_1})\bigcdot (a_1, \dots,a_{t_1-1},  x_{t_1})\in SST_{2n}(\varpi_{l_1+1}+\varpi_{t_1}).\label{T0}
\end{align}
is a two-column semi-standard tableau with first column of length $l_1+1\ge a_{t_1-1}+1\ge t_1$.

Let us now consider the column $T_0(a_{t_1})T_1(a_{t_2})\in  SST_{2n}(\varpi_{l_1+l_2+2})$ and the two first column factors in \eqref{S1column}, with $S^1_1=(a_1, \dots,a_{t_1-1},  x_{t_1}), \;\ell(S_1^1)=l_1$ and $\ell(S^1_2)=\ell(a_{t_1+1},\dots ,a_{t_2-1},x_{t_2})=t_2-t_1$,
$$S^1_1\sqcup S^1_2=(a_1, \dots,a_{t_1-1},  x_{t_1})\sqcup(a_{t_1+1},\dots ,a_{t_2-1},x_{t_2})\in SST_{2n}(\varpi_{t_2}), \; t_2=t_1+(t_2-t_1)$$
where
\begin{align}\label{casell}
T_1=
\begin{cases}
( a_{t_1}+1,\dots, a_{t_1}+l_{2}),
  &
 \mbox{if } a_{t_1}\in 2\mathbb{Z}\\
 ( a_{t_1}+1+1,a_{t_1}+1+2,\dots, a_{t_1}+1+l_{2}),
 &\mbox{if } a_{t_1}\notin 2\mathbb{Z}.\\
\end{cases}
\end{align}
For $\underline{l_2=l-z-l_1}$, the proof goes along the same lines as before:

\begin{align}%\label{lll3}
&l_{2}=l-z-l_1\le
\begin{cases}a_{t_{2}}-a_{t_1}-2, &\mbox{if }
a_{t_1}, a_{t_{2}}\in 2\mathbb{Z} \mbox{ or } a_{t_1}, a_{t_{2}}\notin 2\mathbb{Z}
\\
 a_{t_{2}}-a_{t_1}-1,&\mbox{if }
a_{t_1}\in 2\mathbb{Z}, a_{t_{2}}\notin 2\mathbb{Z} \\
a_{t_{2}}-a_{t_1}-3, &\mbox{if }
 a_{t_1}\notin 2\mathbb{Z}, ~~~ a_{t_{2}}\in 2\mathbb{Z}
\end{cases},\label{lll3}\\
&\Leftrightarrow\nonumber\\
&a_{t_1}+l-z-l_1\le
\begin{cases}a_{t_{2}}-2, &\mbox{if }
a_{t_1}, a_{t_{2}}\in 2\mathbb{Z} \mbox{ or } a_{t_1}, a_{t_{2}}\notin 2\mathbb{Z}
\\
 a_{t_{2}}-1,&\mbox{if }
a_{t_1}\in 2\mathbb{Z}, a_{t_{2}}\notin 2\mathbb{Z} \\
a_{t_{2}}-3, &\mbox{if }
 a_{t_1}\notin 2\mathbb{Z}, ~~~ a_{t_{2}}\in 2\mathbb{Z}
\end{cases},\label{lll3another}
\end{align}
and $l_i=0$, for $i=3,\dots, z+1$, and $T_i=\emptyset$, $i=2,\dots,z$.

Therefore
$$\redu_{z}^{-1}(\mathbf{b}):=T_0(a_{t_1})T_1(a_{t_2},\dots, a_{t_z}) =(1,\dots, l_1)(a_{t_1})  T_1(a_{t_2},\dots, a_{t_z})\in SST_{2n}(\varpi_l)$$
with
\begin{align}\label{casellother}
T_1=
\begin{cases}
( a_{t_1}+1,\dots, a_{t_1}+l-z-l_{1}),
  &
 \mbox{if } a_{t_1}\in 2\mathbb{Z}\\
 ( a_{t_1}+1+1,a_{t_1}+1+2,\dots, a_{t_1}+1+l-z-l_{1}),
 &\mbox{if } a_{t_1}\notin 2\mathbb{Z}.\\
\end{cases}
\end{align}
\iffalse
From \eqref{lll3another} with $a_{t_1}\in 2\mathbb{Z}$, one has
\begin{align}a_{t_1}+l-z-l_{1}\le
\begin{cases} a_{t_{2}}-2& \mbox{ if $a_{t_2}\in 2\mathbb{Z}$}\\
a_{t_{2}}-1& \mbox{ if $a_{t_2}\notin 2\mathbb{Z}$}\\
\end{cases}
\end{align}
and then from \eqref{T0} and \eqref{casellother}

\begin{align*}&T_0(a_{t_1})\bigcdot (a_1, \dots,a_{t_1-1},  x_{t_1})=
&=(1,\dots,t_1-1,\dots,l_1,a_{t_1})\bigcdot (a_1, \dots,a_{t_1-1},  x_{t_1})\in SST_{2n}(\varpi_{l_1+1}+\varpi_{t_1})%\label{T0}
\end{align*},

\begin{align}&\oo T_0(a_{t_1})T_1(a_{t_2})\bigcdot
(a_1, \dots,a_{t_1-1},  x_{t_1},a_{t_1+1},\dots a_{t_2-1},x_{t_2})=\nonumber\\
&=(1,\dots,  l_1,a_{t_1})( a_{t_1}+1,\dots, a_{t_1}+l-z-l_{1},a_{t_2})\bigcdot (a_1, \dots,a_{t_1-1},  x_{t_1},a_{t_1+1},\dots a_{t_2-1},x_{t_2})\\
&\in SST_{2n}(\varpi_{l-z+2}+\varpi_{t_2}),\\
&l-z+1\ge t_1,\nonumber \oo to do
\end{align}
\fi
\medskip

 For $\underline{l_2<l-z-l_1}$,
\begin{align}\label{l2}
&l_{2}=
\begin{cases}a_{t_{2}}-a_{t_1}-2, &\mbox{if }
a_{t_1}, a_{t_{2}}\in 2\mathbb{Z} \mbox{ or } a_{t_1}, a_{t_{2}}\notin 2\mathbb{Z}
\\
 a_{t_{2}}-a_{t_1}-1,&\mbox{if }
a_{t_1}\in 2\mathbb{Z}, a_{t_{2}}\notin 2\mathbb{Z} \\
a_{t_{2}}-a_{t_1}-3,  &\mbox{if }
 a_{t_1}\notin 2\mathbb{Z}, ~~~ a_{t_{2}}\in 2\mathbb{Z}.
\end{cases}
\end{align}
\textbf{Case} $a_{t_1}, a_{t_{2}}\in 2\mathbb{Z}$,
\begin{align}
  a_{t_2}\in 2\mathbb{Z}\Rightarrow a_{t_2}-2\ge a_{t_2-1},\\
 a_{t_1}\in 2\mathbb{Z}\Rightarrow a_{t_1}+1\le a_{t_1+1}.%\\
\end{align}
This means that  $a_{t_1}+l_2=a_{t_1}+a_{t_{2}}-a_{t_1}-2=a_{t_{2}}-2\ge a_{t_2-1}$ and

\begin{align}&T_0(a_{t_1})T_1(a_{t_2})=(1,\dots,l_1,a_{t_1},a_{t_1}+1,a_{t_1}+2,\dots, a_{t_1}+l_{2}=a_{t_{2}}-2,a_{t_2}),\\
&\;l_1\ge a_{t_1-1}\Rightarrow l_1\ge t_1-1\Rightarrow l_1+1\ge t_1. %,\\
%&\;a_{t_1}+l_2=a_{t_2}-2\ge a_{t_2-1}=a_{t_1+{t_2-t_1}-1}\Rightarrow a_{t_1}+l_{2}\ge  a_{t_1}+{t_2-t_1}-1\Rightarrow %l_2+1\ge t_2-t_1.
\end{align}
%This means that the row index of  $a_{t_1}+l_{2}$ in $T_1=(a_{t_1}+1,a_{t_1}+2,\dots, a_{t_1}+l_{2}=a_{t_{2}}-2)$ is at %least ${t_2-t_1}-1$.
Since $a_{t_1}+1\le a_{t_1+1}$, $a_{t_1}+2\le a_{t_1+1}+1\le a_{t_1+2}$, $,\dots$, $a_{t_1}+{t_2-t_1}-1\le a_{t_1+t_2-t_1-1}=a_{t_2-1}$, it follows that
$$a_{t_1}+l_2=a_{t_2}-2\ge a_{t_2-1}=a_{t_1+{t_2-t_1}-1}\ge a_{t_1}+{t_2-t_1}-1\Rightarrow a_{t_1}+l_{2}\ge  a_{t_1}+{t_2-t_1}-1\Rightarrow l_2+1\ge t_2-t_1,$$
and
\begin{align}&T_1(a_{t_2})\bigcdot S^1_2=\nonumber\\
&=T_1(a_{t_2})\bigcdot(a_{t_1+1},\dots ,a_{t_2-1}=a_{t_1+{t_2-t_1}-1},x_{t_2})\nonumber\\
&=(a_{t_1}+1,a_{t_1}+2,\dots,a_{t_1}+t_2-t_1-1,\dots, a_{t_1}+l_{2}=a_{t_{2}}-2,a_{t_{2}})\bigcdot(a_{t_1+1},\dots ,a_{t_2-1}=a_{t_1+{t_2-t_1}-1},x_{t_2})\nonumber\\
&\in SST_{2n}(\varpi_{l_2+1}+\varpi_{t_2-t_1})\nonumber
\end{align}
is a two column semi-standard tableau with first column of length
$$l_2+1\ge  t_2-t_1\Rightarrow l_1+l_2+2\ge  t_1+t_2-t_1=t_2$$

Hence, from \eqref{T0},
\begin{align}T_0(a_{t_1})T_1(a_{t_2})\bigcdot (S_1^1\sqcup S^1_2)\in SST_{2n}(\varpi_{l_1+l_2+2}+\varpi_{t_2}).
\end{align}

\textbf{Case} $a_{t_1}, a_{t_{2}}\notin 2\mathbb{Z}$,
\begin{align}
  a_{t_2}\notin 2\mathbb{Z}\Rightarrow a_{t_2}-1\ge a_{t_2-1},\\
 a_{t_1}\notin 2\mathbb{Z}\Rightarrow a_{t_1}+2\le a_{t_1+1}.\label{ineqodd}
\end{align}

Therefore, from \eqref{ineqodd} and \eqref{l2},
\begin{align}(a_{t_1}+1)+1= a_{t_1}+2\le a_{t_1+1}\Rightarrow (a_{t_1}+1)+2=a_{t_1}+3\le a_{t_1+1}+1\le a_{t_1+2}\label{vult0}
\end{align}
and by induction
\begin{align} (a_{t_1}+1)+(t_2-t_1-1)\le a_{t_1+t_2-t_1-1}=a_{t_2-1}.\label{dino}
\end{align}

Then it follows from \eqref{l2} and \eqref{dino},
\begin{align}\label{vult}
&(a_{t_1}+1)+l_2=a_{t_2}-1\ge a_{t_2-1}=a_{t_1+{t_2-t_1}-1}\ge (a_{t_1}+1)+{t_2-t_1}-1\\
&\Rightarrow (a_{t_1}+1)+l_{2}\ge  (a_{t_1}+1)+{t_2-t_1}-1\Rightarrow l_2+1\ge t_2-t_1.
\end{align}

Henceforth, from \eqref{vult0}, \eqref{dino} and  \eqref{vult},
\begin{align}&T_1(a_{t_2})\bigcdot S^1_2=\nonumber\\
&=T_1(a_{t_2})\bigcdot(a_{t_1+1},\dots ,a_{t_2-1}=a_{t_1+{t_2-t_1}-1},x_{t_2})\nonumber\\
&=(a_{t_1}+1+1,a_{t_1}+1+2,\dots,a_{t_1}+1+t_2-t_1-1,\dots, a_{t_1}+1+l_{2}=a_{t_{2}}-1,a_{t_{2}})\bigcdot\\
&\bigcdot (a_{t_1+1},\dots ,a_{t_2-1}=a_{t_1+{t_2-t_1}-1},x_{t_2})\in SST_{2n}(\varpi_{l_2+1}+\varpi_{t_2-t_1})\nonumber
\end{align}
is a two column semi-standard tableau with first column of length
$$l_2+1\ge  t_2-t_1\Rightarrow l_1+l_2+2\ge  t_1+t_2-t_1=t_2$$

\textbf{Case} $a_{t_1}\in 2\mathbb{Z}$, $ a_{t_{2}}\notin 2\mathbb{Z}$,

\begin{align}
  a_{t_2}\notin 2\mathbb{Z}\Rightarrow a_{t_2}-1\ge a_{t_2-1},\label{again}\\
 a_{t_1}\in 2\mathbb{Z}\Rightarrow a_{t_1}+1\le a_{t_1+1}.\label{ineqodd2}
\end{align}

Since $a_{t_1}+1\le a_{t_1+1}$, $a_{t_1}+2\le a_{t_1+1}+1\le a_{t_1+2}$, $,\dots$, $a_{t_1}+{t_2-t_1}-1\le a_{t_1+t_2-t_1-1}=a_{t_2-1}$.

It follows from \eqref{casell}, $l_2= a_{t_{2}}-a_{t_1}-1$, and from \eqref{again} that
$$a_{t_1}+l_2=a_{t_2}-1\ge a_{t_2-1}=a_{t_1+{t_2-t_1}-1}\ge a_{t_1}+{t_2-t_1}-1\Rightarrow a_{t_1}+l_{2}\ge  a_{t_1}+{t_2-t_1}-1\Rightarrow l_2+1\ge t_2-t_1,$$
Henceforth
\begin{align}&T_1(a_{t_2})\bigcdot S^1_2=\nonumber\\
&=T_1(a_{t_2})\bigcdot(a_{t_1+1},\dots ,a_{t_2-1}=a_{t_1+{t_2-t_1}-1},x_{t_2})\nonumber\\
&=(a_{t_1}+1,a_{t_1}+2,\dots,a_{t_1}+t_2-t_1-1,\dots, a_{t_1}+l_{2}=a_{t_{2}}-1,a_{t_{2}})\bigcdot(a_{t_1+1},\dots ,a_{t_2-1}=a_{t_1+{t_2-t_1}-1},x_{t_2})\nonumber\\
&\in SST_{2n}(\varpi_{l_2+1}+\varpi_{t_2-t_1})\nonumber
\end{align}
is a two column semi-standard tableau with first column of length
$$l_2+1\ge  t_2-t_1\Rightarrow l_1+l_2+2\ge  t_1+t_2-t_1=t_2$$

\textbf{Case} $a_{t_1}\notin 2\mathbb{Z}$, ~~~ $a_{t_{2}}\in 2\mathbb{Z}$.

It follows from \eqref{l2}, $l_2= a_{t_{2}}-a_{t_1}-3$, and

\begin{align}
  a_{t_2}\in 2\mathbb{Z}\Rightarrow a_{t_2}-2\ge a_{t_2-1},\label{again3}\\
 a_{t_1}\notin 2\mathbb{Z}\Rightarrow a_{t_1}+2\le a_{t_1+1}.\label{ineqodd3}
\end{align}

Therefore, from  \eqref{ineqodd3},
\begin{align}&(a_{t_1}+1)+1= a_{t_1}+2\le a_{t_1+1}\Rightarrow (a_{t_1}+1)+2=a_{t_1}+3\le a_{t_1+1}+1\le a_{t_1+2},\nonumber\\
&\dots,(a_{t_1}+1)+(t_2-t_1-1)\le a_{t_1+t_2-t_1-1}=a_{t_2-1}.\label{dino4}
\end{align}

Then it follows from \eqref{l2}, \eqref{again3} and \eqref{dino4},
\begin{align}\label{vult3}
&(a_{t_1}+1)+l_2=a_{t_2}-2\ge a_{t_2-1}=a_{t_1+{t_2-t_1}-1}\ge (a_{t_1}+1)+{t_2-t_1}-1\\
&\Rightarrow (a_{t_1}+1)+l_{2}\ge  (a_{t_1}+1)+{t_2-t_1}-1\Rightarrow l_2+1\ge t_2-t_1 \label{dino5}
\end{align}

Henceforth, from  \eqref{dino4} and  \eqref{dino5},
\begin{align}&T_1(a_{t_2})\bigcdot S^1_2=\nonumber\\
&=T_1(a_{t_2})\bigcdot(a_{t_1+1},\dots ,a_{t_2-1}=a_{t_1+{t_2-t_1}-1},x_{t_2})\nonumber\\
&=(a_{t_1}+1+1,a_{t_1}+1+2,\dots,a_{t_1}+1+t_2-t_1-1,\dots, a_{t_1}+1+l_{2}=a_{t_{2}}-2,a_{t_{2}})\bigcdot \nonumber\\
&\bigcdot (a_{t_1+1},\dots ,a_{t_2-1}=a_{t_1+{t_2-t_1}-1},x_{t_2})\in SST_{2n}(\varpi_{l_2+1}+\varpi_{t_2-t_1})\nonumber
\end{align}
is a two column semi-standard tableau with first column of length
$$l_2+1\ge  t_2-t_1\Rightarrow l_1+l_2+2\ge  t_1+t_2-t_1=t_2$$

Hence in either case

\begin{align}T_0(a_{t_1})T_1(a_{t_2})\bigcdot (S_1^1\sqcup S^1_2)\in SST_{2n}(\varpi_{l_1+l_2+2}+\varpi_{t_2}).\nonumber
\end{align}

Let $m\ge 1$ and by induction  assume that
\begin{align}T_0(a_{t_1})T_1(a_{t_2})\cdots T_{m-1}(a_{t_m})\bigcdot (\bigsqcup_{j=1}^m S^1_j)\in SST_{2n}(\varpi_{l_1+\cdots+l_m+m}+\varpi_{t_m}).\nonumber
\end{align}

%To prove that
The proof for $m+1$,
\begin{align}T_0(a_{t_1})T_1(a_{t_2})\cdots T_{m-1}(a_{t_m})T_{m}(a_{t_{m+1}})\bigcdot (\bigsqcup_{j=1}^{m+1} S^1_j)\in SST_{2n}(\varpi_{l_1+\cdots+l_{m+1}+m+1}+\varpi_{t_{m+1}}),\nonumber
\end{align}
%we have to show that
%\begin{align}T_{m}(a_{t_{m+1}})\bigcdot S^1_{m+1}\in SST_{2n}(\varpi_{l_{m+1}+1}+\varpi_{t_{m+1}-t_m})
%\end{align}
%The proof for $m+1$ is
is just a mimic of the case $m=2$ that we do not repeat.
\end{proof}
\begin{obs}  In the previous theorem, we have considered $(\underset{(r_{t_1},\dots,r_{t_z})}\longleftarrow S)$ with the same length as $S$.
 %to avoid cumbersome notation.
 When the length of $S$ is shortened by the reverse  bumping,  the proof is similar.
\end{obs}
%%%%%%%%%%%%%%%%%%%%%%%%%%%%%%%%%%%Newsubsection%%%%%%%%%%%%%%%%%%%%%%%%%%%%%%%%%%%%%%%%%%%%%%%%%%%%%%%%%%%%%%%%%%%%%%%%%%%%%%%%
\subsection{Surjectivity and the algorithm computing the inverse of the quantum LR map } To prove that $\LRAII^{AII}$ is also a surjection we exhibit the right inverse $\widetilde R$ of $\LRAII^{AII}$.
Recall the reverse column insertion in Subsection \ref{subsec:insert}, Remark \ref{noexplicit}, and the properties of bumping and reverse bumping routes in Remark \ref{routes} and  Remark \ref{revroutes} in there. %Subsection \ref{subsec:insert}.

Let  $\mu\subseteq\lambda\in Par_{\le 2n}$ with  $\mu\in Par_{\le n}$. If $\lambda=\mu$,  $Rec_{2n}(\lambda/\lambda)=Rec_{2n}(())=\widetilde Rec_{2n}(())$ and  $\widetilde R(S, ())=S$,
\begin{align}%${\LRAII^{AII}}^{-1}
\widetilde R:SpT_{2n}(\mu)\times \widetilde Rec_{2n}(())&\longrightarrow SpT_{2n}(\mu)\subset SST_{2n}(\mu)\\
(S,())&\mapsto S,~\redu(S)=S
\end{align}

Let us consider $\mu\subseteq\lambda$ and  fix arbitrarily $S\in SpT_{2n}(\mu)$. In Theorem \ref{thm:generalhole}, based on the slack statistics  introduced in \cite{azslack} (here Subsection \ref{sec:slackdata}), one defines for $N>0$,

\begin{align}\label{Rtilde}
%&{LR_{|S}^{AII}}^{-1}
&\widetilde R_{|S}:\{S\}\times \widetilde Rec_{2n}(\lambda/\mu)\longrightarrow SST_{2n}(\lambda)\nonumber\\
&\qquad\qquad\qquad \quad(S,Q)\mapsto S^{\r^{(N)}\cdots\r^{(2)} \r^{(1)}}=  c \circ (\rm{red}_{t_0^{(1)}}^{-1},\rm{id})\circ (\;,\underset{{\mu^{(1)}/{\mu^{(1)}}'}} \leftarrow
\bigg(\cdots  \nonumber\\
&(\;,\underset{{\mu^{(N-2)}/{\mu^{(N-2)}}'}}\leftarrow\bigg( c \circ (\rm{red}_{t_0^{(N-1)}}^{-1},\rm{id})\circ (\;,\underset{{\mu^{(N-1)}/{\mu^{(N-1)}}'}} \leftarrow
\bigg(c \circ (\rm{red}_{t_0^{(N)}}^{-1},\rm{id})\circ (\;,\underset{{\mu^{(N)}/{\mu^{(N)}}'}}\leftarrow S)\bigg))\bigg))\cdots)\bigg))
\end{align}
where $N>0$ is the number of vertical strips $\{\mu^{(i-1)}/\mu^{(i)}\}_{i=1}^N$  of $Q$ with $|\mu^{(i-1)}/\mu^{(i)}|=Q[i]$, for $i=1,\dots,N$, $$\lambda=\mu^{(0)}\supset_{vert}\mu^{(1)}\supset_{vert}\cdots \supset_{vert} \mu^{(N)}=\mu,$$
 $\underline\t=(t_0^{(N)},\dots,t_0^{(1)})$ is the slack of $Q$, $\underline\r=[\r^{(N)},\cdots,\r^{(2)}, \r^{(1)}]$ is the slack vector sequence of $Q$ that encodes the slackness of its sequence of vertical strips, ${\mu^{(i)}}'=\mu^{(i)}-\delta_{\r^{(i)}}$, $i=1,\dots,N$, as defined  in \cite{azslack} and to be recalled and expanded in the next two subsections, $\rm{red}_{t_0^{(i)}}^{-1}$ is the inverse of the reduction map on $SpT_{2n}(\varpi_{t^{(i)}_0})$, or the expanding map via the reverse removal, defined in Lemma \ref{lem:reductionsurj}, Theorem \ref{lem:reductionsurj2}  and Theorem \ref{thm:nofactors}, (see \cite{azreduction} for a complete account  on the inverse  reduction map),
  and $\c$ is the concatenation in the plactic monoid.  The basic operation is defined for $N=1$, $\underline\t=(t^{(1)}_0)$, in Theorem \ref{thm:verygeneralstrip} in the next subsection,
$$\c \circ (\rm{red}_{t_0}^{-1},\rm{id})\circ (\;,\underset{\mu^{(1)}/{{\mu^{(1)}}'}}\leftarrow S).$$

It consists of the composition of the reverse Schensted column insertion,  Subsection \ref{subsec:insert}, applied to the $t_0$ rows of the symplectic tableau $S\in SpT_{2n}(\mu)$,  indexed in the slack vector $\r^{(1)}$,   followed with the inverse reduction map \cite{azreduction}  of the bumped symplectic column, and lastly the concatenation $\c$ of the previous object with $(\underset{\mu^{(1)}/{{\mu^{(1)}}'}}\leftarrow S)$ the remaining tableau after the action of the reverse Schensted column insertion on $S$. Note that $(\;,\underset{\mu^{(1)}/{{\mu^{(1)}}'}}\leftarrow S)$ is reminiscent of reverse Pieri's combinatorial rule, Subsection \ref{subsec:insert}, \eqref{keystepinverseLRAII}, with the difference that here $\mu^{(1)}/{{\mu^{(1)}}'}$ is just a vertical strip whose row coordinates  are recorded by $\r^{(1)}$ or $\delta_{\r^{(1)}}$.

\subsection{ $1$-vertical strip quantum recording tableaux, slacks and  slack  vectors }\label{subsec:1verticalstrip}
%Let us consider the set $Rec_{2n}(\lambda/\mu)$ of recording tableaux  of shape $\lambda/\mu$  a vertical strip. Then %$\lambda=(\lambda_0,\varpi_{l-\ell(\mu)})$ for some partition

We now recall from \cite{azslack} some statistics for a vertical strip , called \emph{slack} \emph{data} or \emph{slack} \emph{statistics} of a vertical strip $\lambda/\mu$, with $\lambda\in Par_{\le 2n}$ and   $\mu\in Par_{\le n}$.
\begin{defi}\label{slackstrip}Let $\mu\subset_{vert}\lambda$  with $\lambda\in Par_{\le 2n}$  and   $\mu\in Par_{\le n}$. Let $0\le l_0\le\ell(\mu)$ be the number of cells in the vertical strip $\lambda/\mu$ with row coordinates in $[1,\ell(\mu)]$.
%Let $t_0=\ell(\mu)-l_0$ where $l_0$ is the number of cells in the vertical strip $\lambda/\mu$ with row coordinates in $[1,\ell(\mu)]$.
 We call  $t_0=\ell(\mu)-l_0$ the \emph{slack} of the vertical strip $\lambda/\mu$.
 \end{defi}
In other words, the  slack number $t_0$ is the number of gaps of the vertical strip $\lambda/\mu$ in the interval $[1,\ell(\mu)]$. See Example \ref{ex:q}.

\begin{defi}\label{slackvectorstrip} With the same setting as above, let $t_0$ be the slack of the vertical strip $\lambda/\mu$.
We define  $\mathbf{r}=\{r_1<\dots< r_{t_0}\}\subseteq [1,\ell(\mu)]$ to be the complement of the set of the row coordinates of the  cells of the vertical strip $\lambda/\mu$ in $[1,\ell(\mu)]$. We call to $\r=(r_1<\dots< r_{t_0})$ the \emph{slack} \emph{row} \emph{index} \emph{vector} of the vertical strip $\lambda/\mu$, and define $\delta_\mathbf{r}\in\{0,1\}^{\ell(\mu)}$ to be  the  \emph{incidence} \emph{vector} of $\mathbf{r}$ by considering $\r$ as a \emph{subset}   of $ [1,\ell(\mu)]$, that is, $\delta_\r=(x_s)_{s\in [1,\ell(\mu)]}$, $x_s=1$ if $s\in\r$ and $0$ otherwise. We call to $\delta_\mathbf{r}$ the \emph{slack} \emph{incidence} \emph{vector} of the vertical strip $\lambda/\mu$, and write $|\delta_\r|:=\sum_{i=1}^{t_0}\delta_{r_i}=t_0$.
%We call to $\ell(\mu)-l_0$ the \emph{slack} of the vertical strip $\lambda^0/\mu$ (or the vertical strip  $\lambda/\mu$).
\end{defi}
Indeed the vertical strip $\lambda/\mu$ as a dual incidence vector with respect to the interval $[1,\ell(\lambda)]$. The dual of the slack incidence vector is  equal to $\varpi_{\ell(\lambda)}-\delta_\r$ where here we are adding to $\delta_\r$ a tail of zeroes of length $\ell (\lambda)-\ell(\mu)$, and $\mu+(\varpi_{\ell(\lambda)}-\delta_\r)=\lambda$.

Define $\mu'\subset_{vert} \mu$ such that $\mu'_i=\mu_i-1$ if $i\in\r=\{r_1,\dots, r_{t_0}\}$ and $\mu_i$, otherwise. That is $\mu'=\mu-\delta_\mathbf{r}\in Par_{\le n}$. Then the vertical strip $\mu/\mu'$ has $t_0=\ell(\mu)-l_0$ cells with row coordinates   $\r=(r_1<\dots< r_{t_0})$ the  \emph{slack} \emph{row} \emph{index vector} of vertical strip $\lambda/\mu$; and \begin{align}\label{dualslack}\lambda=\mu+(\varpi_{\ell(\lambda)}-\delta_\mathbf{r})=\mu'+\varpi_{\ell(\lambda)}\Leftrightarrow \varpi_{\ell(\lambda)}=\lambda-\mu'.
\end{align}

 The slack data is illustrated in Example \ref{ex:q}.

Let us write $\lambda=(\lambda_0,\varpi_{l-\ell(\mu)})$ where
 $\mu\subseteq_{vert}\lambda^0\in Par_{\le n}$ such that $\ell(\lambda^0)=\ell(\mu)$ and $l=\ell(\lambda)$. Thus,  $\lambda^0/\mu$ is a vertical strip with $l_0=|\lambda_0|-|\mu|\le \ell(\mu)$ cells, and slack $t_0=\ell(\mu)-l_0$.

  Let $Q\in
\widetilde Rec_{2n}((\lambda_0,\varpi_{l-\ell(\mu)})/\mu)$ defined by the vertical strip sequence $\mu^{(0)}=\lambda\supset_{vert}\mu^{(1)}=\mu$. Then $Q[1]+t_0=l$. Furthermore,
the  conditions $(R3)$, $(R4)$, in  Proposition %\ref{prop:tilderec},
impose relations between $l=\ell(\lambda)$ and the  slack $t_0$ of $\lambda/\mu$,
\begin{align}\label{conditions}&l\le 2n-t_0\Leftrightarrow l+t_0\le 2n\Leftrightarrow t_0\le 2n-l,\\
&~\mbox{ and }\nonumber\\
&~0\le l-t_0 \in 2\mathbb{Z}.
\end{align}
%\begin{align}\label{conditions}&l\le 2n-\ell(\mu)+l_0\Leftrightarrow l+\ell(\mu)-l_0\le 2n\Leftrightarrow \ell(\mu)-l_0\le %2n-l,\\
%&~\mbox{ and }\nonumber\\
%&~0\le l-\ell(\mu)+l_0 \in 2\mathbb{Z}.
%\end{align}

We  define the reverse column Schensted insertion  data of $Q$ to be %the vertical strip $\mu/\mu'$ or
the \emph{slack row} \emph{index} \emph{vector} $\r$. For simplicity, we often abuse notation and identify the \emph{slack row} \emph{index} \emph{vector} $\r$ with $\mu^{(1)}/{\mu^{(1)}}'$ and use the slack row index vector $\r$ to apply reverse column insertion to an $S\in SpT_{2n}(\mu)$. We also often  say the \emph{slack} $t_0$, the \emph{slack row} \emph{index} \emph{vector} $\r$ and the \emph{slack} \emph{incidence} \emph{vector} $\delta_\r$ of $Q$ in the sense that the vertical strip $\lambda/\mu^{(1)}$ defines $Q$.

Consequently, the $1$-vertical strip recording tableau of shape $\lambda/\mu$ with slack $t_0$ and $l=\ell(\lambda)$ is  characterized as
\begin{align} \label{rec:1vertical} \widetilde Rec_{2n}((\lambda_0,\varpi_{l-\ell(\mu)})/\mu)=\left\{Q=\YT{0.15in}{}{
{{},{},{},},
 {{},{},{}1},
 {\vdots,{}},
 {{},1},
 {{}},
 {{1}},
 {{1}},
 {{\vdots}},
 {{1}},
 {{1}},
},~\ell(\mu)\le n,~~ 0\le l- t_0\in 2\Z,~~ l\le 2n-t_0\right\}
\end{align}

%Let $\mu\subset_{vert}\lambda^0\in Par_{\le n}$ such that $\ell(\lambda^0)=\ell(\mu)$ and $\lambda^0/\mu$ is a vertical strip of length $l_0=|\lambda_0|-|\mu|$.
%Let $S\in Sp_{2n}(\mu)$ and  $Q\in \widetilde Rec_{2n}((\lambda^0,\varpi_{l-\ell(\mu)})/\mu)$ such that

%\begin{align}\label{conditions}l\le 2n-\ell(\mu)+l_0\Leftrightarrow l+\ell(\mu)-l_0\le 2n, ~\mbox{ and }~0\le %l-\ell(\mu)+l_0 \in 2\mathbb{Z}.
%\end{align}

\begin{ex}\label{ex:q}
For $n=6$, let
\begin{align*}& Q=\YT{0.15in}{}{
 {{},{},{},{}},
 {{},{},,1},
 {{},{},},
 {{},{},},
 {{},{1}},
 {{}},
 {{1}},
 {{1}},
}\in \widetilde Rec_{2n}((\lambda_0,\varpi_{l-\ell(\mu)})/\mu),\nonumber
\end{align*}
where $ \mu=\mu^{(1)}=(4,3,3,3,1,1)\subset_{vert}\lambda^0=(4,4,3,3,2,1)\subseteq \lambda=(\lambda_0,\varpi_{l-\ell(\mu)})=(\lambda_0,\varpi_{2})$, and $\ell(\lambda^0)=\ell(\mu)=6$, $ l_0= 2$, slack $t_0=\ell (\mu)-l_0=4$, $Q[1]+t_0=\ell(\lambda)=8\le 12-4$.

One has ${\mu^{(1)}}'=\mu^{(1)}-\delta_\mathbf{r}=\mu-(1,0,1,1,0,1)=(3,3,2,2,1,0)$ where $\mathbf{r}=\{1,3,4,6\}\subseteq [1,\ell(\mu)]\subseteq [1,\ell(\lambda)]=[1,8]$ is the slack row index vector of the vertical strip $\lambda/\mu^{(1)}$, and
%\begin{align}
$$\begin{array}{llllll}\lambda/\mu^{(1)}={\tiny\ydiagram{0,3+1,0,0,1+1,0,0+1,0+1}}&\quad \mu^{(1)}/{\mu^{(1)}}'={\tiny\ydiagram{3+1,0,2+1,2+1,0,0+1}}&\qquad \r=\YT{0.15in}{}{
 {1},
 {3},
 {4},
 {6},
}&\qquad  \delta_\r=\YT{0.15in}{}{
 {1},
 {0},
 {1},
 {1},
 {0},
 {1}
}
\end{array}$$
%\end{align}
\end{ex}
%For simplicity, we often abuse notation and identify the slack row index vector $\r$ with $\mu/\mu'$.
Indeed $\varpi_{\ell(\lambda)}=\lambda-{\mu^{(1)}}'$.

Let $S\in SpT_{2n}(\mu)$ and $Q\in \widetilde Rec_{2n}((\lambda_0,\varpi_{l-\ell(\mu)})/\mu)$.
Let  $%S'_\r=
(a_1,\dots,a_{t_0})\in Sp_{2n}(\varpi_{t_0})$ be the set of the $t_0$ bumped elements from the first column  of $S$ by applying successively the reverse column Schensted insertion to the  rows of $S$ as prescribed by the \emph{slack row index vector} $\r=(r_1<\dots<r_{t_0})$ of $Q$, going from the largest to the smallest row;   and let $ (\underset{{\r}}\leftarrow S)$ be  the remaining tableau after the application of that reverse column insertion to  $S$ (see Subsection \ref{subsec:insert}).
 %Note that Remark \ref{routes} ensures that $S'_1$ is well defined column.
 By Remark \ref{revroutes}, $(a)$,  for the returned pair $((a_1,\dots,a_{t_0}),\underset{{\r}}\leftarrow S)$ obtained under the action of the reverse Schensted insertion to the entries in the $t_0$ cells  $\mu^{(1)}/{\mu^{(1)}}'$ of $S$, put
 \begin{align}\label{revschensted}((a_1,\dots,a_{t_0}),\underset{{\r}}\leftarrow S))=:(\;,\underset{{\mu^{(1)}/{\mu^{(1)}}'}}\leftarrow S), \mbox{  or $((a_1,\dots,a_{t_0}),\underset{{\r}}\leftarrow S) =:(\;,\underset{{\r}}\leftarrow S)$}.
 \end{align}

Since $S$ is symplectic and $(a_1,\dots,a_{t_0})$ is contained in the first column of $S$,  $(a_1,\dots,a_{t_0})\in SpT_{2n}(\varpi_{t_0})$. (Note any subset of a symplectic column is still symplectic.)

%\begin{obs}\label{reverse-operators}From \eqref{revschensted},
For $Q\in \widetilde Rec_{2n}((\lambda_0,\varpi_{l-\ell(\mu)})/\mu)$ \eqref{rec:1vertical}
\begin{align}\label{tildeR2}{\LRAII^{AII}}^{-1}(S, Q)=\tilde R(S,Q)&= \c \circ (\rm{red}_{t_0}^{-1},\rm{id})\circ (\;,\underset{{\r}}\leftarrow S)\nonumber\\
&=\c\circ (\rm{red}_{t_0}^{-1},\rm{id})((a_1,\dots,a_{t_0}),\underset{{\r}}\leftarrow S) )\nonumber\\
&=\c\circ (\rm{red}_{t_0}^{-1}((a_1,\dots,a_{t_0})),\underset{{\r}}\leftarrow S))\nonumber\\
&:=S^\r
%\nonumber
%&=c\circ (T_0T_1\cdots T_{\ell(\mu)-l_0},S'_{\ge 2})\nonumber\\
%&=(T_0T_1\cdots T_{\ell(\mu)-l_0}).S'_{\ge 2}
\end{align}
where $\c$ means concatenation in the plactic monoid and $\rm{red}_{t_0}^{-1}$ means \emph{reverse or the inverse} \emph{reduction map} applied to a  column in $SpT_{2n}(\varpi_{t_0})$ a symplectic column of length the slack number  $t_0$ of $Q$ according to the rules below. The reduction map $\redu$ on $SST_{2n}(\varpi_l)$ is injective \cite[Proposition 4.3.6, Corollary 4.4.3]{watanabe} and returns symplectic columns of  length  $0\le k\le \min(l,2n-l)$ and $l-k\in 2\Z$. Its surjectivity and inverse is shown in Lemma \ref{lem:reductionsurj}, Theorem \ref{lem:reductionsurj2}, Theorem \ref{thm:nofactors}, Theorem \ref{thm:xAy} and more generally in \cite{azreduction}.
\iffalse
The  inverse $\rm{red}_{t_0}^{-1}$ on $SpT_{2n}(\varpi_{t_0})$

\begin{align} \rm{red}_{t_0}^{-1}: SpT_{2n}(\varpi_{t_0})\rightarrow SST_{2n}(\varpi_{l})
\end{align}
is exhibited below \eqref{inversered} following  the rules %\eqref{suml}, \eqref{verygeneral0}, \eqref{verygeneral200}
in Lemma \ref{lem:reductionsurj}, Theorem \ref{lem:reductionsurj2}  and Theorem \ref{thm:nofactors}.
\fi

Let $(a_1,\dots,a_{t_0})\in SpT_{2n}(\varpi_{t_0})$  be the column of bumped entries  as in \eqref{revschensted}.
Taking into account the inverse  reduction map $\rm{red}_{t_0}^{-1}$ on $SpT_{2n}(\varpi_{t_0})$, Section \ref{sec:inversereduction}, one has the  theorem below.

\begin{thm}\label{thm:verygeneralstrip} With the set up above where $S\in SpT_{2n}(\mu)$ and $Q\in  \widetilde Rec_{2n}((\lambda^0,\varpi_{l-\ell(\mu)})/\mu)$, $l=\ell(\lambda)$, has slack row index vector  $\mathbf{r}=\{r_1<\dots< r_{t_0}\}$, one has the following assertion:

\begin{align}\label{inverse}{\LRAII^{AII}}^{-1}(S,Q)={\widetilde R}(S, Q)&= c \circ (\rm{red}_{t_0}^{-1},\rm{id})\circ (\;,\underset{{\r}}\leftarrow S)\nonumber\\
&=c\circ (\rm{red}_{t_0}^{-1},\rm{id})((a_1,\dots,a_{t_0}),(\underset{{\r}}\leftarrow S) ), \mbox{  by \eqref{noexplicit}}\nonumber\\
&=c\circ (\rm{red}_{t_0}^{-1}(a_1,\dots,a_{t_0}), (\underset{{\r}}\leftarrow S))\nonumber\\
&=\redu_{t_0}^{-1}((a_1,\dots,a_{t_0}))\bigcdot (\underset{{\r}}\leftarrow S)=:S^\r\in SST_{2n}(\lambda).
\end{align}
where
$\redu_{t_0}^{-1}((a_1,\dots,a_{t_0}))\in SST_{2n}(\varpi_{\ell(\lambda)})$  is defined in Section \ref{sec:inversereduction} and \cite{azreduction},
and $(\underset{{\r}}\leftarrow S)\in SpT_{2n}({\mu^{(1)}}')$  with ${\mu^{(1)}}'=\mu^{(1)}-\delta_\r$ and  $\lambda=\mu^{(1)}-\delta_\r+\varpi_{\ell(\lambda)}$.
\end{thm}

%\iffalse
\begin{proof}
It remains to prove that $S^\r$ is a semistandard tableau.
The  entries of the symplectic column $(a_1,\dots,a_{t_0})\in SpT_{2n}(\varpi_{t_0})$ bumped from the first column of $S\in SpT_{2n}(\mu)$, are included in $\redu_{t_0}^{-1}((a_1,\dots,a_{t_0}))$ $\in SST_{2n}(\varpi_{\ell(\lambda)})$. By Remark \ref{revroutes}, the new column $(x_1,\dots,x_{t'_0})$, $0\le t'_0\le t_0$, as a part of the first column of $(\underset{{\r}}\leftarrow S)$  is such that $x_i$ bumped $a_i$ and $x_i\ge a_i$, for $i=1,\dots,t'_0$.

 From, Subsection \ref{subsec:essential}, for $i=1,\dots,t_0$, %$a_i$ belongs to $(a_i)T_i$, and since  the first column of $S$ is a symplectic column  its
the row  coordinate of $a_i$ as an entry of $\redu_{t_0}^{-1}((a_1,\dots,a_{t_0}))$
%$T_0(a_1)T_1\cdots (a_{t_0})T_{t_0}$
is larger or equal than  the row coordinate of $x_i$ (if it exists) because conditions on the formulas in Lemma \ref{lem:reductionsurj}, Theorem \ref{lem:reductionsurj2}  and Theorem \ref{thm:nofactors} and \cite{azreduction} %\eqref{verygeneral0} and \eqref{verygeneral200}
force the push down or keep the $a_i$ row-coordinate as an entry of the first column of $S$.
%{\oo Furthermore, either $a_i$ and $a_{i+1}$ are consecutive entries in $\redu_{t_0}^{-1}((a_1,\dots,a_{t_0}))$ or there are consecutive integers between them
%starting with an odd integer and ending with an even number }
Hence from Subsection \ref{subsec:essential}
%the column insertion of $T_0(a_1)T_1\cdots (a_{t_0})T_{t_0}$ in $S^1$ is just
the concatenation $\redu_{t_0}^{-1}((a_1,\dots,a_{t_0}))\bigcdot(\underset{{\r}}\leftarrow S) \in SST_{2n}(\lambda)$.
\end{proof}

\begin{obs} \label{nullslack}
$(1)$ When  $t_0= 0 \Leftrightarrow l_0=\ell(\mu)\Leftrightarrow \r=()\Leftrightarrow \delta_\r=0$, then $\mu'=\mu$ and the bumped column  is empty $()$ and $S=(\underset{{()}}\leftarrow S)$. That is, $((),S)=(\;,\underset{{\r}}\leftarrow S)$.
If  the slack of  $Q$ is  $t_0=0$, then $l\in 2\Z$ and $\rm{red}_{0}^{-1}$ means reverse reduction in $[1,l]$ applied to a column of length $0$. That is, $\redu_{0}^{-1}(())=(12\dots l)$ with $l\in 2\Z$ $\Leftrightarrow \redu(12\dots l)=()\Leftrightarrow \rm{rem}(12\dots l)=(12\dots l)$ and $l\le 2n$. Hence,

\begin{align}&%{LR^{AII}}^{-1}
\widetilde R=(S, Q)=\c \circ (\rm{red}_{0}^{-1},\rm{id})\circ (\;,\underset{{\r}}\leftarrow S)
=\c\circ (\rm{red}_{0}^{-1},\rm{id})((),S)=\c\circ (\rm{red}_{0}^{-1}(),S)\\
&=\c\circ (\YT{0.15in}{}{
 {{1}},
 {{2}},
 {{\vdots}},
 {{l}},
 },S)=\YT{0.15in}{}{
 {{1}},
 {{2}},
 {{\vdots}},
 {{l}},
 }\bigcdot S=T_0\bigcdot S\in SST_{2n}(\lambda)
 \end{align}

 $\mu'=\mu-\delta_{()}=\mu$ and  $\lambda=\mu+\varpi_{\ell(\lambda)}$.

$(2)$
When $l_0=0\Leftrightarrow t_0=\ell(\mu)\Leftrightarrow\r=[1,\ell(\mu)]\Leftrightarrow\delta_\r=(1^{\ell(\mu)})$, then $\mu'=\mu-(1^{\ell(\mu)})$ and the bumped column $(a_1,\dots,a_{\ell(\mu)})\in SpT_{2n}(\varpi_{\ell(\mu)})$ is the first column of $S$. That is,  $((a_1,\dots,a_{\ell(\mu)}), S^1)=(\;,\underset{{\r}}\leftarrow S)$ where $S^1$ is the symplectic tableau $S\in SpT_{2n}(\mu)$ minus its first column $(a_1,\dots,a_{\ell(\mu)})$. Hence
 \begin{align}&%{LR^{AII}}^{-1}
&\widetilde R (S, Q)=\c \circ (\redu_{\ell(\mu)}^{-1},\rm{id})\circ (\;,\underset{\r}\leftarrow S)
=c\circ (\rm{red}_{\ell(\mu)}^{-1},\rm{id})((a_1,\dots,a_{\ell(\mu)}),S^1)=\rm{red}_{\ell(\mu)}^{-1}((a_1,\dots,a_{\ell(\mu)}))\bigcdot S^1
\end{align}
In this case, $l-\ell(\mu)\in 2\Z$ and $l\le 2n-\ell(\mu)$, and  $\rm{red}_{\ell(\mu)}^{-1}((a_1,\dots,a_{\ell(\mu)}))\in SST_{2n}(\varpi_l)$.

Let $n=4$, $S=\YT{0.13in}{}{
 {{1},{2},5},
 {4,4},
 {5},
 }
 \in SpT_8(3,2,1)$ and $Q=\YT{0.13in}{}{
 {{},{},},
 {,},
 {{}},
 {1},
 {1},
 }\in \widetilde Rec_8((\mu,1,1)/\mu)$, $t_0=3$, $\r=[1,3]$, $l=5$, $l-\ell(\mu)=2\in 2\Z$, $l=5\le 2n-\ell(\mu)=8-3$.
 Then
 $$\rm{red}_{\ell(\mu)}^{-1}((a_1,\dots,a_{\ell(\mu)}))\bigcdot S^1=\rm{red}_{\ell(\mu)}^{-1}(145)\bigcdot S^1=14578\bigcdot\YT{0.13in}{}{
 {{2},5},
 {4},
 },\quad T_0T_1T_2T_3=()\bigcdot 1\bigcdot 4 \bigcdot 5 78\in SST_8(\varpi_5).
 $$

 For $S=\YT{0.13in}{}{
 {{1},{2},5},
 {6,6},
 {7},
 }
 \in SpT_8(3,2,1)$ and the same recording tableau $Q$

 $$%\rm{red}_{\ell(\mu)}^{-1}(S_1)\bigcdot S_{\ge 2}=
 \rm{red}_{\ell(\mu)}^{-1}(167)\bigcdot S^1=1 3467\bigcdot\YT{0.13in}{}{
 {{2},5},
 {6},
 },\quad T_0T_1T_2T_3=()\bigcdot 1 34 \bigcdot 5\bigcdot 7\in SST_8(\varpi_5)
 $$
\end{obs}
 For an illustration of Theorem \ref{thm:verygeneralstrip} see \cite{azslack}.

\subsection{Quantum recording tableaux and slack data}\label{sec:slackdata}

We now  recall and expand from \cite{azslack}  the \emph{slack} \emph{incidence} \emph{matrix} of a  tableau $Q\in \widetilde Rec_{2n}(\lambda/\mu)$ which encodes the information as described in  Proposition \ref{prop:slack} below to define ${\LRAII^{AII}}^{-1}$ on $SpT_{2n}(\mu)\times \widetilde Rec_{2n}(\lambda/\mu)$. We follow closely \cite[Section 3.4]{azslack}.

 \begin{defi}\label{def:tei} Let
 $Q\in  \widetilde Rec_{2n}(\lambda/\mu)$  defined by the sequence of vertical strips
\begin{align}\label{slack}&\mu^{(0)}=\lambda\supset_{vert} \mu^{(1)}\supset_{vert} \cdots\supset_{vert}\mu^{(N-1)}\supset_{vert}\mu^{(N)} =\mu
\end{align}
such that $2\le |\mu^{(i-1)}/\mu^{(i)}|=Q[i]\in 2\Z$, $1\le i\le N$  and $\nu=( Q[1],\dots,Q[N])^t$ is an even partition. For   $1\le i\le N$,
let $t_0^{(i)}=\ell(\mu^{(i)})-l_0^{(i)}$ where $l_0^{(i)}$ is the number of cells in the vertical strip $\mu^{(i-1)}/\mu^{(i)}$ with row coordinates in $[1,\ell(\mu^{(i)})]$. Then,  for $i=1,\dots,N$,
\begin{enumerate}
\item $t^{(i)}_0\in\{0,1,\dots,\ell(\mu^{(i)}\}$ is the \emph{slack} of the \emph{vertical}  \emph{strip} $\mu^{(i-1)}/\mu^{(i)}$, and
    \item $\r^{(i)}\subseteq [1,\ell(\mu^{(i)})]$ with cardinal $|\r^{(i)}|=t_0^{(i)}$,
   the corresponding \emph{slack} (\emph{row} \emph{index}) \emph{vector}.   We write $\r^{(i)}=()$ or $ \emptyset$ when $t_0^{(i)}=0$, and in this case $\ell(\mu^{(i)})=Q[i]\in 2\Z$.
\end{enumerate}
The sequence of \emph{slack} \emph{numbers} $\underline\t=(t_0^{(N)},\dots,t_0^{(1)})$,  is called  the \emph{slack} \emph{sequence} of $Q$  uniquely defined by \eqref{slack}. Analogously, the corresponding sequence of slack (row index) vectors $\underline\r=[\r^{(N)},\dots,\r^{(1)}]$ is called the \emph{slack} (row index) \emph{vector} \emph{sequence} of $Q$.
\end{defi}

We have defined $\delta_{\r^{(i)}}\in \{0,1\}^{\ell(\mu^{(i)})}$, that is, considering  $\r^{(i)}\subseteq [1,\ell(\mu^{(i)})]$, $i=1,\dots,N$. Since $[1,\ell(\mu)]\subseteq [1,\ell(\mu^{(N-1)})] \subseteq \cdots \subseteq [1,\ell(\mu^{(1)})]\subseteq [1,\ell(\lambda)]$, we may consider
$\delta_{\r^{(i)}}\in \{0,1\}^{\ell(\lambda)}$ by adding to $\r^{(i)}$ a tail of $\ell(\lambda)-\ell(\r^{(i)})$ zeroes.
In this sense the  $\ell(\lambda)\times N$ matrix  $\delta_{\underline\r}=[\delta_{\r^{(N)}},\dots,\delta_{\r^{(1)}}]$ is called  the \emph{slack} \emph{incidence} \emph{matrix} of $Q$. We also write $|\delta_{\underline \r}|:=\sum_{i=1}^N|\delta_{\r^{(i)}}|=\sum_{i=1}^N t_0^{(i)}$.

Given the index matrix $\delta_{\underline\r}$ of a slack vector sequence $\underline\r=[\r^{(N)}\le_\r\cdots\le\r^{(1)}]$ of $Q\in  Rec_{2n}(\lambda/\mu)$, we define the submatrix $\delta_{\underline\r^+}$ to be the largest submatrix $\ell(\mu)\times k$ with $N\ge k\ge 1$, of $\delta_{\underline \r}$ such that $\mu-\delta_{\underline\r^+}\in \mathbb{Z}_{\ge 0}$, where $\underline\r^+=[\r^{(N)},{\r^{(N-1)}}',\dots,{\r^{(k)}}']$ and  ${\r^{(i)}}'\subseteq \r^{(i)}$, $ k\le i< N$. For $N=1$, $\underline\r^+=\underline\r$. We call to $\underline \r^+$ the \emph{nonnegative} \emph{part} of $\underline\r$

 From the slack incidence matrix,  one has $\mu^{(i-1)}=\mu^{(i)}-\delta_{\r^{(i)}}+\varpi_{\ell(\mu^{(i-1)})}$ with $\mu^{(i)}-\delta_{\r^{(i)}}\subset_{vert}\mu^{(i)}\in Par_{\le \ell(\mu^{(i)})}$, for $i=1,\dots,N$. (A generalization of \eqref{dualslack}.) Henceforth,
for $i=1,\dots,N$,
\begin{align}&\mu^{(i-1)}=\mu-\sum_{k=i}^N\delta_{\r^{(k)}}+\sum_{k=i}^N\varpi_{\ell(\mu^{(k-1)})}\\
&\mbox{ and }\nonumber\\
&\lambda=\mu-\sum_{k=1}^N\delta_{\r^{(k)}}+\sum_{k=1}^N\varpi_{\ell(\mu^{(k-1)})}\nonumber\\
&=\mu-\delta_{\underline\r}+\sum_{k=1}^N\varpi_{\ell(\mu^{(k-1)})},\nonumber\\
&=(\mu-\delta_{\underline\r^+})+\sum_{\r^{(i)}\notin \r^{+}}\delta_{\r^{(i)}\setminus {\r^{(i)}}'}+\sum_{k=1}^N\varpi_{\ell(\mu^{(k-1)})}
\end{align}
where we convention $\mu-\delta_{\underline\r}:=\mu-\sum_{k=1}^N\delta_{\r^{(k)}}$.

This slack data is illustrated in the next example and in Example \ref{ex:vectorslackn=3}. For additional illustrations see \cite{azslack}.

\begin{ex}In Example \ref{ex:q}, $Q\in \widetilde Rec_{12}(\lambda/\mu)$, $N=1$, $\mu^{(0)}=\lambda$, $\mu^{(1)}=\mu$ and $t_0^{(1)}=\ell(\mu)-l^{(1)}_0=6-2=4$ with $l_0^{(1)}=2$, and $Q[1]=4$. Thus $\ell(\mu^{(0)})=\ell(\lambda)=Q[1]+t_0^{(1)}=4+4$ and $Q[1]+2t_0^{(1)}\le 12$. The slack row index vector is $\r=(1,3,4,6)$, and the $\ell(\mu)\times 1$ slack incidence  matrix is $[\delta_{\r^{(1)}}]=[1\;0\;1\;1\;0\;1]^T$.

\end{ex}

Given $Q\in  \widetilde Rec_{2n}(\lambda/\mu)$, Definition \ref{def:tilde}, $(R3)$, $(R4)$, impose conditions to the slack sequence  of  $Q$, and  $(R5)$ impose conditions to its  \emph{slack} \emph{vector} \emph{sequence} $\underline\r=[\r^{(N)},\dots,\r^{(1)}]$ equivalently to the slack incidence matrix $\delta_{\underline \r}=[\delta_{\r^{(N)}},\dots,\delta_{\r^{(1)}}]$ of $Q$.

\begin{defi}\label{slackorder} Let $n\in \N$ and $0\le m\le 2n$. Given $x=(x_1<\cdots<x_p)\in [1,m]^p$ and $y=(y_1<\cdots<y_q)\in[1,m]^q$ for some $2n\ge p\ge q\ge 0$, we write $x\le_{\r} y$ to mean  $x_i\le y_i$ whenever $i\in[1,q]$. For $q=0$, $x\le_\r ()$. This means that $x\le_{\r} y$ if and only if  the $p\times 2$ incidence matrix  $[\delta_x,\delta_y]$ is such that for any $i\ge 1$, the sum of the first $i$ entries in the first column is always larger or equal than the  sum of the first $i$ entries in the second column.
\end{defi}
{ Let $n\in N$, the relation $\le_{\r} $, called \emph{slack relation},  defines a partial order in the set $2^{[1,2n]}$ by declaring that  $x,y\subseteq [1,2n]$ satisfy $x\le_{\r} y$ if the relation $\le_{\r} $ holds for $x$ and $y$ as strictly increasing vectors.}
See Example \ref{ex:vectorslackn=3}.

\begin{prop}\cite{azslack}\label{prop:slack} Let
 $Q\in  \widetilde Rec_{2n}(\lambda/\mu)$ with vertical strip decomposition \eqref{slack}, and put $t^{(0)}_0:=0$ and $\r^{(0)}:=()$. Then, for each $1\le  i\le N$,

\begin{itemize}
%\item[(a)]$\ell(\mu^{\nu_1-1)})=Q[\nu_1]+t_{\nu_1}$ for some $0\le t_{\nu_1}\le \ell(\mu)$, and

\item[(a)]
 $\ell(\mu^{(i-1)})=Q[i]+t_0^{(i)}\ge \ell(\mu^{(i)}$.

\item[(b)] $0\le t^{(i-1)}_0\le t^{(i)}_{0} \le \ell(\mu^{(i)})\le \ell(\mu^{(i-1)})$. The slack sequence $\underline\t=(t^{(N)}_0\ge \dots \ge t^{(1)}_0)$ is weakly decreasing while the sequence $(\ell(\mu^{(N)})\le \dots \le \ell(\mu^{(1)})\le \ell(\mu^{(0)}) )$ is weakly increasing.

\item[(c)] $2n\ge Q[i]+2t_0^{(i)}=\ell(\mu^{(i-1)})+t_0^{(i)}\Leftrightarrow 2n-t_0^{(i)}\ge Q[i]+t_0^{(i)}=\ell(\mu^{(i-1)})$.

\item[(d)] $\r^{(i)}\le_\r \r^{(i-1)}$. The slack vector  sequence $\underline\r=[\r^{(N)}\le_\r\cdots\le_r\r^{(1)}]$ is weakly increasing.
    \item[(e)] $\r^{(i)}\subseteq[1,\ell(\mu^{(i)}]\subseteq [1,\ell(\mu^{(i-1)}]\subseteq[1,2n-t_0^{(i)}]$.
\end{itemize}
\end{prop}
%\begin{proof} $(\rm d)$ follows from Definition \ref{def:tei} and $(\rm a),(\rm b),(\rm c)$.
%\end{proof}

Condition $(\rm d)$ above guarantees that reverse Schensted column insertion can  correctly be $\r$-iterated to define ${\LRAII^{AII}}^{-1}$ as required in {Remark} \ref{revroutes} $(b)$, $(c)$.

%\begin{prop}\label{iterationpieri}
%\end{prop}

\iffalse
\begin{obs} \label{re:slack} \begin{enumerate}

\item When the slack sequence of $Q$ is the null vector then the conditions on $Q$ are just $\ell(\mu^{(i-1)})=Q[i]\ge \ell(\mu^{(i)})$ and
$2n\ge Q[i]$, for $i=1,\dots, N$. Therefore, $\lambda=\mu+\sum_{i=1}^N\varpi_{Q[i]}$.

\item The admissible slacks for  $n=1,2, 3$ are as follows. Let $N\ge 1$. Proposition \ref{prop:slack} $(b)$ forces $ t_0^{(i)}=0\Rightarrow t_0^{(j)}=0$ for all $i\ge j$.

For $n=1$,  Proposition \ref{prop:slack} $(c)$ implies $2\ge Q[i]+2t_0^{(i)}\ge 2 +2t_0^{(i)}\Leftrightarrow t_0^{(i)}=0$. Hence $Q[i]=2$ and $t_0^{(i)}=0$.

For $n=2$,  Proposition \ref{prop:slack} $(c)$ implies $4\ge Q[i]+2t_0^{(i)}\ge 2 +2t_0^{(i)}\Leftrightarrow 2\ge 2t_0^{(i)}\Leftrightarrow 1\ge t_0^{(i)}$.
Hence either $Q[i]=4$ and $t_0^{(i)}=0$ or $Q[i]=2$ and $t_0^{(i)}=0,1$. Note, $t_0^{(i)}=1\Rightarrow 2\times 2-1\ge Q[i]+1\Rightarrow Q[i]=2$. In particular $4\notin \r^{(i)}$.

For $n=3$,  Proposition \ref{prop:slack} $(c)$ implies $6\ge Q[i]+2t_0^{(i)}\ge 2 +2t_0^{(i)}\Leftrightarrow 2\ge t_0^{(i)} $.
Hence either $Q[i]=6$ and $t_0^{(i)}=0$ or $Q[i]=4$ and $t_0^{(i)}=0,1$ or { $Q[i]=2$} and $t_0^{(i)}=0,1,2$.  In particular $6\notin \r^{(i)}$.
\end{enumerate}
\end{obs}
\fi

\begin{ex}\label{ex:vectorslackn=3}

\begin{enumerate}

\item Let $n=3$.
Let $V=\YT{0.2in}{}{
 {1},
 {4},
 {5},
} \in SpT_6(\varpi_3)$, $\underline \t=(2,2,1,1,0)$, $\underline\r=[(2,3);(3,4);4;5;()]$ and $$Q=\YT{0.12in}{}{
 {,5,4,3,2,1},
 {,4,3,2,1},
 {,3,2,1},
 {5,2,1},
 {3,1},
 {1},
} \in \widetilde Rec_{6}((6,5,4,3,2,1)/(1,1,1)),$$
$$\rm\delta_{\underline\r}=\begin{bmatrix}
0&0&0&0&0\\
1&0&0&0&0\\
1&1&0&0&0\\
0&1&1&0&0\\
0&0&0&1&0\\
0&0&0&0&0\\
\end{bmatrix}_{6\times 5},\quad \r^{(5)}=(2,3)\le_{\r}\r^{(4)}=(3,4)\le_{\r}\r^{(3)}=4\le_{\r}\r^{(2)}=5\le_{\r}\r^{(1)}=().$$

Let $\mu=\varpi_3$. In this case $\underline\r^+=(\r^{(5)}=(2,3))$ and $\mu-\delta_{\underline \r^+}=(1)$.
\begin{align*}& \c \circ (\rm{red}_{2}^{-1},\rm{id})\circ \big(\;,\underset{{(2,3)}}\leftarrow \YT{0.12in}{}{
 {1},
 {4},
 {5},
}\big)
%\nonumber\\
%&
=\c\circ (\rm{red}_{2}^{-1},\rm{id})((4,5), 1)=\redu_2^{-1}(4,5)\bigcdot(1)=\YT{0.12in}{}{
 {1,1},
 {2},
 {4},
 {5},
}=S^{(2,3)},  \\
&\shape(S^{(2,3)})=\mu-\rm\delta_{\r^{(5)}}+\varpi_4 =(1)+\varpi_4\\
\end{align*}
\begin{align*}
&\c \circ (\rm{red}_{2}^{-1},\rm{id})\circ (\;,\underset{{(3,4)}}\leftarrow S^{(2,3)})
%\nonumber\\
%&
=\c\circ (\rm{red}_{2}^{-1},\rm{id})((4,5), \YT{0.12in}{}{
 {1,1},
 {2},
})=\redu_{2}^{-1}(4,5)\bigcdot\YT{0.12in}{}{
 {1,1},
 {2},
}=\YT{0.12in}{}{
 {1,1,1},
 {2,2},
 {4},
 {5},
}=S^{(2,3),(3,4)}\\
&\shape(S^{(2,3),(3,4)})=(\mu-\rm\delta_{\r^{(5)}}+\varpi_4 )-\rm\delta_{\r^{(4)}}+\varpi_4\\
\end{align*}
\begin{align*}
& \c \circ (\rm{red}_{2}^{-1},\rm{id})\circ (\;,\underset{{4}}\leftarrow S^{(2,3),(3,4)})
%\\
%&
=\c \circ (\rm{red}_{2}^{-1},\rm{id})(5,\YT{0.12in}{}{
 {1,1,1},
 {2,2},
 {4},
 }
 )=\redu_1^{-1}(5)\bigcdot\YT{0.12in}{}{
 {1,1,1},
 {2,2},
 {4},
 }\\
 &=
 \YT{0.12in}{}{
 {1,1,1,1},
 {2,2,2},
 {3,4},
 {4},
 {5},
}=S^{(2,3),(3,4),4}\\
&\shape(S^{(2,3),(3,4),4})=(\mu-\rm\delta_{\r^{(5)}}+\varpi_4 )-\rm\delta_{\r^{(4)}}+\varpi_4-\rm\delta_{\r^{(3)}}+\varpi_5
\end{align*}
\begin{align*}
& \c \circ (\rm{red}_{2}^{-1},\rm{id})\circ (\;,\underset{{5}}\leftarrow S^{(2,3),(3,4),4})%\\
%&
=\c \circ (\rm{red}_{2}^{-1},\rm{id})(5,\YT{0.12in}{}{
 {1,1,1,1},
 {2,2,2},
 {3,4},
 {4},
 }
 )=\redu^{-1}(5)\bigcdot S^{(2,3),(3,4),4}\\
 &=\YT{0.12in}{}{
 {1,1,1,1,1},
 {2,2,2,2},
 {3,3,4},
 {4,4},
 {5},
}=S^{(2,3),(3,4),4,5}\\
&\shape(S^{(2,3),(3,4),4,5})=(\mu-\rm\delta_{\r^{(5)}}+\varpi_4 )-\rm\delta_{\r^{(4)}}+\varpi_4-\rm\delta_{\r^{(3)}}+\varpi_5-\rm\delta_{\r^{(2)}}+\varpi_5
\end{align*}
\begin{align*}
& \c \circ (\rm{red}_{0}^{-1},\rm{id})\circ (\;,\underset{{()}}\leftarrow S^{(2,3),(3,4),4,5})=\redu_0^{-1}(())\bigcdot S^{(2,3),(3,4),4,5}\\
&
 =\YT{0.12in}{}{
 {1,1,1,1,1,1},
 {2,2,2,2,2},
 {3,3,3,4},
 {4,4,4},
 {5,5},
 {6},
}=S^{(2,3),(3,4),4,5,()}\in SST_{6}(6,5,4,3,2,1).
\end{align*}
That is,
\begin{align}
S^{(2,3),(3,4),4,5,()}=\nonumber\\
=&\redu_0^{-1}(())\bigcdot\redu_1^{-1}(5)\bigcdot(\redu_1^{-1}(5)\setminus \{5\})\bigcdot(\redu_2^{-1}(4,5)\setminus \{5\})\bigcdot(\redu_2^{-1}(4,5)\setminus \{4,5\})\bigcdot(1),
\end{align}
where $(1)=\mu-\delta_{\r^{(5)}}$.

The output tableau is a $\k$-lowest weight tableau in $SST_{6}(6,5,4,3,2,1)$ of $\k$-weight $-\varpi_3$  as explained in Section \ref{sec:last}.
The shape of $S^{(2,3),(3,4),4,5,()}$ is

\begin{align*}\lambda&=(((\mu-\rm\delta_{\r^{(5)}})+\varpi_4)-\rm\delta_{\r^{(4)}}+\varpi_4)-\rm\delta_{\r^{(3)}}+\varpi_5)-\rm\delta_{\r^{(2)}}
+\varpi_5-\rm\delta_{\r^{(1)}}+\varpi_6\\
&=\mu-\sum_{i=1}^5\delta_{\r^{(i)}}+2\varpi_4+2\varpi_5+\varpi_6=(\mu- \delta_{\r^{+}})+\sum_{\r^{(i)}\notin \r^{+}}\delta_{\r^{(i)}}+2\varpi_4+2\varpi_5+\varpi_6
\end{align*}

\item Let $n=4$. Let $V=\YT{0.2in}{}{
 {1,1},
 {4},
 {5},
} \in SpT_8(2,1,1,0)$, $\ell(\mu)=3\le 4$, $\underline \t=(2,2)$, $\underline\r=[(1,3);(2,6)]$ and
\begin{align}&Q=\YT{0.12in}{}{
 {,,1},
 {,2},
 {,1},
 {2,1},
 {2,1},
 {2},
} \in \widetilde Rec_{8}((3,2,2,2,2,1)/(2,1,1)),\quad\rm\delta_{\underline\r}=\begin{bmatrix}
1&0\\
0&1\\
1&0\\
0&0\\
0&0\\
0&1\\
0&0\\
0&0
\end{bmatrix}_{8\times 2},\quad \underline\r=[\r^{(2)}=(1,3)\le_{\r}\r^{(2)}=(2,6)],
\end{align}
$\underline\r^+=[\r^{(2)}=(1,3),{\r^{(2)}}'=(2)]$ and $\rm\delta_{\underline\r^+}=\begin{bmatrix}
1&0\\
0&1\\
1&0\\
\end{bmatrix}_{\ell(\mu)\times 2}$, $\mu-\delta_{\underline\r^+}=(1)$.

\begin{align*}& \c \circ (\rm{red}_{2}^{-1},\rm{id})\circ \big(\;,\underset{{(1,3)}}\leftarrow \YT{0.12in}{}{
 {1,1},
 {4},
 {5},
}\big)
%\nonumber\\
%&
=\c\circ (\rm{red}_{2}^{-1},\rm{id})((1,5), (1,4))=\redu_2^{-1}(1,5)\bigcdot(1,5)=\YT{0.15in}{}{
 {\circled{1},1},
 {3,4},
 {4},
 {\circled{5}},
 {7},
 {8},
}=S^{(1,3)},  \\
%&\shape(S^{(1,3)})=\mu-\rm\delta_{\r^{(5)}}+\varpi_4 =(1)+\varpi_4\\
\end{align*}

\begin{align*}
&\c \circ (\rm{red}_{2}^{-1},\rm{id})\circ (\;,\underset{{(2,6)}}\leftarrow S^{(1,3)})
%\nonumber\\
%&
=\c\circ (\rm{red}_{2}^{-1},\rm{id})((4,8), \YT{0.14in}{}{
 {1,1},
 {3},
 {4},
 {5},
 {7},
 })=\redu_{2}^{-1}(4,8)\bigcdot\YT{0.15in}{}{
 {1,1},
 {3},
 {4},
 {5},
 {7},
}=\YT{0.15in}{}{
 {1,\circled{1},1},
 {2,3},
 {\circled{4},4},
 {5,\circled{5}},
 {6,7},
 {\circled{8}},
}=S^{(1,3),(2,6)}\\
&\shape(S^{(1,3),(2,6)})=\\
&=(\mu-\delta_{\underline\r^+})-\rm\delta_{\r^{(2)}\setminus {\r^{(2)}}'}+\varpi_6 +\varpi_6
\end{align*}
That is,
\begin{align}
S^{(1,3),(2,6)}=(\redu_2^{-1}(4,8))\bigcdot(\redu_2^{-1}(1,5)\setminus \{8\})\bigcdot(1).
\end{align}
\end{enumerate}
\end{ex}

The following statement considers the case where no bumping is needed, that is, the $0$-slack sequence is the null vector, $\underline\t=(t^{(N)}_0\ge \dots \ge t^{(1)}_0)=\underbrace{(0,\dots,0)}_N$. In this case, then the \emph{}{slack} \emph{vector} {sequence} is written $\underline\r=(\underbrace{(),\dots,()}_N)$. When the slack number is $0$ the corresponding slack vector is just written $\r=()$. It follows from Proposition \ref{prop:slack} and the iteration of Remark \ref{nullslack}.

 \begin{lem}  \label{lem:recordhole0} Let $S\in SpT_{2n}(\mu)$ and  let $Q\in  \widetilde Rec_{2n}(\lambda/\mu)$ such that $\lambda=\mu+\sum_{i=1}^{\nu_1}\varpi_{Q[i]}$
 with $Q[1]\ge\cdots\ge Q[N]\ge \ell(\mu)$  and $\nu=( Q[1],\dots,Q[N])^t$ an even partition.

Then \begin{align}%{\LRAII^{AII}}^{-1}
\widetilde R(S, Q)=Y(\nu)\bigcdot S \in SST_{2n}(\lambda)
\end{align}
where $Y(\nu)$ is the Yamanouchi tableau of shape $\nu$.
\end{lem}

Next one considers constant $1$-$0$-slack  sequences   $\underline\t=(t^{(N)}_0\ge \dots \ge t^{(1)}_0)=(1,1,\dots,1,0,\dots,0)$ possibly with a tail of zeroes. Henceforth, the slack vector sequence $\underline\r$  can be thought as an  increasing sequence of numbers possibly with a  tail of empty sets corresponding to the sequence of zeroes in $\underline\t$.
  Considering Theorem \ref{thm:verygeneralstrip}, \eqref{inverse}, Proposition \ref{prop:slack}, Lemma \ref{lem:recordhole0}, Remark \ref{revroutes},   and Subsection \ref{subsec:essential},
  %The next theorem is a special case of Subsection \ref{subsec:essential}. Despite of it and due to its simplicity
one asserts %and proves %it independently.
   the next theorem and corollary whose proof of the theorem one details as a warm up for the case of a general recording tableau.

\begin{thm}\label{lem:recordhole1} Let $S\in SpT_{2n}(\mu)$ and let
 $Q\in   \widetilde Rec_{2n}(\lambda/\mu)$   with slack sequence $\underline\t=(\underbrace{1,1,\dots,1}_N)$ and slack vector sequence $\underline \r=[\r_{N}\le\cdots\le \r_1]$.  Then

\begin{enumerate}
\item for $1\le i\le N$,

\begin{itemize}
\item [(i)]$\mu^{(i-1)}=(\mu^{(i)}-\delta_{\r_i})+\varpi_{Q[i]+1}$ is such that
\begin{align*}&\lambda=\mu^{(0)}\supset_{vert} \mu^{(1)}\supset_{vert} \cdots\supset_{vert}\mu^{(N-1)}\supset_{vert}\mu^{(N)} =\mu
\end{align*}
\item[(ii)] $2n-1\ge Q[i]+1=\ell(\mu^{(i-1)})$ and $\ell(\mu^{(i-1)})\notin 2\Z$.
 \item [(iii)] $1\le \r_i\le \ell(\mu^{(i)})\le\ell(\mu^{(i-1)})\le 2n-1$.
%and $\r_i\le 2n-1$. % whenever $\r_i\neq ()$,
\end{itemize}
\item
\begin{align}%{\LRAII^{AII}}^{-1}
\widetilde R(S, Q)=S^{\underline \r}=S^{\r_N\cdots\r_2 \r_1}\in SST_{2n}(\lambda),
\end{align}
where $S^{\r_N\cdots\r_2 \r_1}:=
{({(\cdots{((S^{\r_N})}^{\r_{N-1}}){\cdots})}^{\r_2})}^{\r_1}$ and $S^{\r_N\cdots\r_i}\in SST_{2n}(\mu^{(i-1)})$ for $i=1,\dots,N$.
%, is recursively defined by
%$$S$$
\end{enumerate}

\end{thm}
\begin{proof}
The proof is by induction on $N$. It is enough to analyse the case $N=2$.
Recall $t_0^{(i)}=1$, for $i=1,\dots,N$.
Let  $P^{(N)}:=S$, $Q^{(N)}:=Q$ and define ${\mu^{(N)}}'\subset_{vert} \mu^{(N)}$ such that  ${\mu^{(N)}}'=\mu^{(N)}-\delta_{\r_N}$, that is, ${\mu^{(N)}}'_{r_{N}}=\mu_{r_{N}}-1$  and ${\mu_i^{(N)}}'=\mu_i$, $1\le i\neq r_{N}\le \ell(\mu)$. Then the vertical strip $\mu^{(N)}/{\mu^{(N)}}'$ has  the sole cell with row coordinate $\{\r_{N}\}$. Define
 \begin{align}&P^{(N-1)}:=\c \circ (\rm{red}_{1}^{-1},\rm{id})\circ (\;,\underset{\r_N}\leftarrow S)\nonumber\\\
 &=c \circ (\rm{red}_{1}^{-1},\rm{id})\circ (b_{r'_{N}}, S^{N})\nonumber\\
&=c\circ (\rm{red}_{1}^{-1}(b_{r'_{N}}), S^{N})\nonumber\\\
&=\c\circ(S_{\r_N}, S^{N})=S_{\r_N}\bigcdot S^{N}=:S^{\r_N} \in SST_{2n}(\mu^{(N-1)}), \label{srn}\\
&~ S^N\in SpT_{2n}({\mu^{(N)}}'),~ S_{\r_N}\in SST_{2n}(\varpi_{\ell(\mu^{(N-1)})}),~\mu^{(N-1)}={\mu^{(N)}}-\delta_{\r_N}+\varpi_{\ell(\mu^{(N-1)})} \label{shaperedu}
\end{align}
where $S(r'_{N},1)=b_{r'_{N}}$ is the entry bumped from the cell $(r'_{N},1)$ in the first column of $S$ with

\begin{align*}r'_{N}\ge \r_{N}
\end{align*}
and $S^{N}$  is the tableau obtained from $S$ after applying the reverse column insertion to the cell of $\mu/{\mu^{(N)}}'$ in $S$. Hence $S^{N}$ has shape ${\mu^{(N)}}' $, and, in particular,  $ S^{N}(r'_{N},1)\ge b_{r'_{N}}$ and $S^{N}(i,1)=S(i,1)$, for $i\neq r'_{N}$. Since $S$ and $S^{N}$ are symplectic,
from Proposition \ref{prop:symplectic1}, $(2)$,

 $$\r_{N}\le r'_{N}\le \lfloor \frac{b_{r'_{N}}+1}{2}\rfloor.$$

The expansion
 $\rm{red}_{1}^{-1}(b_{r'_{N}})=S_{\r_N}$ is a column of length $Q[N]+1$ obtained  according to the rule in Theorem \ref{thm:nofactors}

 \begin{align}S_{\r_N}=:
\begin{cases} (1,2,\dots, \mathbf{b}_{r'_{N}}-2, \mathbf{b}_{r'_{N}}, \mathbf{b}_{r'_{N}}+1,\mathbf{b}_{r'_{N}}+2,\dots, \mathbf{b}_{r'_{N}}+(Q[N]-(\mathbf{b}_{r'_{N}}-2))), &  \mathbf{b}_{r'_N}\in 2\mathbb{Z}\\
(1,2,\dots \mathbf{b}_{r'_{N}}-1, \mathbf{b}_{r'_{N}}, \mathbf{b}_{r'_{N}}+1+1,\mathbf{b}_{r'_{N}}+1+2,\dots, \mathbf{b}_{r'_{N}}+1+(Q[N]-(\mathbf{b}_{r'_{N}}-1))),& \mathbf{b}_{r'_{N}}\notin 2\mathbb{Z}.
\end{cases}
\end{align}

 %$\rm{red}_{0}^{-1}$ means reverse reduction in $[1,Q[\nu_1]]$ applied to a column of length $0$. That is, $\rm{red}_{0}^{-1}(())=(12\dots %Q[\nu_1])$ with $Q[\nu_1]\in 2\Z$ $\Leftrightarrow \rm{red}(1,2,\dots, Q[\nu_1])=()\Leftrightarrow \rm{rem}(12\dots Q[\nu_1])=(12\dots %Q[\nu_1])$ providing   $Q[\nu_1]\in 2\Z$ which is the case.

 Let $\mathbf{b}_{r'_{N }}\in 2\mathbb{Z}$. Indeed $S'_{\r_N }(i,1)=i\le S^{N}(i,1)$ for $1\le i\le \mathbf{b}_{r'_{N}}-2$. Note $ S'_{\r_N}(b_{r'_{N}}-1,1)=\mathbf{b}_{r'_{N}}\le S^{N}(r'_{N},1)$.

To ensure that $ S_{\r_N}\bigcdot S^{N}$ is a semi-standard tableau \emph{we} \emph{need} \emph{to} \emph{prove} \emph{that} \emph{the} \emph{index} \emph{row} $\mathbf{b}_{r'_{N}}-1$ \emph{is} \emph{at} \emph{least} $r'_{N}$, that is, $\mathbf{b}_{r'_{N}}-1\ge r'_{N}$ when $r'_{N}= \lfloor \frac{b_{r'_{N}}+1}{2}\rfloor=\frac{b_{r'_{N}}}{2}$. In fact $\mathbf{b}_{r'_{N}}-1\ge \frac{b_{r'_{N}}}{2}\Leftrightarrow \mathbf{b}_{r'_{N}}\ge 2.$

Furthermore, $S_{\r_N}(i,1)=i+1 $ for $i\ge b_{r'_{N}}\ge r'_{N}+1 $, and

\begin{align*}\mathbf{b}_{r'_{N}}&=S_{\r_N}(b_{r'_{N}}-1,1)\\
&\le S^{N}(r'_{N},1)<S^{N}(r'_{N}+1,1)\\
&\Rightarrow \mathbf{b}_{r'_{N}}+1=S_{\r_N}(b_{r'_{N}},1)\\
&\le S^{N}(r'_{N}+1,1)\\
&<S^{N}(r'_{N}+2,1).
\end{align*}
Hence, $\mathbf{b}_{r'_{N}}+k\ge r'_{N}+k+1$, and
$$\mathbf{b}_{r'_{N}}+k+1=S'_{\r_N}(b_{r'_{N}}+k,1)\le S^{N}(r'_{N}+k+1,1),~k\ge 0$$

Let $\mathbf{b}_{r'_{N}}\notin 2\mathbb{Z}$. Indeed $S_{\r_N}(i,1)=i\le S^{N}(i,1)$ for $1\le i\le \mathbf{b}_{r'_{N}}-1$. Note $ S'_{N}(b_{r'_{N}},1)=\mathbf{b}_{r'_{N}}\le S^{N}(r'_{N},1)$.

To ensure that $ S_{\r_N}\bigcdot S^{N}$ is a semi-standard tableau \emph{we} \emph{need} \emph{to} \emph{prove} that the \emph{index} row $\mathbf{b}_{r'_{N}}$ \emph{is} \emph{at} \emph{least} $r'_{N}$, that is, $\mathbf{b}_{r'_{N}}\ge r'_{N}$ when $r'_{N}= \lfloor \frac{b_{r'_{N}}+1}{2}\rfloor=\frac{b_{r'_{N}}+1}{2}$. In fact $\mathbf{b}_{r'_{N}}\ge \frac{b_{r'_{N}}+1}{2}\Leftrightarrow \mathbf{b}_{r'_{N}}\ge 1.$

On the other hand,  $S^{N}(r'_{N},1)=\mathbf{b}_{r'_{N}}=r'_{N}\notin 2\mathbb{Z}\Rightarrow S^{N}(i,1)=i$, for $1\le i\le r'_{N}$. Since $S^{N}$ is symplectic, one has $r'_{N}=1=\mathbf{b}_{r'_{N}}$ otherwise $\rem(1,\dots,r'_{N}-1)=(1,\dots,r'_{N}-1)\neq ()$ which is a contradiction. In this case, indeed,
\begin{align*}S_{\r_N}&=( \mathbf{b}_{r'_{N}}=1, \mathbf{b}_{r'_{N}}+1+1=3,\mathbf{b}_{r'_{N}}+1+2=4,\dots, \mathbf{b}_{r'_{N}}+1+Q[N]=1+1+Q[N])\\
&=(1,3,4,\dots,1+Q[N], 2+Q[N]), \mbox{ with $\rem(S_{\r_N})=(3,4,\dots,1+Q[N], 2+Q[N])$}.
\end{align*}
Since $ S^{N}(1,1)=1=S'_{\r_N}(1,1)$ and $ S^{N}$ is symplectic, $ S^{N}(i,1)\ge 2i-1\ge 2i-(i-1)=S'_{N}(i,1)$, $i\ge 2$

It remains to consider the case, $S^{N}(r'_{N},1)>\mathbf{b}_{r'_{N}}=S_{\r_N}(\mathbf{b}_{r'_{N}},1)\ge r'_{N} $
\bigskip

Furthermore, $S_{\r_N}(i,1)=i+1 $ for $i\ge b_{r'_{N}}+1\ge r'_{N}+1 $, and since

$$\mathbf{b}_{r'_{N}}=S'_{\r_N}(b_{r'_{N}},1)< S^{N}(r'_{N},1)<S^{N}(r'_{N}+1,1)\Rightarrow \mathbf{b}_{r'_{N}}+1=S_{\r_N}(b_{r'_{N}},1)\le S^{N}(r'_{N}+1,1)<S^{N}(r'_{N}+2,1)$$
Hence, $\mathbf{b}_{r'_{N}}+k\ge r'_{N}+k$, and
$$\mathbf{b}_{r'_{N}}+k+1=S_{\r_N}(b_{r'_{N}}+k,1)\le S^{N}(r'_{N}+k+1,1),~k\ge 0$$

Define also $Q^{(N-1)}\in \widetilde Rec_{2n}(\lambda/\mu^{(N-1)})$ obtained from $Q^{(N)}$ by deleting the $Q[N]$ entries $N$.

We now start  with $P^{(N-1)}= S_{\r_N}\bigcdot S^{N}\in SST_{2n}(\mu^{(N-1)})$ and $Q^{(N-1)}$. Note the $\r_{N-1}\ge \r_{N}$ is the row index of the cell of $\mu^{(N-1)}/{\mu^{(N-1)}}'$.

Define
 \begin{align*}P^{(N-2)}&:=\c \circ (\rm{red}_{1}^{-1},\rm{id})\circ (\;,\underset{\r_{N-1}}\longleftarrow S_{\r_N}\bigcdot S^{N})\\
&=\c \circ (\rm{red}_{1}^{-1},\rm{id})\circ (b_{r'_{N-1}}, (S_{\r_N}\bigcdot S^{N})^{N-1}), \mbox{ with $(S_{\r_N}\bigcdot S^{N})^{N-1}=(\underset{\r_{N-1}}\longleftarrow S_{\r_N}\bigcdot S^{N})$}\\
&=\c\circ (\rm{red}_{1}^{-1}(b_{r'_{N-1}}), (S_{\r_N}\bigcdot S^{N})^{N-1})\\
&=\c\circ (\rm{red}_{1}^{-1}(b_{r'_{N-1}}), (S^{\r_N})^{N-1}), \mbox{ by \eqref{srn} }\\
&=\c\circ({(S^{\r_N}})_{\r_{N-1}}, (S^{\r_N})^{N-1})\\
&=(S^{\r_N})_{\r_{N-1}}\bigcdot (S_{\r_N}\bigcdot S^{N})^{N-1}=:S^{\r_N\r_{N-1}} \in SST_{2n}(\mu^{(N-2)})\\
&\mu^{(N-2)}={\mu^{(N)}}-\delta_{\r_N}-\delta_{\r_{N-1}}+\varpi_{\ell(\mu^{(N-1)})}+\varpi_{\ell(\mu^{(N-2)})}
\end{align*}
where $S_{\r_N}(r'_{N-1},1)=b_{r'_{N-1}}$ is the entry bumped from the cell $(r'_{N-1},1)$ in the  column $S_{\r_N}$ with $r'_{N-1}\ge \mathbf{b}_{r'_{N}}-1\ge r'_{N}$, if $\mathbf{b}_{r'_{N}}\in 2\Z$ and $r'_{N-1}\ge \mathbf{b}_{r'_{N}}\ge r'_{N}$, if $\mathbf{b}_{r'_{N}}\notin 2\Z$. Therefore $b_{r'_{N-1}}\ge b_{r'_{N}}$ Since, from $(i)$  $\r_{N}\le\cdots\le \r_1$, the previous assertion is guaranteed by Remark \ref{revroutes} on  reverse bumping routes.

$(S_{\r_N}\bigcdot S^{N})^{N-1}$  is the tableau obtained from $S_{\r_N}\bigcdot S^{N}$ after applying the reverse column insertion to the cell of $\mu^{(N-1)}/\mu^{(N-1)'}$ in $S_{r_N}\bigcdot S^{N}$. Hence $(S_{\r_N}\bigcdot S^{N})^{N-1}$ has shape $\mu^{(N-1)'} $, and, in particular,  $ (S_{\r_N}\bigcdot S^{N})^{N-1}(r'_{N-1},1)\ge b_{r'_{N-1}}$ and $(S_{\r_N}\bigcdot S^{N})^{N-1}(i,1)=S'_{\r_N}\bigcdot S^{N}(i,1)$, for $i\neq r'_{N-1}$.

The expansion
 $\rm{red}_{1}^{-1}(b_{r'_{N-1}})=S_{\r_{N-1}}
 $ is a column of length $Q[N-1]+1\ge Q[N]$ obtained by the expansion of order  $[Q[N-1]]$  of the column $(b_{r'_{N-1}})$ of length $1$ according to the rule

 \begin{align}&S_{\r_{N-1}}=:\\
&=:\begin{cases}& (1,2,\dots, \mathbf{b}_{r'_{N-1}}-2, \mathbf{b}_{r'_{N-1}}, \mathbf{b}_{r'_{N-1}}+1,\mathbf{b}_{r'_{N-1}}+2,\dots, \mathbf{b}_{r'_{N-1}}+(Q[N-1]-(\mathbf{b}_{r'_{N-1}}-2))), \\
&  \mathbf{b}_{r'_{N-1}}\in 2\mathbb{Z}\\
&(1,2,\dots \mathbf{b}_{r'_{N-1}}-1, \mathbf{b}_{r'_{N-1}}, \mathbf{b}_{r'_{N-1}}+1+1,\mathbf{b}_{r'_{N-1}}+1+2,\dots, \mathbf{b}_{r'_{N-1}}+1+( Q[N-1]-(\mathbf{b}_{r'_{N-1}}-1))),\\
& \mathbf{b}_{r'_{N-1}}\notin 2\mathbb{Z}
\end{cases}
\end{align}

Since $ b_{r'_{N-1}}\ge b_{r'_{N}}$, it follows that $S_{\r_{N-1}}\bigcdot (S_{r_N}\bigcdot S^{N})^{N-1} \in SST_{2n}(\mu^{(N-2)})$.
\end{proof}

\begin{cor}\label{cor:recordhole1} Let $S\in SpT_{2n}(\mu)$ and let
 $Q\in   \widetilde Rec_{2n}(\lambda/\mu)$   with $1$-$0$ slack sequence $\underline\t=(\underbrace{1,1,\dots,1}_N, 0^M)$ and slack vector sequence $\underline \r=[\r_{N+M}\le\cdots\le \r_{1+M},()^M]$, $M\ge 0$.  Then

\begin{enumerate}
\item for $1\le i\le N$,

\begin{itemize}
\item [(i)]$\mu^{(M+i-1)}=(\mu^{(M+i)}-\delta_{\r_{M+i}})+\varpi_{Q[M+i]+1}$ is such that
\begin{align*}&\lambda\supseteq\mu^{(M)}\supset_{vert} \mu^{(1+M)}\supset_{vert} \cdots\supset_{vert}\mu^{(N-1+M)}\supset_{vert}\mu^{(N+M)} =\mu
\end{align*}
\item[(ii)] $2n-1\ge Q[i]+1=\ell(\mu^{(M+i-1)})$ and $\ell(\mu^{(M+i-1)})\notin 2\Z$.
 \item [(iii)] $1\le \r_{M+i}\le \ell(\mu^{(M+i)})\le\ell(\mu^{(M+i-1)})\le 2n-1$.
%and $\r_i\le 2n-1$. % whenever $\r_i\neq ()$,
\end{itemize}
\item
\begin{align}%{\LRAII^{AII}}^{-1}
\widetilde R(S, Q)=S^{\underline \r}= S^{\r_{M+N}\cdots\r_{M+2} \r_{M+1},()^M}\in SST_{2n}(\lambda),
\end{align}
where $S^{\r_{M+N}\cdots\r_{M+2} \r_{M+1},()^M}:=Y(\nu)\bigcdot S^{\r_{N+M}\cdots\r_{2+M} \r_{1+M}}= Y(\nu)\bigcdot
{({(\cdots{((S^{\r_{N+M}})}^{\r_{M+N-1}}){\cdots})}^{\r_{M+2}})}^{\r_{M+1}}$ and $S^{\r_{M+N}\cdots\r_{M+i}}\in SST_{2n}(\mu^{(M+i-1)})$ for $i=1,\dots,N$.
 where $Q[1]\ge\cdots\ge Q[M]$  and $\nu=( Q[1],\dots,Q[M])^t$
%, is recursively defined by
%$$S$$
\end{enumerate}

\end{cor}

More generally from Theorem  \ref{thm:verygeneralstrip}, \eqref{inverse},  Proposition \ref{prop:slack}, Lemma \ref{lem:recordhole0}, Theorem \ref{lem:recordhole1} and Remark \ref{revroutes}, the next assertion follows.

\iffalse
\begin{lem}\label{lem:symp}\oo If  ${\bf b}=(b_1,\dots,b_t)\in SpT_{2n}(\varpi_t)$ and ${\bf a}=(a_1,\dots,a_{t'})\in SST_{2n}(\varpi_{t'})$, for some $t'\le t$, is such that ${\bf a} \ge {\bf b}$, that is, $a_i\ge b_i$, $1\le i\le t'$,  then ${\bf a}\in SpT_{2n}(\varpi_{t'})$.
\end{lem}
\begin{proof}If ${\bf b}=(b_1,\dots,b_t)\in SpT_{2n}(\varpi_t)$ and ${\bf a} \ge {\bf b}$, $a_k\ge b_k\ge 2k-1$, for $k=1,\dots,t'$. Hence ${\bf a}\in SpT_{2n}(\varpi_{t'})$.
\end{proof}
\fi

\begin{thm}\cite{azslack}\label{thm:generalhole} Let $S\in SpT_{2n}(\mu)$ and let
 $Q\in   \widetilde Rec_{2n}(\lambda/\mu)$  with slack sequence $\underline \t=(t^{(N)}_0\ge \dots \ge t^{(1)}_0)$ and slack (row index) vector sequence $\underline \r=[\r^{(N)}\le_\r\dots\le_\r\r^{(1)}]$.  Then

\begin{enumerate}
\item for $1\le i\le N$,
\begin{itemize}

\item[(i)]$\mu^{(i-1)}=(\mu^{(i)}-\delta_{\r_i})+\varpi_{Q[i]+t_0^{(i)}}$   is such that

\begin{align*}&\lambda=\mu^{(0)}\supset_{vert} \mu^{(1)}\supset_{vert} \cdots\supset_{vert}\mu^{(N-1)}\supset_{vert}\mu^{(N)} =\mu.
\end{align*}

\item[(ii)]  $\r^{(i)}\subseteq[1,\ell(\mu^{(i)}]\subseteq [1,\ell(\mu^{(i-1)}]\subseteq[1,2n-t_0^{(i)}]$.

\item [(iii)]
$2n-t_0^{(i)}\ge Q[i]+t_0^{(i)}=\ell(\mu^{(i-1)}$.

\end{itemize}
\item
\begin{align}%{\LRAII^{AII}}^{-1}
\widetilde R(S, Q)=S^{\underline \r}=S^{\r^{(N)}\cdots\r^{(2)} \r^{(1)}}:=
{\bigg({\bigg(\cdots{\big(\big(S^{\r^{(N)}}\big)}^{\r^{(N-1)}}\big)\cdots\bigg)}^{\r^{(2)}}   \bigg)}^{\r^{(1)}}\in SST_{2n}(\lambda).\label{generalN}
\end{align}
\end{enumerate}

\end{thm}
\begin{proof} $(1)$ Follows from the definition of recording tableau $Q\in   \widetilde Rec_{2n}(\lambda/\mu)$ and the conditions  imposed by it on $\r$ and $\t$.

$(2)$  The proof is by  induction on $N$. The assertion \eqref{generalN} is proved for $N=1$ in Theorem \ref{thm:verygeneralstrip}. Let $N\ge 2$ and by induction assume that the assertion \eqref{generalN} holds for all $1< K+1\le N$, where
\begin{align}%{\LRAII^{AII}}^{-1}
\widetilde R(S, Q)=S^{\underline \r}=S^{\r^{(N)}\cdots\r^{(K+1)}} :=
{\bigg(\cdots{\big(\big(S^{\r^{(N)}}\big)}^{\r^{(N-1)}}\big)\cdots\bigg)}^{\r^{(K+1)}}   , \mbox{ for all $1< K+1\le N$. }\label{inductionk}
\end{align}

One has, by definition of $S^{\r^{(N)}\cdots\r^{(K)}}$,

\begin{align*}
&S^{\r^{(N)}\cdots\r^{(K)}}=\c \circ (\rm{red}_{t^{(K)}_0}^{-1},\rm{id})\circ (\;,\underset{{\r^{(K)}}}\leftarrow S^{\r^{(N)}\cdots\r^{(K+1)}}).
\end{align*}
We have to compute
\begin{align}\label{todo}
&\c \circ (\rm{red}_{t^{(K)}_0}^{-1},\rm{id})\circ (\;,\underset{{\r^{(K)}}}\leftarrow S^{\r^{(N)}\cdots\r^{(K+1)}})
\end{align}
and show that in $(\;,\underset{{\r^{(K)}}}\leftarrow S^{\r^{(N)}\cdots\r^{(K+1)}})$ the bumped column from $S^{\r^{(N)}\cdots\r^{(K+1)}}$ is a symplectic column in $SpT_{2n}(\varpi_{t_0^{(K)}})$, and that the output of \eqref{todo} is a semistandard tableau in $SST_{2n}(\mu^{(K)})$.

By induction
\begin{align}
S^{\r^{(N)}\cdots\r^{(K+1)}}&=(S^{\r^{(N)}\cdots\r^{(K+2)}})^{\r_{k+1}}\nonumber\\
&=\c \circ (\rm{red}_{t^{(K+1)}_0}^{-1},\rm{id})\circ (\;,\underset{{\r^{(K+1)}}}\leftarrow S^{\r^{(N)}\cdots\r^{(K+2)}})\nonumber\\
%&=\c \circ (\rm{red}_{t^{(2)}_0}^{-1},\rm{id})\circ \big(\underset{{\r^{(2)}}}\leftarrow (\c \circ (\rm{red}_{t^{(3)}_0}^{-1},\rm{id})\circ\\
%&(\underset{{\r^{(3)}}}\leftarrow S^{\r^{(N)}\cdots\r^{(4)}}\big)\bigg)\\
&=c\circ (\rm{red}_{t_0^{(K+1)}}^{-1},\rm{id})((a_1,\dots,a_{t_0^{(K+1)}}),S' ), \mbox{ with $S'=(\underset{{\r^{(K+1)}}}\leftarrow S^{\r^{(N)}\cdots\r^{(K+2)}})$}\nonumber\\
&=c\circ (\rm{red}_{t_0^{(K+1)}}^{-1}(a_1,\dots,a_{t_0^{(K+1)}}), S') \nonumber\\
&=\redu_{t_0^{(K+1)}}^{-1}((a_1,\dots,a_{t_0^{(K+1)}}))\bigcdot S'\label{rinduction}
\end{align}
That is, $S^{\r^{(N)}\cdots\r^{(K+1)}}$ is the concatenation of the column $\redu_{t_0^{(K+1)}}^{-1}((a_1,\dots,a_{t_0^{(K+1)}}))$ with $S'$ where, by induction,  the bumped column $(a_1,\dots,a_{t_0^{(K+1)}})$ from $S^{\r^{(N)}\cdots\r^{(K+2)}}$ is symplectic and belongs  to $ SpT_{2n}(\varpi_{t_0^{(K+1)}})$.

\iffalse
We have to compute
\begin{align}\label{todo}
&\c \circ (\rm{red}_{t^{(1)}_0}^{-1},\rm{id})\circ (\;,\underset{{\r^{(1)}}}\leftarrow S^{\r^{(N)}\cdots\r^{(2)}})
\end{align}
and show that $(\;,\underset{{\r^{(1)}}}\leftarrow S^{\r^{(N)}\cdots\r^{(2)}})$ bumps a symplectic column in $SpT_{2n}(\varpi_{t_0^{(1)}})$, and that the output of \eqref{todo} is a semistandard tableau.

\fi

\bigskip
We start by computing $(\;,\underset{{\r^{(K)}}}\leftarrow S^{\r^{(N)}\cdots\r^{(K+1)}})$ in \eqref{todo}. From \eqref{rinduction},
%From \eqref{rinduction},
%and recalling   Definition \ref{slackorder}, $\r^{(K+1)}\le_\r\r^{(K)}$, and Remark \ref{revroutes}, one has
\begin{align}
&(\;,\underset{{\r^{(K)}}}\leftarrow S^{\r^{(N)}\cdots\r^{(K+1)}})=\nonumber\\
&=\bigg(\;,\underset{{\r^{(K)}}}\leftarrow\bigg(\c \circ (\rm{red}_{t^{(K+1)}_0}^{-1},\rm{id})\circ (\;,\underset{{\r^{(K+1)}}}\leftarrow S^{\r^{(N)}\cdots\r^{(K+2)}})\bigg)\bigg)\nonumber\\
&=\bigg(\;,\underset{{\r^{(K)}}}\leftarrow\bigg(c\circ (\rm{red}_{t_0^{(K+1)}}^{-1},\rm{id})\circ\big((a_1,\dots,a_{t_0^{(K+1)}}),S' \big)\bigg)\bigg), \mbox{ with $S'=(\underset{{\r^{(K+1)}}}\leftarrow S^{\r^{(N)}\cdots\r^{(K+2)}})$}\nonumber
\\
&=\bigg(\;,\underset{{\r^{(K)}}}\leftarrow\bigg(\redu_{t_0^{(K+1)}}^{-1}\big((a_1,\dots,a_{t_0^{(K+1)}}))\bigcdot S'\big)\bigg)\bigg)\nonumber\\
&=\bigg(\;,\underset{{\r^{(K)}}}\leftarrow\bigg(L\bigcdot (\underset{{\r^{(K+1)}}}\leftarrow S^{\r^{(N)}\cdots\r^{(K+2)}}\big)\bigg),\; \mbox{with $L:=\redu_{t_0^{(K+1)}}^{-1}((a_1,\dots,a_{t_0^{(K+1)}})).$}
\label{intermediate}
\end{align}

%From \eqref{rinduction},
%and recalling   Definition \ref{slackorder}, $\r^{(K+1)}\le_\r\r^{(K)}$, and Remark \ref{revroutes}, one has

\vskip1cm
To compute  the bumped column in \eqref{intermediate}, according to the reverse Schensted column insertion, Subsection \ref{subsec:insert}, one has  first
\begin{align*}
&\big(\;,\underset{{\r^{(K)}}}\leftarrow(\underset{{\r^{(K+1)}}}\leftarrow S^{\r^{(N)}\cdots\r^{(K+2)}})\big)=\big(\;,\underset{{\r^{(K)}}}\leftarrow S'\big)=((b'_1,\dots,b'_{t_0^{(K)}}), \widetilde S),
\end{align*}
%&=(\redu_{t_0^{(2)}}^{-1}((a_1,\dots,a_{t_0^{(K+1)}}))\leftarrow(\underset{{\r^{(1)}}}\leftarrow S'))\\
and then the column $(b'_1,\dots,b'_{t_0^{(K)}})$ bumped out from $ S'=\big(\underset{{\r^{(K+1)}}}\leftarrow S^{\r^{(N)}\cdots\r^{(K+2)}}\big)$ will proceed, %to the next column to its left which is $\redu_{t_0^{(K+1)}}^{-1}((a_1,\dots,a_{t_0^{(K+1)}}))$
in the fashion of reverse Schensted column insertion, to the next column $L$ to its left $$L:=\redu_{t_0^{(K+1)}}^{-1}((a_1,\dots,a_{t_0^{(K+1)}})).$$

%\begin{align*}&= \redu_{t_0^{(K+1)}}^{-1}((a_1,\dots,a_{t_0^{(K+1)}}))\\
%\end{align*}
At this point, recalling Definition \ref{slackorder}, $\r^{(K+1)}\le_\r\r^{(K)}$, and Remark \ref{revroutes}, we may observe that  $$(a_1,\dots,a_{t_0^{(k+1)}})\le(b'_1,\dots,b'_{t_0^{(k)}}), \; t_0^{(K+1)}\ge t_0^{(K)}, $$
and if $(a_1,\dots,a_{t_0^{(K+1)}})\in SpT_{2n}(\varpi_{t_0^{(K+1)}})$, by Lemma \ref{lem:symp}, also %$(b_1,\dots,b_{t_0^{(1)}}),\,
$(b'_1,\dots,b'_{t_0^{(1)}})\in SpT_{2n}(\varpi_{t_0^{(K)}})$.
% by $\r^{(2)}\le_\r\r^{(1)}$ and Remark \ref{revroutes}

To finalize the computation of \eqref{intermediate}, the entries of the bumped column $(b'_1,\dots,b'_{t_0^{(K)}})$ proceed to the column $L$ on the left, in the fashion of reverse Schensted column insertion, from $b'_{t_0^{(K)}}$ to $b'_1$,  to produce
\begin{align*}
%&\leftarrow((b'_1,\dots,b'_{t_0^{(1)}}), \widetilde S)\\
%&=\bigg(\redu_{t_0^{(K+1)}}^{-1}((a_1,\dots,a_{t_0^{(K+1)}}))\leftarrow(b'_1,\dots,b'_{t_0^{(K)}})\bigg)\bigcdot \widetilde S\\
&\eqref{intermediate}=((b_1,\dots,b_{t_0^{(K)}}),X\bigcdot \widetilde S)=((b_1,\dots,b_{t_0^{(K)}}),X\bigcdot(\underset{{\r^{(K)}}}\leftarrow(\underset{{\r^{(K+1)}}}\leftarrow S^{\r^{(N)}\cdots\r^{(K+2)}})\big)\\
&=\bigg((b_1,\dots,b_{t_0^{(K)}}),\underset{{\r^{(K)}}}\leftarrow\bigg(L\bigcdot (\underset{{\r^{(K+1)}}}\leftarrow S^{\r^{(N)}\cdots\r^{(K+2)}}\big)\bigg)\\
&=\bigg((b_1,\dots,b_{t_0^{(K)}}),\underset{{\r^{(K)}}}\leftarrow S^{\r^{(N)}\cdots\r^{(K+1)}}\bigg)
%&=\redu_{t_0^{(K)}}^{-1}((b_1,\dots,b_{t_0^{(K)}}))(X\bigcdot \widetilde S)
%\mbox{ where}
\end{align*}
where $(b_1,\dots,b_{t_0^{(K)}})$ is the bumped column from $L$ and hence from $S^{\r^{(N)}\cdots\r^{(K+1)}}=(\underset{{\r^{(K+1)}}}\leftarrow S^{\r^{(N)}\cdots\r^{(K+2)}})$ and
$X\bigcdot \widetilde S$ is the  semi-standard tableau that remains after the action of the reverse Schensted column insertion defined by $\r^{(k)}$.
\iffalse
%\begin{align*}
%%&\underset{{\r^{(1)}}}\leftarrow(\underset{{\r^{(2)}}}\leftarrow S^{\r^{(N)}\cdots\r^{(3)}})=\underset{{\r^{(1)}}}\leftarrow((a_1,\dots,a_{t_0^{(2)}}), S' )\\
%&= (a_1,\dots,a_{t_0^{(2)}}),(\underset{{\r^{(1)}}}\leftarrow S')\\
%&=(a_1,\dots,a_{t_0^{(2)}}),((b'_1,\dots,b'_{t_0^{(1)}}), \widetilde S)\\
%&= \redu_{t_0^{(2)}}^{-1}((a_1,\dots,a_{t_0^{(2)}}))\leftarrow((b'_1,\dots,b'_{t_0^{(1)}}), \widetilde S)\\
%& \mbox{ where $(a_1,\dots,a_{t_0^{(K+1)}})\le(b'_1,\dots,b'_{t_0^{(1)}})$ by $\r^{(2)}\le_\r\r^{(1)}$ and Remark \ref{revroutes}} \\
%&=\bigg(\redu_{t_0^{(2)}}^{-1}((a_1,\dots,a_{t_0^{(2)}}))\leftarrow(b'_1,\dots,b'_{t_0^{(1)}})\bigg)\bigcdot \widetilde S\\
%\end{align*}
\fi
%Therefore since
%Note $(b'_1,\dots,b'_{t_0^{(K)}})\ge(b_1,\dots,b_{t_0^{(K)}})\ge(a_1,\dots,a_{t_0^{(K+1)}})$ and if %$(a_1,\dots,a_{t_0^{(K+1)}})\in SpT_{2n}(\varpi_{t_0^{(K+1)}})$ also $(b_1,\dots,b_{t_0^{(K)}})$, %$(b'_1,\dots,b'_{t_0^{(1k)}})\in SpT_{2n}(\varpi_{t_0^{(K)}})$.
Note  from bottom to top each $b'_i$ will look for the lowest  strictly smaller or equal  entry on the column $L$ to its left and since
$(a_1,\dots,a_{t_0^{(K+1)}})$ are entries of $L$ and $b'_i\ge a_i$ for $1\le i\le t_0^{(K)}$, the bumped column $(b_1,\dots,b_{t_0^{(k)}})$ from $L$ satisfies $b'_i\ge b_i\ge a_i$ for $1\le i\le t_0^{(K)}$.

Therefore,  $(b'_1,\dots,b'_{t_0^{(K)}})\ge(b_1,\dots,b_{t_0^{(K)}})\ge(a_1,\dots,a_{t_0^{(K+1)}})$ and if $(a_1,\dots,a_{t_0^{(K+1)}})\in SpT_{2n}(\varpi_{t_0^{(K+1)}})$ also $(b_1,\dots,b_{t_0^{(K)}})$, $(b'_1,\dots,b'_{t_0^{(1k)}})\in SpT_{2n}(\varpi_{t_0^{(K)}})$ and  \begin{align}
&(\;,\underset{{\r^{(K)}}}\leftarrow S^{\r^{(N)}\cdots\r^{(K+1)}})=((b_1,\dots,b_{t_0^{(K)}}),(X\bigcdot \widetilde S)
\end{align}

It remains to conclude that $\redu_{t_0^{(1)}}^{-1}((b_1,\dots,b_{t_0^{(K)}}))\bigcdot (X\bigcdot \widetilde S)$ is a semi-standard tableau. This follows from Subsection \ref{subsec:essential}.
The row coordinate of  $b_i$ is larger or equal than the row coordinate of $a_i$ in $L$, for all $1\le i\le t_0^{(K)}$. The conditions on the formulas in Lemma \ref{lem:reductionsurj}, Theorem \ref{lem:reductionsurj2}  and Theorem \ref{thm:nofactors} and \cite{azreduction},  to compute $\redu_{t_0^{(1)}}^{-1}((b_1,\dots,b_{t_0^{(k)}}))$ %\eqref{verygeneral0} and \eqref{verygeneral200}
force the push down or keep  the $b_i$ row-coordinate  as an entry of $L$ for $1\le i\le t_0^{(K)}$. Hence the concatenation
$$\redu_{t_0^{(1)}}^{-1}((b_1,\dots,b_{t_0^{(K)}}))\bigcdot (X\bigcdot \widetilde S)\in SST_{2n}(\lambda).$$
\end{proof}

From now on we write $\widetilde R= {\LRAII^{AII}}^{-1}$ and $\widetilde Rec_{2n}=Rec_{2n}$.

%%%%%%%%%%%%%%%%%%%%%%%%%%%%%%%%%%%%%%%%%%%%%%%%%%%%%%%%%%%%%%%%%%%%%%%%

%%%%%%%%%%%%%%%%%%%%%%%%%%%%%%%%%%%%%%%%%%%%%%%%%%%%%%%%%%%%%%%%%%%%%%%%%%%%%%%%%%%%%%%%%%%%

\section{ Characterization of the $\mathfrak{k}$-highest weight tableaux  via  the inverse quantum LR map for one vertical strip and $1$-$0$-slack sequence recording tableaux}\label{sec:last}

The explicit surjectivity and consequently the explicit inverse of the quantum Littlewood-Richardson map allows to  explicitly compute  $\mathfrak{k}$-highest or lowest weight tableaux in the
 proof of the Naito–Sagaki conjecture via the branching rule for $\imath$quantum groups \cite{nsw}.
We closely follow the notation in \cite{nsw} and refer to it for further definitions and details.

 Let $n\in\mathbb{N}$ and let us consider the two sequences of positive integers defined in \cite{nsw}. For $k=1,\dots,n$,

\begin{align}&u_k=2k-\frac{1+(-1)^k}{2}=\begin{cases}2k,&\text{ if } k \notin 2\Z,\\
2k-1,&\text{ if } k \in 2\Z,\label{numbers:u}
\end{cases}\\
&v_k=2k-\frac{1+(-1)^{k+1}}{2}=\begin{cases}2k,&\text{ if } k \in 2\Z,\\
2k-1,&\text{ if } k \notin 2\Z. \label{numbers:v}
\end{cases}
\end{align}
These two sequences
\begin{align}\{u_i\}_{i=1}^n=\begin{cases}\{2,3,6,7,10,\dots, 2(n-1)-1,2n\},& n\notin 2\Z,\\
\{2,3,6,7,\dots, 2(n-1),2n-1\},& n\in 2\Z\end{cases}\label{numbers:uu}
\end{align}
%$\subseteq [1,2n]$
and
\begin{align}\{v_i\}_{i=1}^n=\begin{cases}\{1,4,5,8,\dots, 2(n-1)-1,2n\},& n\in 2\Z\\
\{1,4,5,\dots, 2(n-1),2n-1\},& n\notin 2\Z,\end{cases}\label{numbers:vv}
\end{align}
 have no common values and its union  gives  $\{u_i\}_{i=1}^n\sqcup\{v_i\}_{i=1}^n=[1,2n]$.
\iffalse
 Their recursive generation is as follows:

\begin{align*}
u_1=2, \quad u_k=\begin{cases}u_{k-1}+1,& \text{ if } k\ge 2,~  k \in 2\Z,\\
u_{k-1}+3,& \text{ if }  k\ge 2, ~k \notin 2\Z,
\end{cases}
\end{align*}

\begin{align*}
v_1=1, \quad v_k=\begin{cases}v_{k-1}+3,& \text{ if } k\ge 2,~  k \in 2\Z,\\
v_{k-1}+1,& \text{ if }  k\ge 2, ~k \notin 2\Z.
\end{cases}
\end{align*}
\fi

\begin{ex} For $n=6$, $u_i=2,3,6,7,10,11$ and $v_i=1,4,5,8,9,12$, $1\le i\le 6$, and $\{u_i\}_{i=1}^6\cup \{v_i\}_{i=1}^6=\{1,2,\dots,12\}$; and for $n=7$, $u_i=2,3,6,7,10,11, 14$ and $v_i=1,4,5,8,9,12,13$, $1\le i\le 7$,
and $\{u_i\}_{i=1}^6\cup \{v_i\}_{i=1}^6=\{1,2,\dots,14\}$.
\end{ex}

 For $S\in SST_{2n}(\lambda)$, its $\mathfrak{k}$-weight is \cite[Section 5.3]{nsw}
\begin{align}\wt_\mathfrak{k}(S)=(S[u_1]-S[v_1])\tilde\varepsilon_1+(S[u_2]-S[v_2])\tilde\varepsilon_2+\dots+(S[u_n]-S[v_n])\tilde\varepsilon_n.\label{kweight}
\end{align}

For example for $n=3$, $\wt_\mathfrak{k}(S)=(S[2]-S[1])\tilde\varepsilon_1+(S[3]-S[4])\tilde\varepsilon_2+(S[6]-S[5])\tilde\varepsilon_3$.

The set $\widetilde P^+=\{\mu_1 \tilde\varepsilon_1+\cdots+\mu_n \tilde\varepsilon_n\in\widetilde P:\mu_1\ge\cdots\ge \mu_n\ge 0\}$ can be identified with the set $Par_{\le n}$ \cite[Section 5.1]{nsw}.

 If $S$ is  the symplectic tableau below on the LHS \eqref{symphw-lw}  then the shape is equal to  the  corresponding $\mathfrak{k}$-weight,
  $\wt_\mathfrak{k}(S)=(S[u_1]\ge S[u_2]\ge \dots\ge S[u_n])$,  and if $S$ is on the RHS \eqref{symphw-lw} then the $\mathfrak{k}$-weight is

   \begin{align*}& \eqref{symphw-lw}RHS, \wt_\mathfrak{k}(S) =-S[v_1]\tilde\varepsilon_1-S[v_2]\tilde\varepsilon_2-\dots-S[v_n]\tilde\varepsilon_n
   \\
   &=-(S[v_1]\tilde\varepsilon_1+S[v_2]\tilde\varepsilon_2+\dots+S[v_n]\tilde\varepsilon_n)\\
   &\mbox{ identified with } -(S[v_1]\ge S[v_2]\ge \dots\ge S[v_n])
   \end{align*}  and the shape is
  $$w_0\wt_\mathfrak{k}(S)=w_0(-S[v_1],-S[v_2],\dots,-S[v_n])=(S[v_1]\ge S[v_2]\ge\dots\ge S[v_n])$$
  \noindent where $w_0$ is the longest element of the Weyl group of the Lie algebra $\mathfrak{k}$ isomorphic to the symplectic Lie algebra $\mathfrak{sp}_{2n}(\mathbb{C})$,

\begin{align}\label{symphw-lw}
\YT{0.2in}{}{
 {{u_1},{u_1},\cdots,\cdots,\cdots,\cdots,u_1},
 {{u_2},\cdots,\cdots,\cdots,\cdots,{u_2}},
 {{\vdots},\vdots,\vdots,{\vdots}},
 {u_n,\cdots,u_n},
} &\qquad\qquad
\YT{0.2in}{}{
 {{v_1},{v_1},\cdots,\cdots,\cdots,\cdots,v_1},
 {{v_2},\cdots,\cdots,\cdots,\cdots, {v_2}},
 {{\vdots},\vdots,\vdots,{\vdots}},
 {v_n,\cdots,v_n},
}\\
\mbox{ symplectic $\mathfrak{k}$-highest weight tableau } & \quad \mbox{ symplectic $\mathfrak{k}$-lowest weight tableau }\nonumber
\end{align}

Let $SST_{2n}$ denote the set of all semi-standard tableaux in  the alphabet $[1,2n]$ and $SpT_{2n}$ its subset  of symplectic tableaux.
Next, $\mathfrak{k}$-highest weight tableaux and $\mathfrak{k}$-lowest weight tableaux in $SST_{2n}$ are defined.
\begin{defi} \cite{nsw} Let $S\in SST_{2n}$.
\begin{enumerate}
\item
$S$ is called a $\mathfrak{k}$-highest weight tableau if $P^{AII}(S)$ is a symplectic tableau of the form shown in the LHS of \eqref{symphw-lw}; let $SST^{\mathfrak{k}-hw}
_{2n}(\lambda)\subseteq SST_{2n}(\lambda)$ denote the set of all $\mathfrak{k}$-highest weight tableaux of shape $\lambda$ in $SST_{2n}$.

\item
$S$ is called a $\mathfrak{k}$-lowest weight tableau if $P^{AII}(S)$ is a symplectic tableau of the form shown in RHS of \eqref{symphw-lw}; let $SST^{\mathfrak{k}-lw}
_{2n}(\lambda)\subseteq SST_{2n}(\lambda)$ denote the set of all $\mathfrak{k}$-lowest weight tableaux of shape $\lambda$ in $SST_{2n}$.
\end{enumerate}
\end{defi}

%Let $SST_{2n}$ denote the set of all semi-standard tableaux in  the alphabet $[1,2n]$ and $SpT_{2n}$ its subset  of symplectic tableaux.
Since $P^{AII}(S)=S$ for $S\in SpT_{2n}$, this definition for $S\in SpT_{2n}$ obviously restricts to \eqref{symphw-lw}.
For $n=1,2,3$, the symplectic $\mathfrak{k}$-highest weight respectively $\mathfrak{k}$-lowest weight tableaux are then respectively of the form
\begin{align}\label{1symphw-lw2}n=1:~~
&S^H=\YT{0.2in}{}{
 {{2},{2},\cdots,\cdots,\cdots,\cdots,2},
},\qquad S_L=\YT{0.2in}{}{
 {{1},{1},\cdots,\cdots,\cdots,1},
 }\in SpT_2;\\\nonumber
 \\
n=2:~~&S^H=\YT{0.2in}{}{
 {{2},{2},\cdots,\cdots,\cdots,\cdots,2},
 {{3},\cdots,\cdots,\cdots,\cdots,{3}},
},\qquad S_L= \YT{0.2in}{}{
 {{1},{1},\cdots,\cdots,\cdots,1},
 {{4},\cdots,\cdots,\cdots,{4}},
 }\in SpT_4\label{2symphw-lw2}\\ \nonumber
 \\
 n=3:~~&S^H=\YT{0.2in}{}{
 {{2},{2},2,\cdots,\cdots,\cdots,\cdots,2},
 {{3},3,\cdots,\cdots,\cdots,\cdots,{3}},
 {{6},\cdots,\cdots,\cdots,\cdots,{6}},
},\qquad S_L= \YT{0.2in}{}{
 {{1},{1},1,\cdots,\cdots,\cdots,1},
 {{4},4,\cdots,\cdots,\cdots,{4}},
 {{5},\cdots,\cdots,\cdots,{5}},
 }\in SpT_6\label{3symphw-lw3}
\end{align}
\begin{align}
n=4:~~&S^H=\YT{0.2in}{}{
 {{2},{2},2,2,\cdots,\cdots,\cdots,\cdots,2},
 {{3},3,3,\cdots,\cdots,\cdots,\cdots,{3}},
 {{6},6,\cdots,\cdots,\cdots,\cdots,{6}},
 {{7},\cdots,\cdots,\cdots,{7}},
},\qquad S_L= \YT{0.2in}{}{
 {{1},{1},1,1,\cdots,\cdots,\cdots,1},
 {{4},4,4,\cdots,\cdots,\cdots,{4}},
 {{5},5,\cdots,\cdots,\cdots,{5}},
 {{8},\cdots,\cdots,\cdots,{8}},
 }\in SpT_6\label{4symphw-lw4}
\end{align}
In general, for $n\in\mathbb{N}$, the symplectic $\mathfrak{k}$-highest weight  tableaux  in $SpT_{2n}$ comprise the symplectic columns
\begin{align}u_1\cdots u_{n-1}u_n\in SpT_{2n}(\varpi_n),~u_1\cdots u_{n-1}\in SpT_{2n}(\varpi_{n-1}),\dots, u_1u_2\in SpT_{2n}(\varpi_{2}), ~u_1\in SpT_{2n}(\varpi_1)
\end{align}
and the symplectic $\mathfrak{k}$-lowest weight  tableaux  in $SST_{2n}$, comprise the symplectic columns
\begin{align}v_1\cdots v_{n-1}v_n\in SpT_{2n}(\varpi_n),~v_1\cdots v_{n-1}\in SpT_{2n}(\varpi_{n-1}),\dots, v_1v_2\in SpT_{2n}(\varpi_{2}), ~v_1\in SpT_{2n}(\varpi_1).
\end{align}
\begin{align}
n\notin 2\Z:\nonumber\\
&S^H=\YT{0.28in}{}{
 {{2},{2},2,2,\cdots,\cdots,\cdots,\cdots,2},
 {{3},3,3,\cdots,\cdots,\cdots,\cdots,{3}},
 {{6},6,\cdots,\cdots,\cdots,\cdots,{6}},
 {{7},\cdots,\cdots,\cdots,\cdots,{7}},
 {{10},\cdots,\cdots,\cdots,\cdots,{10}},
 {\vdots,\cdots,\cdots,\cdots,\vdots},
 {\scriptstyle {2n-3},\cdots, \cdots,\scriptstyle{2n-3}},
 {\scriptstyle{2n},\cdots,\scriptstyle{2n}},
 }\in SpT_{2n}\label{noddsymphw-hw}
 %\nonumber
 %& S_L= \YT{0.2in}{}{
 %{{1},{1},1,1,\cdots,\cdots,\cdots,1},
 %{{4},4,4,\cdots,\cdots,\cdots,{4}},
 %{{5},5,\cdots,\cdots,\cdots,{5}},
 %{{8},\cdots,\cdots,\cdots,{8}},
 %}\in SpT_6\label{}\nonumber
\end{align}

%\iffalse
\begin{align}
n\in 2\Z:\nonumber\\
&S^H=\YT{0.28in}{}{
 {{2},{2},2,2,\cdots,\cdots,\cdots,\cdots,2},
 {{3},3,3,\cdots,\cdots,\cdots,\cdots,{3}},
 {{6},6,\cdots,\cdots,\cdots,\cdots,{6}},
 {{7},\cdots,\cdots,\cdots,{7}},
 {\vdots,\cdots,\cdots,\cdots,\vdots},
  {\scriptstyle{2n-5},\cdots, \cdots,\scriptstyle{2n-5}},
 {\scriptstyle{2n-2},\cdots, \cdots,\scriptstyle{2n-2}},
 {\scriptstyle{2n-1},\cdots,\scriptstyle{2n-1}},
 }\in SpT_{2n}\label{nevensymphw-hw}
\end{align}

%\fi

It would be interesting to have an explicit characterization  of $SST^{\mathfrak{k}-hw}
_{2n}(\lambda)$ or $SST^{\mathfrak{k}-lw}
_{2n}(\lambda)$. This is known for $n=1, 2$ \cite{nsw}. Indeed one has an algorithm for ${\LRAII^{AII}}^{-1}$  to compute numerically the elements of those sets.  However, taking into account the $n=2$ case, it is expected that these tableaux satisfy restrictions in the form of inequalities between the multiplicities of some of its  columns. We  provide an interpretation of this phenomena via the reverse Schensted insertion ruled by the slack data on the quantum recording tableaux with one vertical strip and  with   $1$-$0$-slack vector sequences.

From the quantum Littlewood-Richardson bijection \eqref{wat0}, we see that for each $\mu\in Par_{\le n}$, the recording tableaux $Rec_{2n}(\lambda/\mu)\overset{\sim}\rightarrow LRS_{2n}(\lambda/\mu)$ determine the pairs consisting of $\mathfrak{k}$-highest  respectively $\mathfrak{k}$-lowest weight tableaux $S^{H,\mu}$ and $ S_{L,-\mu} \in
SST_{2n}(\lambda)$ where $\wt_{\mathfrak{k}}(S^{H,\mu})=\mu$  respectively $\wt_{\mathfrak{k}}(S_{L,-\mu})=-\mu$. That is, the set $Rec_{2n}(\lambda/\mu)$ determine the tableaux in
$SST^{\mathfrak{k}-hw}
_{2n}(\lambda)$  respectively $SST^{\mathfrak{k}-lw}
_{2n}(\lambda)$) such that $\wt_{\mathfrak{k}}(T)=\mu$  and respectively $wt_{\mathfrak{k}}(T)=-\mu$. Consequently, our  algorithm
%  showing explicitly the surjectivity of $LR^{AII}$ in \eqref{Rtilde} provides  an algorithm
for ${\LRAII^{AII}}^{-1}$  compute those sets.

\begin{prop} \label{prop:hwlw} Let $S^{H,\mu}, S_{L,-\mu}\in SpT_{2n}(\mu)$ be the symplectic $\mathfrak{k}-hw$ respectively $\mathfrak{k}-lw$ weight tableaux of $SpT_{2n}(\mu)$. Then
\begin{enumerate}
\item
 ${\LRAII^{AII}}^{-1}(\{S^{H,\mu} \}\times Rec_{2n}(\lambda/\mu))= \{S\in    SST^{\mathfrak{k}-hw}
_{2n}(\lambda)| \wt_{\mathfrak{k}}(S)=\mu\} \subseteq SST_{2n}(\lambda)$.

\item ${\LRAII^{AII}}^{-1}(\{S_{L,-\mu} \}\times Rec_{2n}(\lambda/\mu))=    \{S\in    SST^{\mathfrak{k}-lw}
_{2n}(\lambda)| \wt_{\mathfrak{k}}(S)=-\mu\} \subseteq SST_{2n}(\lambda)$.

\end{enumerate}
Therefore \begin{align}SST^{\mathfrak{k}-hw}
_{2n}(\lambda)=\bigsqcup_{\begin{smallmatrix}\mu\in Par_{\le n}\\
\mu\subseteq\lambda\end{smallmatrix}} {\LRAII^{AII}}^{-1}(\{S^{H,\mu}\}\times Rec_{2n}(\lambda/\mu))
\end{align}
and
\begin{align}SST^{\mathfrak{k}-lw}
_{2n}(\lambda)=\bigsqcup_{\begin{smallmatrix}\mu\in Par_{\le n}\\
\mu\subseteq\lambda\end{smallmatrix}} {\LRAII^{AII}}^{-1}(\{S_{L,-\mu}\}\times Rec_{2n}(\lambda/\mu)).
\end{align}
\end{prop}

\subsection{$\k$-highest and lowest weight tableaux determined by one vertical strip quantum recording tableaux}
Let $\mu\subseteq_{vert}\lambda$, $t_0$ the slack of the vertical strip $\lambda/\mu$ and $\r$ the corresponding  slack row index vector. Let $\mu'=\mu-\delta_\mathbf{r}$ and note $\lambda=\mu'+\varpi_{\ell(\lambda)}$.
%that is, $\lambda=(\lambda_0, \varpi_{\ell(\lambda)-\ell(\mu)})$ with $i_0=\ell(\lambda_0)=\ell(\mu)$.
For  $S^{H,\mu}$  the $\mathfrak{k}$-highest respectively $S_{L,-\mu}$ $\mathfrak{k}$-lowest weight tableau
in $SpT_{2n}(\mu)$,  Theorem \ref{thm:verygeneralstrip} implies the following assertion.

\begin{cor}\cite{azslack}\label{lem:H-L} Let $S^{H,\mu}$ be the $\mathfrak{k}$-highest and let $S_{L,-\mu}$ be the $\mathfrak{k}$-lowest weight tableaux  %and $S_2$ a $\mathfrak{k}$-lowest weight tableau
in $SpT_{2n}(\mu)$ as in  \eqref{2symphw-lw2}.
%Let $ Rec_{2n}(\lambda/\mu)$  be the set of  recording tableaux with shape $\lambda/\mu$   a vertical strip.
Let $Q\in Rec_{2n}(\lambda/\mu)$ with slack row index  vector $ \mathbf{r}=(r_1,\dots, r_{t_0})$ and $\mu'=\mu-\delta_\mathbf{r}$. Let $u_\r=(u_{r_1},\dots, u_{r_{t_0}})$ and $v_\r=(v_{r_1},\dots, v_{r_{t_0}})$ be sequence of positive numbers in [1,2n] as in \eqref{numbers:u}  respectively in \eqref{numbers:v}.
Then,
\begin{enumerate}
%\item
 %{\oo $\lambda_{r_i}-\lambda_{r_i+1}\ge 1$,
%and
%$\lambda_{r_i}=\lambda_{r_{i+1}}$,  only if ${r_{i+1}}={r_i}+1$, {\oo  ($\lambda_{r_i}-\lambda_{r_i+1}\ge 1$, for %${r_{i+1}}>{r_i}+1$)}, $1\le i<\ell(\mu)-l_0$.}
\item
  ${LR^{AII}}^{-1}(S^{H,\mu},Q)$ returns the following
  $\mathfrak{k}$-highest  weight tableau in
$SST_{2n}(\lambda)$ with ${\mathfrak{k}}$-weight  $\mu$:

\begin{align}\label{TUH}(T^u_0(u_{r_1})T^u_1\cdots (u_{r_{t_0}})T^u_{t_0})S^{H,\mu'}\in SST_{2n}(\lambda)
\end{align}  where  $S^{H,\mu'}=(S^{H,\mu})^1$, \eqref{revschensted},  is the  $\mathfrak{k}$-highest weight tableau in  $SpT_{2n}(\mu')$
and $(T^u_0(u_{r_1})T^u_1\cdots (u_{r_{t_0}})T^u_{t_0})$ $\in SST_{2n}(\varpi_{\ell(\lambda)})$, given by
\begin{align}
&T^u_0=(1,2,\dots l_1),\nonumber\\
&(u_{r_i})T^u_i=
\begin{cases}
 (u_{{r}_i}, u_{{r}_i}+1,\dots, u_{{r}_i}+l_{i+1}),
  &
 \mbox{if } u_{{r}_i}\in 2\mathbb{Z}\\
 (u_{{r}_i}, u_{{r}_i}+1+1,u_{{r}_i}+1+2,\dots, u_{{r}_i}+1+l_{i+1}),
 &\mbox{if } u_{{r}_i}\notin 2\mathbb{Z}\\
\end{cases}& 1\le i\le t_0,
\end{align}
with $l_1,\dots,l_{t_0+1}$ as in Theorem \ref{thm:nofactors}
%\eqref{suml}, \eqref{verygeneral0},\eqref{verygeneral200}.

\item  ${LR^{AII}}^{-1}(S_{L,-\mu},Q)$ returns the
  $\mathfrak{k}$-lowest  weight tableau in
$SST_{2n}(\lambda)$ with ${\mathfrak{k}}$-weight  $-\mu$:

\begin{align}\label{TVL}(T^v_0(v_{r_1})T^v_1\cdots (v_{r_{t_0}})T^v_{{t_0}})S_{L,-\mu'}\in SST_{2n}(\lambda)
\end{align}  where  $S_{L,-\mu'}=(S_{L,-\mu})^1$, \eqref{revschensted}, is the  $\mathfrak{k}$-lowest weight tableau in  $SpT_{2n}(\mu')$
and $(T^v_0(v_{r_1})T^v_1\cdots (v_{r_{t_0}})T^v_{{t_0}})$ $\in SST_{2n}(\varpi_{\ell(\lambda)})$, given by
\begin{align}
&T^v_0=(1,2,\dots l_1),\nonumber\\
&(v_{r_i})T^v_i=
\begin{cases}
 (v_{{r}_i}, v_{{r}_i}+1,\dots, v_{{r}_i}+l_{i+1}),
  &
 \mbox{if } v_{{r}_i}\in 2\mathbb{Z}\\
 (v_{{r}_i}, v_{{r}_i}+1+1,v_{{r}_i}+1+2,\dots, v_{{r}_i}+1+l_{i+1}),
 &\mbox{if } v_{{r}_i}\notin 2\mathbb{Z}\\
\end{cases}& 1\le i\le t_0,
\end{align}
with $l_1,\dots,l_{t_0+1}$ as in Theorem \ref{thm:nofactors}
%\eqref{suml}, \eqref{verygeneral0}, \eqref{verygeneral200}.
\end{enumerate}

\end{cor}

%This result is illustrated for $n=6$ in Example \ref{ex:qq}.
\iffalse
\begin{obs}
Given  $\mu\subseteq_{vert}\lambda$, the data given by the slack number $t_0$ and the slack row index vector $\r$ of $\lambda/\mu$ completely characterizes the pair of $\mathfrak{k}$-highest, -lowest weight tableaux in $SST_{2n}(\lambda)$ with $\mathfrak{k}$- weights $\mu$, -$\mu$).
That is, for $\mu\subseteq_{vert}\lambda$,  $S\in    SST^{\mathfrak{k}-hw}_{2n}(\lambda)$ ($SST^{\mathfrak{k}-lw}_{2n}(\lambda)$)  with $wt_{\mathfrak{k}}(S)=\mu$ ($wt_{\mathfrak{k}}(S)=-\mu$) if and only if $S=$\eqref{TUH} ($=$\eqref{TVL}).

The procedure is very easy: given $S^{H,\mu'}\in SpT_{2n}(\mu')$ and $0\le k\le n $ choose a subset $u=(u_{r_1},\dots, u_{r_k})\subseteq \{u_i\}_{k=1}^n$ \eqref{numbers:u} such that $\mu=\mu'+\delta_\r\in Par_{\le n}$ with $\r=({r_1},\dots, {r_k})\subseteq [1,n]$. Then $u=(u_{r_1},\dots, u_{r_k})\in SpT_{2n}(\varpi_k)$  and
\begin{align}
\redu_{k}^{-1}(u).S^{H,\mu'}\in SST^{\mathfrak{k}-hw}_{2n}(\lambda) \mbox{ with $\mathfrak{k}$-weight $\mu$ }
\end{align}

such that  $\lambda=\mu-\delta_\r+\varpi_l$ where $k\le l\le 2n, ~2n-l$ and  $l-k\in 2\Z$.
\end{obs}
\fi

The following example illustrates the $\k$-lowest weight  case. For the $\k$-highest weight case see \cite{azslack}.
\begin{ex} Let $n=3$, $\mu=(4,2,1)$, $t_0=1$ and $\r=\{3\}$.

 Let $v=\{v_3=5\}\subseteq \{v_1=1,v_2=4,v_3=5\}$. Let $S_{L,-\mu'}=\YT{0.13in}{}{
 {1,1,1,1},
 {4,4},
}\in SpT_6(\mu')$ with  $\mu'=\mu-\delta_\r=(4,2,1)-(0,0,1)=(4,2,0)$.
Then $\redu_1^{-1} (5)=(12345)$, $\wt_\k(12345)=(1-1)\tilde \varepsilon_1+(1-1)\tilde \varepsilon_2+(0-1)\tilde \varepsilon_3=-\tilde \varepsilon_3$,
 and $$S=(12346)\bigcdot S_{L,-\mu'}=\YT{0.13in}{}{
 {1,1,1,1,1},
 {2,4,4},
{3},
{4},
{5},
}\in SST^{\mathfrak{k}-lw}_{2\times 3}(\lambda), ~~l=\ell(\lambda)=5, ~~\wt_{\mathfrak{k}}(S)=-\mu=\wt_{\mathfrak{k}}(S_{L,-\mu})$$
\noindent where $\lambda=\mu-\delta_\r+\varpi_5$ and $1=t_0\le l\le 6, ~6-5$ and  $l-t_0=5-1=4\in 2\Z$.

%\end{enumerate}
\end{ex}

\subsection{$\k$-highest weight tableaux of shape length 2n or 2n-1 determined by quantum recording tableaux with $1$-$0$-slack sequences}

We now consider quantum recording tableaux with $1$-$0$-slack sequences.
In order to proceed, we need some notation.
Given $T\in SST_{2n}$ and $\mathbf{b}$ a column in $SST_{2n}$ we write $m_\mathbf{b}$ to mean the multiplicity of $\mathbf{b}$ in $T$.
We also write  $\circledcirc_{q}^p$, to mean  the concatenation in the plactic monoid of $q-p+1$  words, say  $w_q\bigcdot\dots\bigcdot  w_p$, with the convention that when $q-p+1\le 0$ the concatenation is the empty word.

\begin{thm}\label{thm:n=oddslack1}Let $1\le n\notin 2\Z$. Consider $S^{H,\mu}$  the $\mathfrak{k}$-highest weight tableau
in $SpT_{2n}(\mu)$ as in \eqref{noddsymphw-hw}. Let  $Q\in  Rec_{2n}(\lambda/\mu)$ with $\ell(\lambda)\in\{2n,2n-1\}$, and  $1$-$0$-slack sequence of the form $\underline \t=(1,\dots,1,0^M)$ and  corresponding  slack row index vector sequence   of the form
 $\underline \r=(1,\dots,1,2,\dots,2,3,\dots,3,\dots,n,\dots,n,2n-1,\dots,2n-1,()^M)$.  Let $u_1=2,\dots,u_n=2n$ be the numbers in \eqref{numbers:uu} or \eqref{noddsymphw-hw}, and $\mu'=\mu-\delta_{\underline\r^+}$ .
Then,
\begin{enumerate}
\item $2n-1\notin \underline\r^+$.

\item for some $0\le k\le n$,
  ${\LRAII^{AII}}^{-1}(S^{H,\mu},Q)$ returns   the $\mathfrak{k}$-highest  weight tableau in
$SST_{2n}(\lambda)$, with ${\mathfrak{k}}$-weight  $\mu$, of the form

\begin{align}\label{character:n=oddslack1}
&Y(M^{2n})\bigcdot\nonumber\\
&\bigcdot\displaystyle\circledcirc_{j=n}^{k+1}\big(\redu_1^{-1}({u_j})\big)^{m_{\redu^{-1}(u_j)}}\bigcdot
\big(\redu_1^{-1}({u_{k}})\big)^{m_{\redu^{-1}(u_k)}}\bigcdot
{\displaystyle\circledcirc_{j=k}^{1}\big(\redu_1^{-1}({u_j})\setminus\{u_n\}\big)^{m_{\redu^{-1}(u_i)\setminus \{2n\}}}}\bigcdot S^{H,\mu'},
\end{align}
where $S_{2n}^{H,\mu'}$ is the $\k$-highest weight tableau in $SpT_{2n}(\mu')$,
$$\displaystyle \sum_{j=n}^{k+1} m_{\redu^{-1}(u_j)}+m_{\redu^{-1}(u_k)}+\sum_{j=1}^k m_{\redu^{-1}(u_i)\setminus \{2n\}}=|\delta_{\underline\r}|,$$
and satisfying restrictions on linear inequalities on the multiplicities of the columns

\begin{align*} &\redu^{-1}(u_i),~~  \redu^{-1}(u_i)\setminus \{u_n\}, \mbox{ $i=2,\dots,n$, on the LHS of $S^{H,\mu'}$},\\
&\mbox{ and the symplectic columns  $u_1u_2\cdots u_i\in SpT_{2n}(\varpi_i) \eqref{noddsymphw-hw}$, $i=1,\dots,n$,}\nonumber\\
&u_1u_2\cdots u_n,~~ u_1u_2\cdots u_{n-1},\dots,u_1u_2,~~u_1
\end{align*} as follows

\begin{align}
&m_{\redu^{-1}(u_i)}+m_{\redu^{-1}(u_i)\setminus \{2n\}}\le m_{u_1\cdots u_{i-1}}, ~i=2,\dots,n,\\
&0\le m_{\redu^{-1}(u_n)}-\sum_{i=1}^{n-1}m_{\redu^{-1}(u_i)\setminus\{u_n\}}\le m_{u_1\cdots u_{n-1}}.
\end{align}

\end{enumerate}
\end{thm}
\begin{proof}
It follows from Theorem \ref{lem:recordhole1} and a detailed analysis of reverse Schensted column insertion  for the given slack vector sequence $\underline\r$.  It is easily checked that $P^{AII}(\eqref{character:n=oddslack1})=S^{H,\mu}$.
\end{proof}

\begin{cor} \label{cor:n=1}Let $S^{H,(u)}$ be the  $\mathfrak{k}$-highest weight tableau  and let $S_{L,-(u)}$ be the $\mathfrak{k}$-lowest weight tableau in $SpT_2(u)$, with $(u)$ a one-row partition, as in  \eqref{1symphw-lw2}. Let $Q\in  Rec_{2}(\lambda/(u))$ of weight $\nu=(Q[1]^M)^t=(2^M)^t=(M,M)$.% as in Lemma \ref{lem:recordhole0}.
Then for $n=1$, the $\mathfrak{k}$-highest (-lowest) weight tableaux in
$SST_{2}(\lambda)$ are respectively the hook-shape tableaux
\begin{align}&Y(M,M)\bigcdot S^{H,u},\quad \YT{0.15in}{}{
 {1,\cdots,1,{2},{2},\cdots,\cdots,\cdots,2},
{2,\cdots,2},
}\mbox{ and }\\
& Y(M,M)\bigcdot S_{L,-u},\quad \YT{0.15in}{}{
 {1,\cdots,1,{1},{1},\cdots,\cdots,\cdots,1},
{2,\cdots,2},
}
\end{align}
with ${\mathfrak{k}}$-weights respectively $(u)$ and $(-u)$.
\end{cor}
\begin{proof}The identities follow from  Lemma \ref{lem:recordhole0} and also from previous theorem with $n=1$. From the previous theorem with $n=1$, it means
\begin{align*}&\big(\redu_0^{-1}(())\big)^M\bigcdot S^{H,u}=(1,2)^M\bigcdot S^{H,u}=Y(M,M)\bigcdot S^{H,u}\\
&\big(\redu_0^{-1}(())\big)^M\bigcdot S_{L,-u}=Y(M,M)\bigcdot S_{L,-u}.
\end{align*}
\end{proof}

\begin{cor}\label{thm:n=3slack1}Let $n=3$. Consider $S^{H,\mu}$  the $\mathfrak{k}$-highest weight tableau
in $SpT_6(\mu)$ as in LHS of \eqref{3symphw-lw3}. Let  $Q\in  Rec_{6}(\lambda/\mu)$ with  slack sequence and  corresponding  slack row index vector sequence   of the form  $\underline \t=(1,\dots,1,0^M)$ respectively
 $\underline \r=(1,\cdots,1,2,\cdots,2,3,\cdots,3,5,\dots,5,()^M)$.
Then, %for $n=2$, %the possible iterated reverse column Schensted insertions  with input  $S^H$ respectively $S_L$, determined by $Q$, have  row indices sequences
%\begin{enumerate}
%\item
  ${\LRAII^{AII}}^{-1}(S^H,Q)$ returns   the $\mathfrak{k}$-highest  weight tableau in
$SST_{6}(\lambda)$ with ${\mathfrak{k}}$-weight  $\mu$    in either  form, with the following legend, circled elements indicate bumped entries from $S^{H,\mu}$, blue circle indicates the column $\redu_1^{-1}(2)$, orange circle indicates the column $\redu_1^{-1}(3)$ and brown circle indicates the column $\redu_1^{-1}(6)$:

\begin{align}\label{type31}
&(S^H)^{1,\dots,1,2,\dots,2,3,\dots,3,(5)^0,()^M}=\YT{0.15in}{}{
 {1,\cdots,1,\bro{1},\cdots,\bro{1}, \ora{1},\cdots,\ora{1},\blue{\circled{\bf 2}},\cdots,\blue{\circled{\bf 2}}, 2,\cdots, 2,2,\cdots,{2},2,\cdots,{2}},
 {2,\cdots,2,\bro{2},\cdots,2, \ora{ 2},\cdots,2,\blue{3},\cdots,\blue{3},{3},\cdots,{3},3,\cdots,3},
 {3,\cdots,3,\bro{3},\cdots,\bro{3},  \ora{\circled {\bf 3}},\cdots,\ora{\circled {\bf 3}},\blue{4},\cdots,\blue{4},{6},\cdots,{6}},
 {4,\cdots,4,\bro{4},\cdots,\bro{4},  \ora{5},\cdots,\ora{5},\blue{5},\cdots,\blue{5}},
 {5,\cdots,5,\bro{\circled{\bf 6}},\cdots,\bro{\circled{\bf 6}},  \ora{6},\cdots,\ora{6},\blue{6},\cdots,\blue{6}},
 {6,\cdots,6},
}\nonumber\\
&m_{12356}\le m_2 \mbox{ and }0\le m_{12346}\le m_{23}
\end{align}
{ or }
\begin{align}\label{type32}
&(S^H)^{1,\dots,1,2,\dots,2,3,\dots,3,(5)^a, ()^M}=\YT{0.15in}{}{
 {1,\cdots,1,\bro{1},\cdots,1,\bro{1},\cdots,1, \ora{1},\cdots,1,\blue{\circled{\bf2}},\cdots,\blue{\circled{\bf2}},\blue{\circled{\bf2}}, \cdots,\blue{\circled{\bf2}},2,\cdots, 2,2,\cdots, 2, 2,\cdots,{2}},
 {2,\cdots,2,\bro{2},\cdots,2, \bro{2},\cdots,2, \ora{2},\cdots,2,\blue{3},\cdots,3,{3},\cdots,3,3,\cdots,{3},3,\cdots,3},
 {3,\cdots,3,\bro{3},\cdots,3,\bro{3},\cdots,3,  \ora{\circled{\bf 3}},\cdots,\ora{\circled{\bf 3}},\blue{4},\cdots,4, 4,\cdots,4, 6,\cdots,{6}},
 {4,\cdots,4,\bro{4},\cdots,4, \bro{4},\cdots,4, \ora{5},\cdots,5,\blue{5},\cdots,5, 5,\cdots,5},
 {5,\cdots,5,\bro{\circled{\bf 6}},\cdots,\bro{\circled{\bf 6}},{6},\cdots,6,  {6},\cdots,6,{6},\cdots,6},
 {6,\cdots,6},
}\quad a>0\nonumber\\
&m_{12356}+m_{1235}\le m_2 \mbox{ and }0\le m_{12346}-m_{1235}-m_{2345}\le m_{23}
\end{align}
or
\begin{align}\label{type33}
&(S^H)^{1,\dots,1,2,\dots,2,3,\dots,3,5^b,()^M}=\YT{0.15in}{}{
 {1,\cdots,1,\bro{1},\cdots,1, \bro{1},\cdots,1,\ora{1},\cdots,1, \blue{\circled{\bf2}},\cdots,\blue{\circled{\bf2}}, 2,\cdots, 2, 2,\cdots,{2},2,\cdots,{2}},
 {2,\cdots,2,\bro{2},\cdots,2, \bro{2},\cdots,2,\ora{ 2},\cdots,2,\blue{3},\cdots,3,3,\cdots,{3},3,\cdots,3},
 {3,\cdots,3,\bro{3},\cdots,3,\bro{3},\cdots,3, \ora{\circled{\bf 3}},\cdots,\ora{\circled{\bf 3}}, \blue{4},\cdots,4, 6,\cdots,{6}},
 {4,\cdots,4,\bro{4},\cdots,4,\bro{4},\cdots,4,\ora{ 5},\cdots,5, \blue{5},\cdots,{5}},
 {5,\cdots,5,\bro{\circled{\bf 6}},\cdots,\bro{\circled{\bf 6}}, 6,\cdots,6,6,\cdots,6},
 {6,\cdots,6},
}\quad b>0\nonumber\\
&m_{12356}+m_{1235}\le m_2 \mbox{ and }0\le m_{12346}-m_{1235}-m_{2345}\le m_{23}
\end{align}
or\begin{align}\label{type34}
&(S^H)^{1,\dots,1,2,\dots,2,3,\dots,3,5^c,()^m}=\YT{0.15in}{}{
 {1,\cdots,1,\bro{1},\cdots,1, \bro{1},\cdots,1,\ora{1},\cdots,1,1,\cdots,1,\blue{\circled{\bf 2}},\cdots,\blue{\circled{\bf 2}}, 2,\cdots, 2, 2,\cdots,{2},2,\cdots,{2}},
 {2,\cdots,2,\bro{2},\cdots,2,\bro{2},\cdots,2, \ora{2},\cdots,2, 2,\cdots,2,\blue{3},\cdots,3,3,\cdots,{3},3,\cdots,3},
 {3,\cdots,3,\bro{3},\cdots,3, \bro{3},\cdots,3,\ora{\circled{\bf 3}},\cdots,\ora{\circled{\bf 3}},  \ora{\circled{\bf 3}},\cdots,\ora{\circled{\bf 3}},\blue{4},\cdots,4, 6,\cdots,{6}},
 {4,\cdots,4,\bro{4},\cdots,4, \bro{4},\cdots,4,\ora{5},\cdots,5,5,\cdots,5, \blue{5},\cdots,5},
 {5,\cdots,5,\bro{\circled{\bf 6}},\cdots,\bro{\circled{\bf 6}}, {6},\cdots,6,{6},\cdots,6},
 {6,\cdots,6},
}\quad c>0\nonumber\\
&m_{12356}+m_{1235}\le m_2 \mbox{ and }0\le m_{12346}-m_{1235}-m_{2345}\le m_{23}
\end{align}
or\begin{align}\label{type35}
&(S^H)^{1,\dots,1,2,\dots,2,3,\dots,3,5^d,()^M}=\YT{0.15in}{}{
 {1,\cdots,1,\bro{1},\cdots,1,\bro{1},\cdots,1,\ora{1},\cdots,1, \blue{\circled{\bf 2}},\cdots,\blue{\circled{\bf 2}}, 2,\cdots, 2, 2,\cdots,{2},2,\cdots,{2}},
 {2,\cdots,2,\bro{2},\cdots,2,\bro{2},\cdots,2, \ora{2},\cdots,2,\blue{3},\cdots,3,3,\cdots,{3},3,\cdots,3},
 {3,\cdots,3,\bro{3},\cdots,3,\bro{3},\cdots,3, \ora{\circled{\bf 3}},\cdots, \ora{\circled{\bf 3}},\blue{4},\cdots,4, 6,\cdots,{6}},
 {4,\cdots,4,\bro{4},\cdots,4,\bro{4},\cdots,4, \ora{5},\cdots,5, \blue{5},\cdots,5},
 {5,\cdots,5,\bro{\circled{\bf 6}},\cdots,\bro{\circled{\bf 6}},6,\cdots,6},
 {6,\cdots,6},
}\quad d>0\nonumber\\
&m_{12{\bf 3}56}+m_{1235}\le m_{\bf 2} \mbox{ and }0\le m_{1234\bf 6}-m_{1235}-m_{2345}\le m_{2\bf 3}
\end{align}

or\begin{align}\label{type36}
&(S^H)^{1,\dots,1,2,\dots,2,3,\dots,3,5^e,()^M}=\YT{0.15in}{}{
 {1,\cdots,1,\bro{1},\cdots,1,\bro{1},\cdots,1,1,\cdots,1,\ora{1},\cdots,1,   \blue{\circled{\bf 2}},\cdots,\blue{\circled{\bf 2}}, 2,\cdots, 2, 2,\cdots,{2},2,\cdots,{2}},
 {2,\cdots,2,\bro{2},\cdots,2,\bro{2},\cdots,2,2,\cdots,2, \ora {2},\cdots,2,   \blue{3},\cdots,3,3,\cdots,{3},3,\cdots,3},
 {3,\cdots,3,\bro{3},\cdots,3,\bro{3},\cdots,3,  3,\cdots,3,\ora{\circled{\bf 3}},\cdots,\ora{\circled{\bf 3}},\blue{4},\cdots,4, 6,\cdots,{6}},
 {4,\cdots,4,\bro{4},\cdots,4,\bro{4},\cdots,4, 4,\cdots,4,\ora{5},\cdots,5, \blue{5},\cdots,5},
 {5,\cdots,5,\bro{\circled{\bf 6}},\cdots,\bro{\circled{\bf 6}},6,\cdots,6},
 {6,\cdots,6},
}\nonumber\\
&\quad e>0\nonumber\\
&m_{12356}+m_{1235}\le m_2 \mbox{ and }0\le m_{12346}-m_{1235}-m_{2345}\le m_{23}
\end{align}
%\end{enumerate}

These tableaux satisfy inequalities  on the multiplicities % let $x$ be the multiplicity of
of the following columns: if $m_{12356}$ is the multiplicity of $\redu^{-1}(3)=(12356)$, $m_{12346}$ the multiplicity of $\redu^{-1}(6)=(12346)$,
  $m_{23}$ the multiplicity of $(2,3)\in SpT_6(\varpi_2)$,  and $m_{2}$ the multiplicity of $(2)\in SpT_6(\varpi_1) $,
 then

\begin{align}m_{12{\bf 3}56}+m_{1235}\le m_{\bf 2} \mbox{ and }0\le m_{1234\bf 6}-m_{1235}-m_{2345}\le m_{2\bf 3}.\label{ineqn=3}
\end{align}
%{\oo Tableaux of type \eqref{type1} are produced with slack vectors where $\#3$'s$>\#1$'$\ge 0$, and those of type \eqref{type2} with slack %vectors where $0\le\#3$'s$\le\#1$'.}
%\end{enumerate}
\end{cor}

\begin{proof}
From \eqref{kweight}, for $n=3$, the $\k$-weight of $S\in SST_{4}$ is equal to $$\wt_\k(S)=(S[2]-S[1])\tilde\varepsilon_1+(S[3]-S[4])\tilde\varepsilon_2+(S[6]-S[5])\tilde\varepsilon_3.$$  The tableau patterns
 are made of columns $(123456)$, $(12356)$, $(12346)$, $(23456)$, $(1234)$,  $(1235)$, $2345$ $236$, $23$ and $(2)$. In fact,
 $$(123456)=\redu^{-1}(()) \mbox{  with } \wt_\k()=0$$

  $$(12356)=\redu^{-1}((3)) \mbox{  with } \wt_\k(3)=1\tilde\varepsilon_2$$

  $$(12346)=\redu^{-1}((6)) \mbox{ with } \wt_\k(6)=\tilde \varepsilon_3$$

  $$(23456)=\redu^{-1}((2)) \mbox{  with } \wt_\k(23456)=\tilde\varepsilon_1$$ and

  $$ \c\circ(\redu^{-1},id)\circ(\underset{5}\leftarrow (12346))=(12346)(1234),\wt_\k(1234)=(0,0,0)$$
  $$ \c\circ(\redu^{-1},id)\circ(\underset{5}\leftarrow (12356))=(12346)(1235),\wt_\k(1235)=1\tilde\varepsilon_2-\tilde\varepsilon_3$$
  $$ \c\circ(\redu^{-1},id)\circ(\underset{5}\leftarrow (23456))=(12346)(2345),\wt_\k(2345)=\tilde\varepsilon_1-1\tilde\varepsilon_3$$
\end{proof}

For an illustration of this corollary see \cite{azslack}.
\iffalse
\fi

\begin{thm}\label{thm:n=evenslack1}Let $2\le n\in 2\Z$. Consider $S^{H,\mu}$  the $\mathfrak{k}$-highest weight tableau
in $SpT_{2n}(\mu)$ as in \eqref{nevensymphw-hw}. Let  $Q\in  Rec_{2n}(\lambda/\mu)$, $ \ell(\lambda)\in \{2n-1,2n\}$, with  $0$-$1$ slack sequence of the form  $\underline \t=(1,\dots,1,0^M)$ and  corresponding  slack row index vector sequence
 $\underline \r=(1,\dots,1,2,\dots,2,3,\dots,3,\dots,n,\dots,n,2n-1,\dots,2n-1,()^M)$. Let $u_1=2,\dots,u_n=2n-1$ be the numbers in \eqref{numbers:uu} or \eqref{nevensymphw-hw}, and $\mu'=\mu-\delta_{\underline \r^+}$.
Then,
\begin{enumerate}
\item $2n-1\notin \underline \r^+$

\item for some $0\le k\le n$,
 ${\LRAII^{AII}}^{-1}(S^{H,\mu},Q)$ returns   the $\mathfrak{k}$-highest  weight tableau in
$SST_{2n}(\lambda)$ with ${\mathfrak{k}}$-weight  $\mu$,    in either  form:

\begin{align}\label{record1}
%&Y(M^{2n})\bigcdot\nonumber\\
%&\bigcdot\displaystyle\circledcirc_{j=n}^{k+1}\big(\redu_1^{-1}({u_j})\big)^{m_{\redu^{-1}(u_j)}}\bigcdot
%\big(\redu_1^{-1}({u_{k}})\big)^{m_{\redu^{-1}(u_k)}}\bigcdot\\
&Y(M^{2n})\bigcdot \big(\redu_1^{-1}(2n-1))^{m_{\redu^{-1}(2n-1) }}\bigcdot
{\displaystyle\circledcirc_{j=n-1}^{1}\big(\redu_1^{-1}(u_j)\big)^{m_{\redu^{-1}(u_j) }}}
\bigcdot S^{H,\mu'}\\
=&Y(M^{2n})\bigcdot \big((12\cdots (2n-3).(2n-2).(2n-1)\big)^{m_{12\cdots (2n-3).(2n-2).2n-1 }}\bigcdot
{\displaystyle\circledcirc_{j=n-1}^{1}\big(\redu_1^{-1}(u_j)\big)^{m_{\redu^{-1}(u_j) }}}
\bigcdot S^{H,\mu'},
\end{align}

or

\begin{align}\label{record2}
&Y(M^{2n})\bigcdot \big(\redu_1^{-1}(2n-1)\big)^{m_{\redu^{-1}(2n-1) }}
\bigcdot\big(12\cdots (2n-2)\cdot 2n\big)^{m_{12\cdots 2n-2\cdot 2n}}
\bigcdot
\big(12\cdots (2n-2)\big)^{m_{(12\cdots (2n-2)) }}\bigcdot
\nonumber\\
&
%\big((\redu_1^{-1}(2n-1))\big)^{m_{\redu^{-1}(2n-1) }}
\bigcdot
{\displaystyle\circledcirc_{j=n-1}^{1}\big(\redu_1^{-1}(u_j)\setminus\{2n\}\big)^{m_{\redu^{-1}(u_j)\setminus\{2n\} }}}
\bigcdot S^{H,\mu'}\\
=&Y(M^{2n})
\bigcdot\big((12\cdots (2n-3).(2n-2).(2n-1)\big)^{m_{12\cdots (2n-3).(2n-2).2n-1 }}\bigcdot\big(12\cdots (2n-2)\cdot 2n\big)^{m_{12\cdots 2n-2\cdot 2n}}\nonumber\\
&\bigcdot
\big(12\cdots (2n-2)\big)^{m_{12\cdots (2n-2) }}\bigcdot
{\displaystyle\circledcirc_{j=n-1}^{1}\big(\redu_1^{-1}(u_j)\setminus\{2n\}\big)^{m_{\redu^{-1}(u_j)\setminus\{2n\} }}}
\bigcdot S^{H,\mu'}
\end{align}

or

\begin{align}\label{record3}
&Y(M^{2n})\bigcdot \big(\redu_1^{-1}(2n-1)\big)^{m_{\redu^{-1}(2n-1) }}
\bigcdot\big(12\cdots (2n-2)\cdot 2n\big)^{m_{12\cdots 2n-2\cdot 2n}}
\bigcdot\nonumber\\
%\big(\redu_1^{-1}(2n-1)\setminus \{2n-1\}\big)^{m_{\redu^{-1}(2n-1)\setminus\{2n-1\} }}\bigcdot
&\bigcdot\displaystyle\circledcirc_{j=n-1}^{k+1}\big(\redu_1^{-1}({u_j})\big)^{m_{\redu^{-1}(u_j)}}
\bigcdot\big(\redu_1^{-1}({u_{k}})\big)^{m_{\redu^{-1}(u_k)}}\bigcdot\big(\redu_1^{-1}(u_k)\setminus\{2n\}\big)^{m_{\redu^{-1}(u_k)\setminus\{2n\} }}\nonumber\\
%\big((\redu_1^{-1}(2n-1))\big)^{m_{\redu^{-1}(2n-1) }}
&\bigcdot
{\displaystyle\circledcirc_{j=k-1}^{1}\big(\redu_1^{-1}(u_j)\setminus\{2n\}\big)^{m_{\redu^{-1}(u_j)\setminus\{2n\} }}}
\bigcdot S^{H,\mu'},
\end{align}

where $\displaystyle \sum_{j=n}^{k+1} m_{\redu^{-1}(u_j)}+m_{\redu^{-1}(u_k)}+\sum_{j=k}^1 m_{\redu^{-1}(u_i)\setminus \{2n\}}=|\underline\delta_\r|-M$,
and
satisfying the restrictions on linear inequalities on the multiplicities of the columns

\begin{align} &\redu^{-1}(u_i),~~  \redu^{-1}(u_i)\setminus \{2n\}, \mbox{ $i=1,2,\dots,n$, } \mbox{ on LHS of $S^{H,\mu'},$}\nonumber\\
&\mbox{ and the symplectic columns  $u_1u_2\cdots u_i\in SpT_{2n}(\varpi_i) \eqref{nevensymphw-hw}$, $i=1,\dots,n$,}\nonumber\\
&u_1u_2\cdots u_n,~ u_1u_2\cdots u_{n-1},\dots,u_1u_2,~~u_1\nonumber
\end{align} as follows

\begin{align*}
&m_{\redu^{-1}(u_i)}+m_{\redu^{-1}(u_i)\setminus \{2n\}}\le m_{u_1\cdots u_{i-1}}, ~i=2,\dots,n\\
 %& \oo m_{1234568}= m_{12{\bf 3}567}+ m_{1234{\bf 6}7}+ m_{123456\bf 7}\\
 & m_{12\cdots (2n-2)2n}
 =\sum_{i=1}^{n-1} m_{\redu^{-1}(u_i)\setminus\{2n\}}.
 %+\redu^{-1}(u_2)\setminus\{2n\}+\redu^{-1}(u_3)\setminus\{2n\}+
 %\cdots+\redu^{-1}(u_{n-1})\setminus\{2n\}.
 %
\end{align*}

\end{enumerate}
\end{thm}
\begin{proof}
Follows from Theorem \ref{lem:recordhole1} and a detailed analysis of reverse Schensted column insertion  for the given slack vector sequence $\underline\r$. It is easily checked that $ P^{AII}(\eqref{record1})=P^{AII}(\eqref{record2})=P^{AII}(\eqref{record3})=S^{H,\mu}$.
\end{proof}

%From Proposition \ref{prop:slack}, %Remark \ref{re:slack}
   %Lemma \ref{lem:recordhole1} and previous theorem, it follows

\begin{cor}\label{cor:n=2}Let $n=2$. Consider $S^H$  the $\mathfrak{k}$-highest weight tableau
in $SpT_4(\mu)$ as in LHS of \eqref{2symphw-lw2}. Let $Q\in  Rec_{4}(\lambda/\mu)$.
%  as in Lemma \ref{lem:recordhole1}.
Then, %for $n=2$, %the possible iterated reverse column Schensted insertions  with input  $S^H$ respectively $S_L$, determined by $Q$, have  row indices sequences
\begin{enumerate}
\item the slack sequence and the corresponding  slack row index vector sequence of $Q$ are of the form  $\underline \t=(1,\dots,1,0,\dots,0)$, respectively
 $\underline \r=(1,\cdots,1,2,\cdots,2,3,\cdots,3,(),\dots,())$, and
  \item ${\LRAII^{AII}}^{-1}(S^H,Q)$ returns   the $\mathfrak{k}$-highest  weight tableau in
$SST_{4}(\lambda)$ with ${\mathfrak{k}}$-weight  $\mu$    as   described in \cite[Lemma 6.2,(c)]{nsw}:
\begin{align}\label{type2}
&\YT{0.16in}{}{
 {1,\cdots,1,{\bro 1},\cdots,{\bro 1},{\bro 1},\cdots,{\bro 1},                                 {\blue\circled {2}},\cdots,{ \blue\circled {2}},  {\blue\circled 2},\cdots,{\blue\circled {2}},{2},\cdots,{2},{2},\cdots,{2}},
 {2,\cdots,2,{\bro 2},\cdots,{\bro 2},{\bro 2},\cdots,{\bro 2},                                  {\blue 3},\cdots,{\blue 3},          {\blue 3},\cdots,{\blue 3},{3},\cdots,{3}},
 {3,\cdots,3,\bro{\circled{\bf 3}},\cdots,{\bro \circled{\bf 3}}, {{\blue 4}},\cdots,{\blue 4},{\blue 4},\cdots,{\blue 4}},
 {4,\cdots,4},
}
\end{align}
\mbox{ or }

\begin{align}
\label{type1}
&\YT{0.16in}{}{
 {1,\cdots,1, {\bro 1},\cdots,1,                       {\bro 1},\cdots,{\bro 1},{1},\cdots,1,{\blue\circled{\bf 2}},\cdots,
 \blue{\circled{\bf 2}},{2},\cdots,{2},{2},\cdots,{2}},
 {2,\cdots,2,{\bro 2},\cdots,{\bro 2},                          {\bro 2},\cdots,{\bro 2}, {2},\cdots,2,{{\blue 3}},\cdots,{\blue 3},{3},\cdots,{3}},
 {3,\cdots,3,\bro{\circled{\bf 3}},\cdots,\bro{\circled{\bf 3}}, {\blue 4},\cdots,{\blue 4}},
 {4,\cdots,4},
}  %\lambda_3-\lambda\le \lambda_1-\lambda_2;
\end{align}
\end{enumerate}

These tableaux satisfy inequalities  on the multiplicities % let $x$ be the multiplicity of
of the following columns: if $m_{123}$ is the multiplicity of $\redu^{-1}(3)=(1,2,3)$, $m_{124}$ the multiplicity of $\redu^{-1}(4)=(1,2,4)$, $m_{23}$ the multiplicity of $(2,3)\in SpT_4(\varpi_2)$ and $m_2$ the multiplicity of $(2)\in SpT_4(\varpi_1) $, then

\begin{align}m_{12\bf 3}\le m_{\bf 2} \mbox{ and } m_{12\bf 4}\le m_{2\bf 3}.\label{ineqn=2}
\end{align}
Tableaux of type \eqref{type1} are produced with slack vectors where $\#3$'s$>\#1$'$\ge 0$, and those of type \eqref{type2} with slack vectors where $0\le\#3$'s$\le\#1$'.
\end{cor}

\begin{proof}  They are particular cases with $n=2$ of the previous theorem.

The pattern \eqref{type1} appears after the pattern \eqref{type2} by increasing enough the  number of $3$'s
in $\underline \r=(1,\dots,1,2,\dots,2,3,\dots,3)$ once pattern \eqref{type2} appears:
\begin{enumerate}
\item columns $(234)$ are produced with the $1$'s in $\underline\r$ and then columns $(123)$ are obtained with the $2$'s in $\underline\r$. The number of columns $(123)$ in the pattern just obtained $S^{(1,\dots,1,2,\dots,2)}$ is equal to the multiplicity of $2$'s in $\underline\r$. Therefore  $m_{123}\le m_{2}$ in $S^{(1,\dots,1,2,\dots,2)}$.
At this point we have $$S^{(1,\dots,1,2,\dots,2)}$$ whose pattern \eqref{type2} is of either form with $m_{123}\le m_2 $ and $0= m_{124}\le m_{23}.$
\begin{align}(A):\quad
&\c\circ(\redu^{-1},id)\circ(\underset{1}\leftarrow \YT{0.13in}{}{
 {{2}},
})
=\YT{0.13in}{}{
 {{2}},
 {3},
 {4}
}=S^{(1)} \mbox{ of type \eqref{type2}}
%\iffalse
%&\c\circ(\redu^{-1},id)\circ(\underset{2}\leftarrow \YT{0.13in}{}{
% {1,{2}},
% {{2}},
% {3}
%})= \YT{0.13in}{}{
% {1,{2}},
% {2},
% {3}
%}=S^{1,2}\label{} \mbox{ of type \eqref{type2}}
%\fi
\end{align}
or
\begin{align}\nonumber (B):\quad
&\c\circ(\redu^{-1},id)\circ(\underset{2}\leftarrow \YT{0.13in}{}{
 {{2}},
 {{3}},
})
=\YT{0.13in}{}{
 {1,{2}},
 {{2}},
 {3}
}=S^{(2)} \mbox{ of type \eqref{type2}}
%\iffalse
%&\c\circ(\redu^{-1},id)\circ(\underset{3}\leftarrow \YT{0.13in}{}{
% {1,{2}},
% {{2}},
% {3}
%})= \YT{0.13in}{}{
% {1,1,{2}},
 %{2,{2}},
% {3}
%}=S^{2,3}\label{} \mbox{ of type \eqref{type1}}
%\fi
\end{align}
$(B)$ indicates that  $0= m_{124}$, $m_{123}\le m_2$.

or
\begin{align}\nonumber (C):\quad
&\c\circ(\redu^{-1},id)\circ(\underset{1}\leftarrow \YT{0.13in}{}{
 {2,2},
 {{3}},
})
=\YT{0.13in}{}{
 {2,{2}},
 {{3},3},
 {4}
}=S^{(1)}\\
&\c\circ(\redu^{-1},id)\circ(\underset{2}\leftarrow S^{(1)}= \YT{0.13in}{}{
 {1,{2},2},
 {{2},3},
 {3,4},
}
%= \YT{0.13in}{}{
% {1,1,{2}},
% {2,{2}},
 %{3}
%}
=S^{1,2}\label{small0} \mbox{ of type \eqref{type2}}\\
& 0= m_{124},   \quad m_{123}\le m_2\nonumber
\end{align}
or
\begin{align}\nonumber (D):\quad
&\c\circ(\redu^{-1},id)\circ(\underset{1}\leftarrow \YT{0.13in}{}{
 {2,2,2},
 {{3}},
})
=\YT{0.13in}{}{
 {2,{2},2},
 {{3},3},
 {4}
}=S^{(1)}\nonumber\\
&\c\circ(\redu^{-1},id)\circ(\underset{1}\leftarrow S^{(1)})=\YT{0.13in}{}{
 {2,{2},2},
 {{3},3,3},
 {4,4},
}=S^{1,1}\nonumber\\
&\c\circ(\redu^{-1},id)\circ(\underset{2}\leftarrow S^{1,1})= \YT{0.13in}{}{
 {1,2,{2},2},
 {2,{3},3},
 {3,4,4},
}
%= \YT{0.13in}{}{
% {1,1,{2}},
% {2,{2}},
 %{3}
%}
=S^{1,1,2}\label{small00} \mbox{ of type \eqref{type2}}\\
& 0= m_{124},\quad m_{123}\le m_2\nonumber
\end{align}
\item passing from $\underline \r=(1,\dots,1,2,\dots,2)$ to $\underline \r=(1,\dots,1,2,\dots,2,3,\dots,3)$ where $\# 3$'s$>\#1$'s, we reach pattern \eqref{type1} with $m_{123}\le m_2 $ and $0\le m_{124}\le m_{23}.$  Note with $\# 3$'s$=\#1$'s we remain in pattern \eqref{type2}. From above we respectively  get
    \begin{align}(A):\quad
   & \c\circ(\redu^{-1},id)\circ(\underset{3}\leftarrow \YT{0.13in}{}{
 {{2}},
 {3},
 {4},
}=S^{(1)})=\YT{0.13in}{}{
 {1,{2}},
 {2,3},
 {4},
 }=S^{1,3},\nonumber\\
 & \c\circ(\redu^{-1},id)\circ(\underset{3}\leftarrow S^{1,3})= \YT{0.13in}{}{
 {1,1,{2}},
 {2,2,3},
 {4},
})=S^{1,3,3}\nonumber\\
&1=m_{124}\le m_{23},\quad m_{123}\le m_2
\end{align}

\begin{align}\nonumber (B):\quad
&\c\circ(\redu^{-1},id)\circ(\underset{3}\leftarrow \YT{0.13in}{}{
 {1,{2}},
 {{2}},
 {3}
}=S^2)
=\YT{0.13in}{}{
 {1,1,{2}},
 {{2},2},
 {3},
}=S^{2,3}
 \end{align}
  \begin{align} (C):\quad
 & \c\circ(\redu^{-1},id)\circ(\underset{3}\leftarrow S^{1,2}= \YT{0.13in}{}{
 {1,{2},2},
 {{2},3},
 {3,4},
}  )=\YT{0.13in}{}{
 {1,1,{2},2},
 {2,{2},3},
 {3,4},
}=S^{1,2,3},\nonumber\\
& \c\circ(\redu^{-1},id)\circ(\underset{3}\leftarrow S^{1,2,3,3})= \YT{0.13in}{}{
 {1,1,1,{2},2},
 {2,2,{2},3},
 {3,4},
}\nonumber\\
&1=m_{124}\le m_{23},\quad m_{123}\le m_2
\end{align}

\begin{align} (D):\quad
&\c\circ(\redu^{-1},id)\circ(\underset{3}\leftarrow  \YT{0.13in}{}{
 {1,2,{2},2},
 {2,{3},3},
 {3,4,4},
}
=S^{1,1,2})=\YT{0.13in}{}{
 {1,1,2,{2},2},
 {2,2,{3},3},
 {3,4,4},
}=S^{1,1,2,3},\nonumber\\
&1=m_{124}\le m_{23}\nonumber\\
&\c\circ(\redu^{-1},id)\circ(\underset{3}\leftarrow  S^{1,1,2,3})=S^{1,1,2,3,3})=\YT{0.13in}{}{
 {1,1,1,2,{2},2},
 {2,2,2,{3},3},
 {3,4,4},
}\nonumber\\
&2=m_{124}\le m_{23}, \quad m_{123}\le m_2\\
&\c\circ(\redu^{-1},id)\circ(\underset{3}\leftarrow S^{1,1,2,3,3})=S^{1,1,2,3,3,3}=\YT{0.13in}{}{
 {1,1,1,1,2,{2},2},
 {2,2,2,2,{3},3},
 {3,4,4},
}\nonumber\\
&2=m_{124}\le m_{23},\quad m_{123}\le m_2
\end{align}
\end{enumerate}
From \eqref{kweight}, for $n=2$, the $\k$-weight of $S\in SST_{4}$ is equal to $\wt_\k(S)=(S[2]-S[1],S[3]-S[4])$. The tableau patterns
\eqref{type2} and \eqref{type1} are made of columns $(1234)$, $(123)$, $(124)$, $(12)$, $(234)$, $(23)$ and $(2)$. In fact, $$(1234)=\redu^{-1}(()) \mbox{ with } \wt_\k()=0$$
$$(1,2,3)=\redu^{-1}((3)) \mbox{ with } \wt_\k(3)=1\tilde e_2$$
$$(1,2,4)=\redu^{-1}((4)) \mbox{ with } \wt_\k(4)=-1\tilde e_2$$

$$(2,3,4)=\redu^{-1}((2)) \mbox{ with } wt_\k(234)=(1,0)$$ and  $$\wt_\k(12)=(0,0)$$
\iffalse
Therefore the $\k$-weight either of \eqref{type2} or \eqref{type1} is equal to  $\mu=(\mu_1,\mu_2)$ where $\mu_1$ is equal to the number of $2$'s in the first row of \eqref{type2} or \eqref{type1} plus the number of columns $(234)$, that is, $\mu_1=m_2+m_{23}+m_{234}$; and $\mu_2$  is equal to the number of $3$'s in the second row plus $m_{123}-m_{124}$ (the number of columns $(123)$ minus the number of columns $(124)$ in those tableaux), that is, $\mu_2=m_{123}+(m_{23}-m_{124})$. Inequalities \eqref{ineqn=2} guarantee that $\mu$ is  a partition:

\begin{align}\mu=&m_{1234}\wt_\k(1,2,3,4)+m_{123}\wt_\k(1,2,3)+m_{124}\wt_\k(1,2,4)+
m_{234}\wt_\k(2,3,4)+m_{12}\wt_\k(1,2)+\mu'\nonumber\\
=&m_{1234}(0,0)+m_{123}\wt_\k(3)+m_{124}\wt_\k(4)+m_{234}\wt_\k(2)+m_{12}\wt_\k(0,0)+\mu'\nonumber\\
=&m_{123}(0,1)+m_{124}(0,-1)+m_{234}(1,0)+\mu'\nonumber\\
=&(m_2+(m_{23}+m_{234}),m_{123}+(m_{23}-m_{124})),~ ~\mbox{ where $m_{123}\le m_2$  and  $m_{124}\le m_{23}$, by \eqref{ineqn=2}}.
\end{align}
where $\mu'=m_{23}\varpi_2+m_2\varpi_1=(m_{23}+m_2,m_{23})$.
\fi
\end{proof}

\begin{cor}\label{thm:n=4slack1}Let $n=4$. Consider $S^H$  the $\mathfrak{k}$-highest weight tableau
in $SpT_8(\mu)$ as in LHS of \eqref{4symphw-lw4}. Let  $Q\in  Rec_{8}(\lambda/\mu)$ with  slack sequence and  corresponding  slack row index vector sequence   of the form  $\underline \t=(1,\dots,1,0^M)$ respectively
 $\underline \r=(1,\dots,1,2,\dots,2,3,\dots,3,4\dots,4,7,\dots,7,()^M)$.
Then,
 ${\LRAII^{AII}}^{-1}(S^{H,\mu},Q)$ returns   the $\mathfrak{k}$-highest  weight tableau in
$SST_{8}(\lambda)$ with ${\mathfrak{k}}$-weight  $\mu$     either of the form:

\begin{align}\label{type41}
&(S^H)^{1,\dots,1,2,\dots,2,3,\dots,3,4,\dots,4,7^0,()^m}=
\YT{0.15in}{}{
 {1,\cdots,1,\bro{1},\cdots,\bro{1},\ora{1},\cdots,\ora{1}, \blue{1},\cdots,1,\red{\bf \circled{\bf 2}},\cdots,\red{\bf \circled{\bf 2}},2,\cdots, 2, 2,\cdots, 2,2,\cdots,{2},2,\cdots,{2}},
 {2,\cdots,2,\bro{2},\cdots,\bro{2},   \ora{2}, \cdots,\ora{2},         \blue{2},\cdots,2,\red{3},\cdots,3,3,\cdots,3,{3},\cdots,{3},3,\cdots,3},
 {3,\cdots,3,\bro{3},\cdots,\bro{3},   \ora{3},\cdots,\ora{3},       \blue{ \circled{\bf 3}},\cdots,\blue{ \circled{\bf 3}},\red{4},\cdots,4,6,\cdots,6,{6},\cdots,{6}},
 {4,\cdots,4,\bro{4},\cdots,\bro{4},\ora{4},\cdots,\ora{4},        \blue{5},\cdots,5,\red{5},\cdots,5,7,\cdots,7},
 {5,\cdots,5,\bro{5 },\cdots,\bro{5},\ora{\bf\circled {\bf 6}},\cdots,\ora{\bf\circled {\bf 6}},  \blue{6},\cdots,6,\red{6},\cdots,6},
 {6,\cdots,6,\bro{6},\cdots,\bro{6},\ora{7},\cdots,\ora{7},          \blue{7},\cdots,{7},\red{7},\cdots,7},
 {7,\cdots,7,\bro{ \bf\circled{\bf 7}},  \cdots,\bro{ \bf\circled{\bf 7}},   \ora{8},\cdots,\ora{8},  \blue{8},\cdots,8,\red{8},\cdots,8},
 {8,\cdots,8},
}\nonumber\\
 &m_{12{\bf 3}5678}\le m_2, \quad 0\le m_{1234{\bf 6}78}\le m_{23},\quad 0\le m_{123456\bf 7}\le m_{236}\nonumber\\
  &m_{1234568}=  m_{{\bf 2}34567}+m_{12{\bf 3}567}+ m_{1234{\bf 6}7}
\end{align}
 or

  %{\oo missing column}
\begin{align}\label{type42}
&(S^H)^{1,\dots,1,2,\dots,2,3,\dots,3,7^a}=\nonumber\\
&\YT{0.15in}{}{
 {1,\cdots,1,\bro{1},\cdots,\bro{1},\bro{1},\cdots,\bro{1},      \ora{1},\cdots,\ora{1}, \blue{1},\cdots,1,
 \red{\bf \circled{\bf 2}},\cdots,\red{\bf \circled{\bf 2}},2,\cdots, 2, 2,\cdots, 2,2,\cdots,{2},2,\cdots,{2}},
 {2,\cdots,2,\bro{2},\cdots,\bro{2}, \bro{2},\cdots,\bro{2},  \ora{2}, \cdots,\ora{2},         \blue{2},\cdots,2,\red{3},\cdots,3,3,\cdots,3,{3},\cdots,{3},3,\cdots,3},
 {3,\cdots,3,\bro{3},\cdots,\bro{3}, \bro{3},\cdots,\bro{3},  \ora{3},\cdots,\ora{3},       \blue{ \circled{\bf 3}},\cdots,\blue{ \circled{\bf 3}},\red{4},\cdots,4,6,\cdots,6,{6},\cdots,{6}},
 {4,\cdots,4,\bro{4},\cdots,\bro{4},\bro{4},\cdots,\bro{4},\ora{4},\cdots,\ora{4},        \blue{5},\cdots,5,\red{5},\cdots,5,7,\cdots,7},
 {5,\cdots,5,\bro{5 },\cdots,\bro{5},\bro{5},\cdots,\bro{5},\ora{\bf\circled {\bf 6}},\cdots,\ora{\bf\circled {\bf 6}},  \blue{6},\cdots,6,\red{6},\cdots,6},
 {6,\cdots,6,\bro{6},\cdots,\bro{6},\bro{6},\cdots,\bro{6},\ora{7},\cdots,\ora{7},          \blue{7},\cdots,{7},\red{7},\cdots,7},
 {7,\cdots,7,\bro{\bf \circled{\bf 7}},\cdots,\bro{ \bf\circled{\bf 7}},   {8},\cdots,{8},  {8},\cdots,8,{8},\cdots,8},
 {8,\cdots,8},
},\quad %a>0
\nonumber\\
 &m_{12{\bf 3}5678}+m_{12{\bf 3}567}\le m_2, \quad 0\le m_{1234{\bf 6}78}+m_{1234{\bf 6}7}\le m_{23},\quad 0\le m_{123456\bf 7}\le m_{236}\\
  &m_{1234568}=  m_{{\bf 2}34567}+m_{12{\bf 3}567}+ m_{1234{\bf 6}7}
\end{align}
%{\oo missing column}
or
\begin{align}\label{type43}
&(S^H)^{1,\dots,1,2,\dots,2,3,\dots,3,7^b,()^m}=\nonumber\\
&\YT{0.14in}{}{
 {1,\cdots,1,\bro{1},\cdots,\bro{1},\bro{1},\cdots,\bro{1},      \ora{1},\cdots,\ora{1}, \blue{1},\cdots,1,\red{\bf \circled{\bf 2}},\cdots,\red{\bf \circled{\bf 2}},2,\cdots, 2, 2,\cdots, 2,2,\cdots,{2},2,\cdots,{2}},
 {2,\cdots,2,\bro{2},\cdots,\bro{2}, \bro{2},\cdots,\bro{2},  \ora{2}, \cdots,\ora{2},         \blue{2},\cdots,2,\red{3},\cdots,3,3,\cdots,3,{3},\cdots,{3},3,\cdots,3},
 {3,\cdots,3,\bro{3},\cdots,\bro{3}, \bro{3},\cdots,\bro{3},  \ora{3},\cdots,\ora{3},
 \blue{ \circled{\bf 3}},\cdots, \blue{ \bf\circled{\bf 3}},\red{4},\cdots,4,6,\cdots,6,{6},\cdots,{6}},
 {4,\cdots,4,\bro{4},\cdots,\bro{4},\bro{4},\cdots,\bro{4},\ora{4},\cdots,\ora{4},        \blue{5},\cdots,5,\red{5},\cdots,5,7,\cdots,7},
 {5,\cdots,5,\bro{5 },\cdots,\bro{5},\bro{5},\cdots,\bro{5},\ora{\bf\circled {\bf 6}},\cdots,\ora{\bf\circled {\bf 6}},  \blue{6},\cdots,6,\red{6},\cdots,6},
 {6,\cdots,6,\bro{6},\cdots,\bro{6},\bro{6},\cdots,\bro{6},\ora{7},\cdots,\ora{7},          \blue{7},\cdots,{7},\red{7},\cdots,7},
 {7,\cdots,7,\bro{\bf \circled{\bf 7}},\cdots,\bro{\bf \circled{\bf 7}},     {8},\cdots,8,{8},\cdots,8},
 {8,\cdots,8},
},\quad
b>0
\nonumber\\
&m_{12{\bf 3}5678}+m_{12{\bf 3}567}\le m_2, \quad 0\le m_{1234{\bf 6}78}+m_{1234{\bf 6}7}\le m_{23},\quad 0\le m_{123456\bf 7}\le m_{236},\quad\\
  &m_{1234568}=  m_{{\bf 2}34567}+m_{12{\bf 3}567}+ m_{1234{\bf 6}7}
\end{align}
%{\oo missing column}

or\begin{align}\label{type44}
&(S^H)^{1,\dots,1,2,\dots,2,3,\dots,3,7^c, ()^m}=\nonumber\\
&\YT{0.15in}{}{
 {1,\cdots,1,\bro{1},\cdots,\bro{1},\bro{1},\cdots,\bro{1},      \ora{1},\cdots,\ora{1}, \blue{1},\cdots,1,\red{\bf \circled{\bf 2}},\cdots,\red{\bf \circled{\bf 2}},2,\cdots, 2, 2,\cdots, 2,2,\cdots,{2},2,\cdots,{2}},
 {2,\cdots,2,\bro{2},\cdots,\bro{2}, \bro{2},\cdots,\bro{2},  \ora{2}, \cdots,\ora{2},         \blue{2},\cdots,2,\red{3},\cdots,3,3,\cdots,3,{3},\cdots,{3},3,\cdots,3},
 {3,\cdots,3,\bro{3},\cdots,\bro{3}, \bro{3},\cdots,\bro{3},  \ora{3},\cdots,\ora{3},
  \blue{ \circled{\bf 3}},\cdots, \blue{\bf \circled{\bf 3}},\red{4},\cdots,4,6,\cdots,6,{6},\cdots,{6}},
 {4,\cdots,4,\bro{4},\cdots,\bro{4},\bro{4},\cdots,\bro{4},\ora{4},\cdots,\ora{4},        \blue{5},\cdots,5,\red{5},\cdots,5,7,\cdots,7},
 {5,\cdots,5,\bro{5 },\cdots,\bro{5},\bro{5},\cdots,\bro{5},\ora{\circled {\bf 6}},\cdots,\ora{\circled {\bf 6}},  \blue{6},\cdots,6,\red{6},\cdots,6},
 {6,\cdots,6,\bro{6},\cdots,\bro{6},\bro{6},\cdots,\bro{6},\ora{7},\cdots,\ora{7},          \blue{7},\cdots,{7},\red{7},\cdots,7},
 {7,\cdots,7,\bro{\bf \circled{\bf 7}},\cdots,\bro{\bf \circled{\bf 7}},     {8},\cdots,8},
 {8,\cdots,8},
},\quad
c>0
\nonumber
\\
 &m_{12{\bf 3}5678}+m_{12{\bf 3}567}\le m_2, \quad 0\le m_{1234{\bf 6}78}+m_{1234{\bf 6}7}\le m_{23},\quad 0\le m_{123456\bf 7}\le m_{236}\quad\\
  &m_{1234568}=m_{{\bf 2}34567}+ m_{12{\bf 3}567}+ m_{1234{\bf 6}7}
\end{align}
%{\oo missing column}

or\begin{align}\label{type45}
&(S^H)^{1,\dots,1,2,\dots,2,3,\dots,3,7^d,()^m}=\nonumber\\
&\YT{0.15in}{}{
 {1,\cdots,1,\bro{1},\cdots,\bro{1},   \bro{1},\cdots,\bro{1}, \bro{1},\cdots,\bro{1}, \ora{1},\cdots,\ora{1}, \blue{1},\cdots,1,
 \red{\bf \circled{\bf 2}},\cdots, \red{\bf \circled{\bf 2}},2,\cdots, 2, 2,\cdots, 2,2,\cdots,{2},2,\cdots,{2}},
 {2,\cdots,2,\bro{2},\cdots,\bro{2}, \bro{2},\cdots,\bro{2},   \bro{2},\cdots,\bro{2},   \ora{2}, \cdots,\ora{2},         \blue{2},\cdots,2,\red{3},\cdots,3,3,\cdots,3,{3},\cdots,{3},3,\cdots,3},
 {3,\cdots,3,\bro{3},\cdots,\bro{3}, \bro{3},\cdots,\bro{3},  \bro{3},\cdots,\bro{3},    \ora{3},\cdots,\ora{3},
  \blue{\bf \circled{\bf 3}},\cdots,\blue{\bf \circled{\bf 3}},\red{4},\cdots,4,6,\cdots,6,{6},\cdots,{6}},
 {4,\cdots,4,\bro{4},\cdots,\bro{4},\bro{4},\cdots,\bro{4},  \bro{4},\cdots,\bro{4},      \ora{4},\cdots,\ora{4},        \blue{5},\cdots,5,\red{5},\cdots,5,7,\cdots,7},
 {5,\cdots,5,\bro{5 },\cdots,\bro{5},\bro{5},\cdots,\bro{5}, \bro{5},\cdots,\bro{5},        \ora{\bf\circled {\bf 6}},\cdots,\ora{\bf\circled {\bf 6}},  \blue{6},\cdots,6,\red{6},\cdots,6},
 {6,\cdots,6,\bro{6},\cdots,\bro{6},\bro{6},\cdots,\bro{6},  \bro{6},\cdots,\bro{6},                  \ora{7},\cdots,           \ora{7},          \blue{7},\cdots,{7},\red{7},\cdots,7},
 {7,\cdots,7,\bro{ \bf\circled{\bf 7}},\cdots,\bro{ \bf\circled{\bf 7}},     {8},\cdots,8},
 {8,\cdots,8},
},\quad d>0
\nonumber\\
&m_{12{\bf 3}5678}+m_{12{\bf 3}567}\le m_2, \quad 0\le m_{1234{\bf 6}78}+m_{1234{\bf 6}7}\le m_{23},\quad 0\le m_{123456\bf 7}\le m_{236}\quad\\
  &m_{1234568}=m_{{\bf 2}34567}+ m_{12{\bf 3}567}+ m_{1234{\bf 6}7}
\end{align}

These tableaux satisfy inequalities  on the multiplicities % let $x$ be the multiplicity of
of the following columns:  $m_{2345678}$ is the multiplicity of $\redu^{-1}(2)=(2345678)$, $m_{1235678}$ is the multiplicity of $\redu^{-1}(3)=(1235678)$, $m_{1234678}$ the multiplicity of $\redu^{-1}(6)=(1234678)$,
$m_{2}$ the multiplicity of $(2)\in SpT_6(\varpi_1)$, $m_{23}$ the multiplicity of $(23)\in SpT_6(\varpi_2) $, $m_{236}$ the multiplicity of $(236)\in SpT_6(\varpi_3)$.
\end{cor}

\bibliography{sample17}

\begin{thebibliography}{KTW04}

\bibitem[ACM09]{acm09}
O.~Azenhas, A.~Conflitti, and R.~Mamede.
\newblock Linear time equivalent {L}ittlewood--{R}ichardson coefficient maps.
\newblock {\em Discrete Math. Theor. Comput. Sci. Proceedings, 21st
  International Conference on Formal Power Series and Algebraic Combinatorics
  (FPSAC 2009)}, pages 127--144, 2009.

\bibitem[ACM25]{azkoma25}
O.~Azenhas, A.~Conflitti, and R.~Mamede.
\newblock A uniform action of the dihedral group on {L}ittlewood--{R}ichardson
  coefficients.
\newblock {\em arXiv:2501.01947}, pages 1--59, 2025.

\bibitem[AKT16]{akt16}
O.~Azenhas, R.~C. King, and I.~Terada.
\newblock The involutive nature of the commutativity of
  {L}ittlewood-{R}ichardson coefficients.
\newblock {\em arXiv:1603.05037}, pages 1--109, 2016.

\bibitem[Ale26]{alexandersson}
P.~Alexandersson.
\newblock The symmetric functions catalog.
\newblock {\em https://www.symmetricfunctions.com/littlewoodRichardson.htm},
  4-4-2026.

\bibitem[Aze99]{az99}
O.~Azenhas.
\newblock The admissible interval for the invariant factors of a product of
  matrices.
\newblock {\em Linear Multilinear Algebra}, (46):51--99, 1999.

\bibitem[Aze25]{az18v5}
O.~Azenhas.
\newblock Skew {RSK} and the switching on ballot tableau pairs.
\newblock {\em arXiv:1808.06095v5}, pages 1--47, 2018, 2025.

\bibitem[Aze26a]{azreduction}
O.~Azenhas.
\newblock The inverse reduction map in the quantum {L}ittlewood-{R}ichardson
  bijection.
\newblock {\em arXiv:2606.24840v2}, pages 1--27, 2026.

\bibitem[Aze26b]{azslack}
O.~Azenhas.
\newblock The slack data of the recording tableaux in the quantum
  {L}ittlewood-{R}ichardson map determines its inverse: some applications.
\newblock {\em arXiv:2606.24840v2}, pages 1--18, 2026.

\bibitem[Aze26c]{az26}
O.~Azenhas.
\newblock The symplectic left companion of a
  {L}ittlewood-{R}ichardson-{S}undaram tableau and the {K}won property.
\newblock {\em arXiv:2601.06930v2}, page~20, 2026.

\bibitem[Buc00]{fultonbuch}
A.~S. Buch.
\newblock The saturation conjecture (after {A. Knutson} and {T. Tao}). {W}ith
  an appendix by {W}illiam {F}ulton.
\newblock {\em Enseign. Math.}, 46(1–2)(2):43–--60, 2000.

\bibitem[BZ88]{BZphy}
A.~D. Berenstein and A.~V. Zelevinsky.
\newblock Tensor product multiplicities and convex polytopes in partition
  space.
\newblock {\em J. Geom. Phys.}, 5:453--472, 1988.

\bibitem[BZ92]{bz}
A.~D. Berenstein and A.~V. Zelevinsky.
\newblock Triple multiplicities for {$sl(r + 1)$} and the spectrum of the
  exterior algebra of the adjoint representation.
\newblock {\em J. Algebraic Combin.}, (1):7--22, 1992.

\bibitem[CZ18]{lrvolume}
R.~Coquereaux and J.-B. Zuber.
\newblock From orbital measures to {L}ittlewood--{R}ichardson coefficients and
  hive polytopes.
\newblock {\em Ann. Inst. Henri Poincaré Comb. Phys. Interact.}, (5):339--386,
  2018.

\bibitem[DW02]{lrpolynomial}
H.~Derksen and J.~Weyman.
\newblock On the {L}ittlewood–{R}ichardson polynomials.
\newblock {\em J. Algebr}, (255 (2)):247--257, 2002.

\bibitem[Ehr62]{ehrhart}
E.~Ehrhart.
\newblock Sur les poly\`edres rationnels homoth\`etiques \`a {$n$} dimensions.
\newblock {\em C. R. Acad. Sci. Paris}, 24:A714(254):616--618, 1962.

\bibitem[Ful97]{fulton}
W.~Fulton.
\newblock {\em {Y}oung {T}ableaux with {A}pplications to {R}epresentation
  {T}heory and { G}eometry}, volume~35 of {\em Cambridge Univ. Press}.
\newblock London Math. Soc. Student Texts, 1997.

\bibitem[GT50]{gt50}
I.~M. Gelfand and M.~L. Tsetlin.
\newblock Finite-dimensional representations of the group of unimodular
  matrices.
\newblock {\em Dokl. Akad. Nauk SSSR (in Russian), English transl. in: I.M.
  Gelfand, Collected papers. Vol II, Berlin: Springer-Verlag}, (71):653--656,
  1950.

\bibitem[GZ85]{GZ85}
M.~Gelfand and A.~V. Zelevinsky.
\newblock Polyhedra in a space of diagrams and the canonical basis in
  irreducible representations of $gl_3$.
\newblock {\em Funct. Anal. Appl.}, 19:41--144, 1985.

\bibitem[GZ86]{gzpolyed}
I.~M. Gelfand and A.~V. Zelevinsky.
\newblock Multiplicities and proper bases for ${gl}_n $.
\newblock {\em Group Theoretical Methods in Physics, Proc. 3rd Yurmala Seminar
  1985, VNU Sci. Press, Utrecht}, II:147--159, 1986.

\bibitem[HK06a]{HK2}
A.~Henriques and J.~Kamnitzer.
\newblock Crystals and coboundary categories.
\newblock {\em Duke Math. J.}, 132(2):191--216, 2006.

\bibitem[HK06b]{HenKam}
A.~Henriques and J.~Kamnitzer.
\newblock The octahedron recurrence and $\mathfrak{gl}_n$-crystals.
\newblock {\em Adv. Math.}, 206(1):211--249, 2006.

\bibitem[Kin76]{king76}
R.~C. King.
\newblock Weight multiplicities for the classical groups.
\newblock {\em Group theoretical methods in physics, (Fourth Internat. Colloq.,
  Nijmegen, 1975), Lecture Notes in Phys., Springer, Berlin}, Vol.
  50:490–499, 1976.

\bibitem[KT99]{kt1}
A.~Knutson and T.~Tao.
\newblock The honeycomb model of ${GL}_n(\mathbb{C})$ tensor products {I}:
  Proof of the saturation conjecture.
\newblock {\em J. Amer. Math. Soc.}, (12):1055--1090, 1999.

\bibitem[KT25]{sathishtorres}
V.~S. Kumar and J.~Torres.
\newblock The branching models of {K}won and {S}udaram via flagged hives.
\newblock {\em Journal of Algebraic Combinatorics}, 62(5):1--14, 2025.

\bibitem[KTT06]{KTT1}
R.~C. King, C.~Tollu, and F.~Toumazet.
\newblock The hive model and the polynomial nature of stretched
  littlewood--richardson coefficients.
\newblock {\em S\'eminaire Lotharingien Comb.}, (54A), 2006.

\bibitem[KTT09]{KTT2}
R.~C. King, C.~Tollu, and F.~Toumazet.
\newblock Factorisation of littlewood--richardson coefficients.
\newblock {\em J. Combin. Theory A}, (116):314--333, 2009.

\bibitem[KTW04]{knutson}
A.~Knutson, T.~Tao, and C.~Woodward.
\newblock The honeycomb model of ${GL}_n(\mathbb{C})$ tensor products {II}:
  {P}uzzles determine facets of the {L}ittlewood--{R}ichardson cone.
\newblock {\em J. Amer. Math. Soc.}, 17:19--48, 2004.

\bibitem[Kwo09]{kwon09}
Jae-Hoon Kwon.
\newblock Cystal graphs and the combinatorics of {Y}oung tableaux in { M}.
  {H}azewinkel (ed.).
\newblock {\em Handbook of Algebra}, Vol 6:473--504, 2009.

\bibitem[Kwo18]{kwon18}
J.-H. Kwon.
\newblock {C}ombinatorial extension of stable branching rules for classical
  groups.
\newblock {\em Transactions of the American Mathematical Society},
  370(9):6125--6152, 2018.

\bibitem[Let99]{letzter}
G.~Letzter.
\newblock Symmetric pairs for quantized enveloping algebras.
\newblock {\em Journal of Algebra}, 220(2):729--767, 1999.

\bibitem[Lit44]{littlewood}
D.~E. Littlewood.
\newblock On invariant theory under restricted groups.
\newblock {\em Philos. Trans. Roy. Soc. London Ser. A, Math. Phys. Sci.},
  (239--809):387–417, 1944.

\bibitem[LL20]{leclen}
C.~Lecouvey and C.~Lenart.
\newblock Combinatorics of generalized exponents.
\newblock {\em International Mathematics Research Notices}, (16):4942--4992,
  2020.

\bibitem[LR34]{LR}
D.~E. Littlewood and A.~Richardson.
\newblock Group characters and algebra.
\newblock {\em Phil. Trans. Roy. Soc. London Ser. A}, (233):99--141, 1934.

\bibitem[Mol06]{molevq}
A.~I. Molev.
\newblock Representations of the twisted quantized enveloping algebra of type
  {$C_n$}.
\newblock {\em Mosc. Math. J.}, 6(3):531–--551, 2006.

\bibitem[Mun25]{muniz}
B.~Muniz.
\newblock Symplectic branching through crystals.
\newblock {\em arXiv:2505.21738}, 2025.

\bibitem[NSW26]{nsw}
S.~Naito, Y.~Suzuki, and H.~Watanabe.
\newblock A proof of the { Naito}–{Sagaki} conjecture via the branching rule
  for {$\imath $}quantum groups.
\newblock {\em Journal of Algebra}, 691:32--87, 2026.

\bibitem[Ras04]{rassartlrpolynomial}
E.~Rassart.
\newblock A polynomiality property for {L}ittlewood–-{R}ichardson
  coefficients.
\newblock {\em J. Comb. Theory Ser. A}, (107 (2)):161--179, 2004.

\bibitem[ST18]{schumanntorres}
B.~Schumann and J.~Torres.
\newblock A non-{L}evi branching rule in terms of {L}ittelmann paths.
\newblock {\em Proc. Lond. Math. Soc.}, 117(5):1077--1100, 2018.

\bibitem[Sta98]{stanley}
R.~P. Stanley.
\newblock {\em Enumerative Combinatorics}, volume~62 of {\em Cambridge Studies
  in Advanced Mathematics}.
\newblock Cambridge University Press, Cambridge, 1998.

\bibitem[Sun86]{sundaram}
Sheila Sundaram.
\newblock {\em On the combinatorics of representations of ${S}p (2n,
  \mathbb{C})$. {P}hD thesis}.
\newblock Massachusetts Institute of Technology. 1986.

\bibitem[Sun90]{sundaram90}
S.~Sundaram.
\newblock Tableaux in the representation theory of the classical {L}ie groups.
\newblock {\em Invariant theory and tableaux (Minneapolis, MN, 1988) IMA Vol.
  Math. Appl., 19, Springer, New York,}, pages 191--225, 1990.

\bibitem[Tho78]{tho78}
G.~P. Thomas.
\newblock On {S}chensted's construction and the multiplication of {S}chur
  functions.
\newblock {\em Adv. Math.}, 30:8--32, 1978.

\bibitem[Wan23]{wang}
W.~Wang.
\newblock {\em Quantum Symmetric pairs in Proc. Int. Cong. Math. 2022},
  volume~4.
\newblock EMS Press, Berlin, 2023.

\bibitem[Wat25a]{watanabe25}
H.~Watanabe.
\newblock Berele row-insertion and quantum symmetric pairs.
\newblock {\em arXiv:2509.00853v2}, page~32, 2025.

\bibitem[Wat25b]{watanabe}
H.~Watanabe.
\newblock Symplectic tableaux and quantum symmetric pairs.
\newblock {\em J. Comb. Algebra DOI 10.4171/JCA/113}, page~57, 2025.

\end{thebibliography}
\bibliographystyle{alpha}

\end{document}